\documentclass[12pt]{amsart}

\usepackage{mathrsfs}
\usepackage{amssymb}
\usepackage{cite}
\usepackage[all]{xy}
\usepackage{graphicx}
\usepackage{textcomp}
\usepackage{charter}
\usepackage[vflt]{floatflt}
\usepackage{mathtools}
\usepackage{enumerate}
\usepackage{wasysym}
\usepackage{rotating}
\usepackage{pdflscape}
\usepackage{fullpage}

\usepackage[pdftex,colorlinks]{hyperref}
\hypersetup{
pdfauthor={Shamovich, E.},
bookmarksnumbered,
pdfstartview={FitH}}%

\newtheorem{thm}{Theorem}[section]
\newtheorem{prop}[thm]{Proposition}
\newtheorem{lem}[thm]{Lemma}
\newtheorem{cor}[thm]{Corollary}
\newtheorem*{thm*}{Theorem}
\newtheorem*{cor*}{Corollary}
\newtheorem{quest}[thm]{Question}

\theoremstyle{definition}
\newtheorem{dfn}[thm]{Definition}

\newtheorem{example}[thm]{Example}
\newtheorem{rem}[thm]{Remark}

\theoremstyle{remark}

\newtheorem{remark}[thm]{Remark}

\numberwithin{equation}{section}

\usepackage{mathrsfs}

\newcommand{\cA}{{\mathcal A}}
\newcommand{\cB}{{\mathcal B}}

\newcommand{\cE}{{\mathcal E}}
\newcommand{\cF}{{\mathcal F}}
\newcommand{\cG}{{\mathcal G}}
\newcommand{\cH}{{\mathcal H}}
\newcommand{\cI}{{\mathcal I}}
\newcommand{\cJ}{{\mathcal J}}

\newcommand{\cL}{{\mathcal L}}
\newcommand{\cM}{{\mathcal M}}
\newcommand{\cN}{{\mathcal N}}
\newcommand{\cO}{{\mathcal O}}

\newcommand{\cT}{{\mathcal T}}
\newcommand{\cU}{{\mathcal U}}
\newcommand{\cV}{{\mathcal V}}

\newcommand{\cX}{{\mathcal X}}
\newcommand{\cY}{{\mathcal Y}}

\newcommand{\fA}{{\mathfrak A}}
\newcommand{\fB}{{\mathfrak B}}
\newcommand{\fC}{{\mathfrak C}}

\newcommand{\fH}{{\mathfrak H}}

\newcommand{\fS}{{\mathfrak S}}

\newcommand{\fV}{{\mathfrak V}}
\newcommand{\fW}{{\mathfrak W}}

\newcommand{\fm}{{\mathfrak m}}

\newcommand{\C}{{\mathbb C}}

\newcommand{\N}{{\mathbb N}}

\newcommand{\M}{{\mathbb M}}
\newcommand{\F}{{\mathbb F}}
\newcommand{\W}{{\mathbb W}}
\newcommand{\D}{{\mathbb D}}

\newcommand{\bB}{\mathbb{B}}
\newcommand{\bC}{\mathbb{C}}
\newcommand{\bD}{\mathbb{D}}
\newcommand{\bH}{\mathbb{H}}
\newcommand{\bM}{\mathbb{M}}
\newcommand{\bN}{\mathbb{N}}

\newcommand{\id}{{\bf id}}

\newcommand{\mlt}{\operatorname{Mult}}
\newcommand{\alg}{\operatorname{alg}}

\newcommand{\dist}{\operatorname{dist}}
\newcommand{\spn}{\operatorname{span}}
\newcommand{\Rep}{\operatorname{Rep}}
\newcommand{\mspn}{\operatorname{mat-span}}

\newcommand{\ol}{\overline}

\def\mcc{M\raise.5ex\hbox{c}C}
\def\mccarthy{M\raise.5ex\hbox{c}Carthy}

\begin{document}

\title[Nc analytic functions on nc varieties]{Algebras of bounded noncommutative analytic functions on subvarieties of the noncommutative unit ball}

\author{Guy Salomon}
\address{Department of  Mathematics\\
Technion --- Israel Institute of Technology\\
Haifa, 3200003, Israel}
\email{guy.salomon@technion.ac.il}
\author{Orr M. Shalit}
\email{oshalit@technion.ac.il}
\author{Eli Shamovich}
\email{eshamovich@uwaterloo.ca}


\thanks{The first author was partially supported by the Clore Foundation.
The second author was partially supported by Israel Science Foundation Grants no. 474/12 and 195/16, and by
EU FP7/2007-2013 Grant no. 321749}

\begin{abstract}
We study algebras of bounded, noncommutative (nc) analytic functions on nc subvarieties of the nc unit ball.
Given an nc variety $\fV$ in the nc unit ball $\fB_d$, we identify the algebra of bounded analytic functions on $\fV$ --- denoted $H^\infty(\fV)$ --- as the multiplier algebra $\mlt \cH_\fV$ of a certain reproducing kernel Hilbert space $\cH_\fV$ consisting of nc functions on $\fV$.
We find that every such algebra $H^\infty(\fV)$ is completely isometrically isomorphic to the quotient $H^\infty(\fB_d)/ \cJ_\fV$ of the algebra of bounded nc holomorphic functions on the ball by the ideal $\cJ_\fV$ of bounded nc holomorphic functions which vanish on $\fV$.
In order to demonstrate this isomorphism, we prove that the space $\cH_\fV$ is an nc complete Pick space (a fact recently proved --- by other methods --- by Ball, Marx and Vinnikov).

We investigate the problem of when two algebras $H^\infty(\fV)$ and $H^\infty(\fW)$ are (completely) isometrically isomorphic.
If the variety $\fW$ is the image of $\fV$ under an nc analytic automorphism of $\fB_d$, then $H^\infty(\fV)$ and $H^\infty(\fW)$ are completely isometrically isomorphic.
We prove that the converse holds in the case where the varieties are homogeneous; in general we can only show that if the algebras are completely isometrically isomorphic, then there must be nc holomorphic maps between the varieties (in the case $d = \infty$ we need to assume that the isomorphism is also weak-$*$ continuous).

We also consider similar problems regarding the bounded analytic functions that extend continuously to the boundary of $\fB_d$ and related norm closed algebras; the results in the norm closed setting are somewhat simpler and work for the case $d = \infty$ without further assumptions. 

Along the way, we are led to consider some interesting problems on function theory in the nc unit ball.
For example, we study various versions of the Nullstellensatz (that is, the problem of to what extent an ideal is determined by its zero set), and we obtain perfect Nullstellensatz in both the homogeneous as well as the commutative cases.
\end{abstract}

\maketitle
\tableofcontents

\section{Introduction}
\subsection{Historical Background and Motivation}
The study of bounded analytic functions on open domains in $\C^n$ is well-entrenched.
In particular, the algebra $H^{\infty}(\D)$ of bounded analytic functions on the disc was extensively studied by many, starting from Hardy \cite{Har15} and Riesz \cite{Rie23}; see also the excellent books \cite{Gar07}, \cite{Rudin} and \cite{Rud69}.
One area that stood out, in particular, due to its applications, is the theory of interpolation of bounded analytic functions on the disc initiated by Pick \cite{Pick16} and Nevanlinna  \cite{Nev19,Nev29}.
These concepts were later given a fresh approach from the operator theoretic perspective by Sarason in \cite{Sar85} and by others (see \cite{AM_Book}).
In this approach, one regards $H^{\infty}(\D)$ as an algebra of operators on the Hilbert space $H^2(\D)$, the space of analytic functions on the disc with square summable Taylor coefficients at the origin.


Another connection of the classical theory of bounded analytic functions on the disc to operator theory was discovered by von Neumann in \cite{vNeu51}, where he proved his celebrated inequality.
The inequality of von Neumann was extended to the bidisc by Ando in \cite{An63}.
After many unsuccessful generalization attempts, it was Parrot \cite{Par70} who showed that the von Neumann inequality fails for the tridisc.
However, in the case of commuting row contractions, Drury observed in \cite{Drury} that one can obtain a von Neumann inequality if we replace the bounded analytic functions on the unit ball with the algebra of multipliers of the Drury--Arveson space (also known as the symmetric Fock space), where the norm of the algebra is the multiplier norm instead of the supremum norm (see also \cite{MV93} and \cite{Arv98}).
The Drury--Arveson space is, in fact, a reproducing kernel Hilbert space of analytic functions on the unit ball and has the complete Pick property, i.e., the interpolation question for matrix valued functions has a satisfactory answer (see \cite{ArPop95,Popescu95b} and \cite[Section 8.9]{AM_Book} and the references therein).
Moreover, Agler and {{\mccarthy}} proved in \cite{AM00} that this space is universal among the spaces with the complete Pick property.

Let $d \in \bN \cup \{\infty\}$, and let $H^2_d$ denote the Drury--Arveson space, and let $\cM_d$ denote the multiplier algebra of the Drury--Arveson space (see the survey \cite{ShalitSurvey}).
We note that $\cM_d$ is closed in the weak-operator topology ($\textsc{wot}$) on $B(H^2_d)$ and, furthermore, it is obtained as the $\textsc{wot}$-closure of the algebra generated by the shifts $M_{z_i} : f \mapsto z_i f$ ($i=1,\ldots, d$).
Every analytic subvariety of the ball cut out by functions in $H^2_d$ can be cut out by functions in $\cM_d$.
Following \cite{DRS15}, with every such subvariety $V \subset \bB_d$ we associate the subspace $\cH_V \subset H^2_d$, which is the closure of the subspace spanned by the kernel functions corresponding to points of $V$.
This is a reproducing kernel Hilbert space, and we let $\cM_V$ denote the multiplier algebra of $\cH_V$. Using the complete Pick property, one obtains that there is a completely contractive and surjective  map $\cM_d \to \cM_V$ and its kernel is a $\textsc{wot}$-closed ideal.
This consideration tells us that $\cM_d$ enjoys a property similar to the property of Stein manifolds and affine schemes, namely that every ``function'' on a subvariety lifts to a global ``function''.
Therefore, it stands to reason to ask, to what extent do the variety and the algebra of multipliers on it determine each other.

This question was answered in the case of complete isometric isomorphism by Davidson, Ramsey and Shalit in \cite[Theorem 4.5]{DRS15}.
They proved that $\cM_V \cong \cM_W$ completely isometrically if and only if there exists an automorphism of the ball $F$, such that $F(V) = W$.
In particular, the multiplier algebras ``see'' not only the analytic structure of the variety, but also give information about the embedding of the variety into the ball.
In the case of homogeneous subvarieties of $\bB_d$ the result is much stronger.
In fact Davidson, Ramsey and Shalit in \cite{DRS11} and Hartz in \cite{Hartz12} proved that if $d < \infty$, then  for homogeneous varieties $V$ and $W$ we have that $\cM_V \cong \cM_W$ algebraically if and only if there exists a linear map $\varphi \in GL_d(\C)$, such that $\varphi(V) = W$.
We refer the reader to \cite{SalomonShalit} for a detailed survey and also additional results on these questions.

The theory described above gives a satisfactory answer to the classification of quotient algebras of the form $\cM_d / \cJ_V$, where $\cJ_V$ is the kernel of the restriction map $f \mapsto f\big|_V$. 
The limitation, however, is that it deals with radical ideals only.
One way to see higher order vanishing of a function of one variable at a point $\lambda$ is to consider $f(\begin{smallmatrix} \lambda & 1 \\ 0 & \lambda \end{smallmatrix})$.
This consideration among others leads us to consider the noncommutative (also called ``free") setting.


Noncommutative (nc for short) functions are functions defined on subsets of matrices of all sizes which respect direct sums and similarities (see Section \ref{sec:prelim} for precise definitions).
Such functions were first introduced by Taylor in \cite{Tay72frame,Tay73} and also by Voiculescu in \cite{Voic85,Voic86,Voic04,Voic10}.
Recently, many works laid out the foundations of noncommutative free analysis, such as \cite{KVBook}, \cite{AM15,AM15a,AM15c} and \cite{Popescu06b}.
The field of noncommutative analysis has enjoyed such rapid growth, due to applications in many fields such as free probability \cite{PopVin13,BPV13} and real and convex algebraic geometry \cite{Hel02,HM2012,HKM13,HKM13-relax,HKMS15}.
Noncommutative functions appeared earlier in the realm of operator algebras in the works of Bunce \cite{Bunce}, Frazho \cite{Frazho} and Popescu \cite{Popescu91}, that generalized von Neumann's inequality to an arbitrary (non commuting) row contraction; here the shift $M_{z_1}, \ldots, M_{z_d}$ on Drury--Arveson space is replaced by the left creation operators $L_1, \ldots, L_d$ on the full Fock space (we shall explain below how to interpret $L_i$ as the noncommutative coordinate function on the nc unit ball).
The $\textsc{wot}$-closed algebra generated by $L_1, \ldots, L_d$ was studied in detail by Arias and Popescu \cite{AriasPopescu}, Davidson and Pitts \cite{DavPitts1,DavPitts2,DavPittsPick}, Muhly and Solel \cite{MS04,MS11,MS13} and Popescu \cite{Popescu95,Popescu06b, Popescu10}.

Analogues of the Nevanlinna-Pick interpolation on the noncommutative ball first appeared in \cite{DavPittsPick} and \cite{AriasPopescu}.
More general noncommutative versions of the classical interpolation and realization results appeared recently in the works of Agler and {\mccarthy} \cite{AM15d} and Ball, Marx and Vinnikov \cite{BMV15a,BMV15b}, who also introduced a generalization of reproducing kernel Hilbert spaces to the free setting.

Our first goal in this work is to show that the full Fock space is a noncommutative reproducing kernel Hilbert space (nc RKHS), and its algebra of multipliers is, on the one hand, the algebra of bounded functions on the noncommutative ball (such that the multiplier norm and the supremum norm coincide), and, on the other hand, that this algebra coincides with the $\textsc{wot}$ closed algebra considered by Arias--Popescu and Davidson--Pitts.
With this identification in hand, our second goal is to show that several results of \cite{BMV15b} in the case of the noncommutative ball follow from established operator algebraic techniques and results, in particular, the complete Pick property of the noncommutative kernel of the full Fock space.

We then proceed to study subvarieties cut out by multipliers in the noncommutative ball. 
Let $H^{\infty}(\fB_d)$ denote the algebra of bounded nc functions on the noncommutative ball, and let $\cH^2_d$ denote the full Fock space.
Then, as above, for every subvariety $\fV \subset \fB_d$ we associate an nc RKHS $\cH^2_\fV \subset \cH^2_d$ and its algebra of multipliers $\mlt \cH^2_\fV$. 
We prove that $\mlt \cH^2_\fV = H^{\infty}(\fV)$ and show that $H^{\infty}(\fV)$ is a quotient of $H^{\infty}(\fB_d)$ by a $\textsc{wot}$-closed ideal, as in the commutative case.
Thus, we are led to study the isomorphism problem for such algebras.
We obtain the following (partial) generalization of the result of \cite{DRS15}.
\begin{thm*}[Theorem \ref{thm:isomorphism}]
Let $\fV \subseteq \fB_d$ and $\fW \subseteq \fB_{e}$ be nc varieties, and let $\alpha : H^\infty(\fV) \to H^\infty(\fW)$ be a completely isometric isomorphism. Assume that $d$ and $e$ are finite or that $\alpha$ is weak-$*$ continuous.
Then there exists an nc map $G: \fB_e \to \fB_d$ such that $G \big|_{\fW}  = G_\alpha$ maps $\fW$ bijectively onto $\fV$, which implements $\alpha$ by the formula
\[
\alpha(f) = f \circ G \quad, \quad f \in H^\infty(\fV).
\]
\end{thm*}

More satisfactory results are obtained in the homogeneous case. First, we show that in the homogeneous case we have a strong form of the Nullstellensatz, that does not require taking radicals.
\begin{thm*}[Theorem \ref{thm:null_poly}]
Let $d \in \bN$ and let $J \triangleleft \F_d$ be a homogeneous ideal.
Then
\[I(V_{\fB_d}(J)) = J. \]
\end{thm*}
Then we show that --- as in the commutative case --- completely isometric isomorphisms are implemented by automorphisms of the ball.
\begin{thm*}[Theorem \ref{thm:isomorphism_homo} and Corollary \ref{cor:equivalence_is_linear_homog}]
Let $\fV \subseteq \fB_d$ and $\fW \subseteq \fB_{e}$ be homogeneous nc varieties, and let $\alpha : H^\infty(\fV) \to H^\infty(\fW)$ be a completely isometric isomorphism. Assume that $d$ and $e$ are finite or that $\alpha$ is weak-$*$ continuous.
Then $\fV$ and $\fW$ are conformally equivalent, in the sense that one may assume that there is some $k$ such that $\fV, \fW \subseteq \fB_k$, and that under this assumption there exists an automorphism $G \in \operatorname{Aut}(\fB_k)$ such that $G(\fW)  = \fV$, and such that
\[
\alpha(f) = f \circ G \quad, \quad f \in H^\infty(\fV).
\]
Furthermore, in this case there exists a unitary mapping $\fV$ onto $\fW$.
\end{thm*}
In Theorem \ref{thm:iso_homo_cont_older} we obtain a closely related result: if $\fV, \fW \subseteq \fB_d$ ($d<\infty$) are homogeneous varieties, then $H^\infty(\fV)$ is isometrically isomorphic to $H^\infty(\fW)$ if and only if there exists a unitary mapping $\fV$ onto $\fW$.

We then proceed to discuss the norm closure of the free algebra in the supremum norm on the noncommutative ball; this should be considered as the nc analogue of the disc algebra $A(\D)$.
We discuss the conditions for a bounded noncommutative function on a homogeneous variety to be in the norm closure of the polynomial functions on the variety, and discuss the isomorphism problem for the norm closed algebras instead of the $\textsc{wot}$-closed algebras; the classification scheme turns out to be the same. 

Lastly, we discuss connections to subproduct systems and to the commutative case.
In \cite{EisHoch79} Eisenbud and Hochester proved what they called a version of Nullstellensatz with nilpotents. We provide a different proof to this perfect ``free commutative Nullstellensatz", which shows how nc varieties encode the higher order zeros of commuting polynomials.
This version of the Nullstellensatz is ``perfect", in the sense that an appropriately defined zero locus of an ideal captures all the information about that ideal, in a way that does not involve radicals.

To briefly explain the result, let $\fC\bM_d$ denote the disjoint union $\sqcup_n CM_n^d$, consisting of all commuting $d$-tuples of $n \times n$ matrices, where $n$ varies through $\bN$.
Let $\bC[z] = \bC[z_1, \ldots, z_d]$ be the algebra of polynomials in $d$ (commuting) variables.
Given $\Omega \subset \fC\bM_d$ and $S \subseteq \bC[z]$, we let
\[
I_{\bC[z]}(\Omega) = \{p \in \bC[z] :  p(X) = 0 \,\, \textrm{ for all } \,\, X \in \Omega \}
\]
and
\[
V_{\fC\bM_d}(S) = \{X \in \Omega : p(X) = 0 \,\, \textrm{ for all } \,\, p \in S\}.
\]
Then our commutative free Nullstellensatz reads as follows.
\begin{cor*}[Commutative free Nullstellensatz --- Corollary \ref{cor:free_com_NSTZ}]
For every ideal $J \triangleleft \bC[z]$,
\[
I_{\bC[z]}(V_{\fC\bM_d}(J)) = J.
\]
\end{cor*}

\subsection{Readers' Guide}

This subsection contains a more detailed outline of the structure of the paper for the convenience of the reader.

Section \ref{sec:prelim} contains the preliminaries and notations. We deal with the notion of noncommutative functions and sets.
Then we proceed to discuss noncommutative completely positive kernels, nc RKHS and multipliers following \cite{BMV15a}.

Section \ref{sec:szego_and_bounded_functions} contains the comparison of objects that we intend to study with objects that are already established in the literature (in particular in the works of Arias and Popescu, Davidson and Pitts, and Popescu).
We demonstrate that the full Fock space is isomorphic to an nc RKHS associated to the noncommutative Szego kernel.
We then show that this map induces unitary equivalence between the algebra of multipliers and the algebra of bounded analytic nc functions, and we conclude that it is also unitarily equivalent to the $\textsc{wot}$-closure of the free algebra generated by the left creation operators on the full Fock space.
In Section \ref{sec:complete_pick} we proceed to show, using operator algebraic methods, that the Szego kernel has the complete Pick property (this result was first obtained in \cite{BMV15b}, by different methods).

Section \ref{sec:quotients} begins our discussion of subvarieties of the noncommutative ball.
We show that we can associate to every subvariety a reproducing kernel Hilbert space and thus an algebra of multipliers. Every such algebra is completely isometrically isomorphic to a quotient of the algebra of multipliers of the noncommutative Szego kernel.
We also show that every multiplier is, in fact, a bounded nc function on the variety and the multiplier norm coincides with the supremum norm.
We then proceed, in Section \ref{sec:isomorphisms}, to discuss the isomorphism problem for subvarieties of the ball.
This section contains one of the main results of the paper, namely that multiplier algebras of two subvarieties of the noncommutative ball are completely isometrically isomorphic if and only if the varieties are biholomorphic (and, moreover, that a biholomorphism implements the isomorphism).

Sections \ref{sec:homog_case} and \ref{sec:homog_isom} contain the discussion of the homogeneous case.
In the homogeneous setting, the extra structure afforded by the action of the multiplicative group $\C^{\times}$ allows us to show that if the multiplier algebras of two homogeneous varieties are completely isometrically isomorphic, then the varieties are mapped one onto another by a linear automorphism of the commutative ball. In addition, we show that a free homogeneous Nullstellensatz holds, i.e, that the ideal of functions vanishing on the variety cut out by a $\textsc{wot}$-closed homogeneous ideal $J$ is $J$ itself.

Section \ref{sec:continuous} contains the discussion of the algebra that is the norm closure of the algebra generated by the left creation operators.
This algebra is the multivariable noncommutative analogue of the disc algebra.
We discuss conditions that allow us to approximate a multiplier by polynomials in norm, and we consider examples.
We then proceed to discuss the norm closed version of the isomorphism problem.

Section \ref{sec:subproduct} connects our paper with the study of subproduct systems initiated by Shalit and Solel in \cite{ShalitSolel}. 
We show that the study of subproduct systems in the case when $d < \infty$ is equivalent to the study of homogeneous ideals of multipliers. 
We explain how the results of this paper add to what is known, and also contribute to the longstanding problem of classifying tensor algebras of subproduct systems in terms of the subproduct systems (see Proposition \ref{prop:iso_homo_dinfty}). 

Lastly, in Section \ref{sec:connections_to_comm} we discuss the connection of this work to the commutative case; in particular, we explain the connection to the isomorphism problem for complete Pick algebras \cite{DHS14,DRS11,DRS15,Hartz12,Hartz16,KerMcSh13,RamseyThesis,SalomonShalit}.
We prove an algebraic version of the matricial Nullstellensatz (the ``commutative free Nullstellensatz") and show some obstructions to such a Nullstellensatz in the case of the algebras of bounded functions on the noncommutative ball.

\section{Preliminaries}\label{sec:prelim}

\subsection{Noncommutative sets and noncommutative functions}

We consider noncommutative (nc) function theory in $d$ complex variables, where $d \in \bN$ or $d = \infty$. 
Let $M_n = M_n(\bC)$ denote the set of all $n \times n$ matrices over $\bC$, and let $M_n^d$ be the set of all $d$-tuples $X = (X_1, X_2, \ldots)$ of such matrices such that the row $X$ determines a bounded operator from $\bC^n \oplus \bC^n \oplus \ldots$ to $\bC^n$ (of course, this specification matters only when $d = \infty$).
We norm $M_n^d$ with the row operator norm (that is, $\|X\| = \|\sum_j X_j  X_j^*\|^{1/2}$), and endow $M_n^d$ with the induced topology.

Let
\[
\bM_d = \sqcup_{n=1}^{\infty} M_n^d.
\]
A set $\Omega \subset \bM_d$ is said to be a {\em nc set} if it is closed under direct sums.
If $\Omega$ is an nc set, we denote $\Omega_n = \Omega \cap M_n^d$.
We also use the notation $\Omega(n) = \Omega \cap M_n^d$.

Let $\cV$ be a vector space.
A function $f$ from an nc set $\Omega \subseteq \bM_d$ to $\sqcup_{n=1}^\infty M_n(\cV)$ is said to be a {\em nc function (with values in $\cV$)} if
\begin{enumerate} 
\item $f$ is graded: $X \in \Omega_n \Rightarrow f(X) \in M_n(\cV)$,
\item $f$ respects direct sums: $f(X \oplus Y) = f(X) \oplus f(Y)$,
\item $f$ respects similarities: if $X \in \Omega_n$ and $S \in M_n$ is invertible, and if $S^{-1} X S \in \Omega_n$, then $f(S^{-1} X S) = S^{-1} f(X) S$.
\end{enumerate}
An nc function with values in $\bC$ is said to be a {\em scalar valued} nc function.

We will be mostly interested in scalar valued nc functions, but we shall also require the cases where $\cV = \cE$ or $\cV = B(\cE)$, where $\cE$ is a Hilbert space.
In the case $\cV = \cE$, we identify $\cE$ with $B(\C,\cE)$ (bounded operators from $\C$ into $\cE$) and we identify $M_n(B(\C,\cE))$ with $B(\C^n, \cE^n)$.

A {\em free polynomial} is an element in $\F_d := \bC\langle z_1, \ldots, z_d\rangle$ (the free algebra in $d$ variables).
Every free polynomial is a (scalar valued) nc function.
Let $\W_d$ be the free monoid on $d$ generators.
A polynomial $p(z) = \sum_{k \in \W_d} a_k z^k$ can be written in a unique way as
$p(z) = \sum_{n \in \bN} p_n(z)$
where
\[p_n(z) = \sum_{k\in \W_d, |k|=n} a_k z^k .
\]
The polynomial $p_n$ is called {\em the homogeneous component of degree $n$} of $p$.

The  {\em ($d$-dimensional) open matrix unit ball} $\fB_d$ is defined to be
\[
\fB_d = \left\{ X \in \bM_d : \|X\|^2 = \left\|\sum X_j X_j^*\right\| < 1\right\}.
\]

A subset $\Omega \subseteq \bM_d$ is said to be {\em open/closed} if for all $n$, $\Omega_n$ is open/closed. 
An nc set $\Omega \subseteq \M_d$ will be said to be a {\em nc domain} if it is open and if every $\Omega_n$ is connected. 
The topology determined by this collection of open sets is sometimes called the {\em disjoint union (du) topology}.
The {\em boundary of $\Omega$}, denoted $\partial \Omega$, is defined to be $\sqcup_{n=1}^\infty \partial \Omega_n$.

A function $f$ defined on an nc open set $\Omega$ is said to be {\em nc holomorphic} if it is an nc function and, in addition, it is locally bounded.

By {\em locally bounded} we mean that for every $X \in \Omega$, there is a set $U \ni X$, which is open in the du topology, such that $f$ is bounded on $U$.
(There are other topologies one may consider on $\bM_d$, which lead to different notions of local boundedness and hence to different notions of holomorphy.
Since we are mainly interested with bounded nc functions on $\fB_d$ ($d \in \bN \cup \{ \infty\}$), which is open in all topologies of interest, local boundedness will not be an issue.)
It turns out that an nc holomorphic function is really a holomorphic function when considered as a function $f : \Omega_n \to M_n$, for all $n$, and moreover it has a ``Taylor series'' at every point (see \cite{KVBook}).

A {\em noncommutative (nc) algebraic variety} is a set of the form
\[
V_\Omega(S) = \{X \in \Omega :  p(X) = 0 \,\, \textrm{ for all } p \in S\},
\]
where $S \subseteq \F_d$.
Likewise, we define a {\em nc holomorphic variety in $\Omega$} to be the joint zero set of a set of scalar valued nc holomorphic functions on $\Omega$.
There are potentially more general definitions that may be worth considering, but this will be our working definition. 
Note that nc algebraic and nc holomorphic varieties are nc sets.

If $\Omega$ is an open nc set, $\fV \subset \Omega$ is an nc variety, and $f : \fV \to \bM_1$ is a function, we say that $f$ is a {\em nc function} if it satisfies the nc function conditions, and we say that it is {\em nc holomorphic} if, in addition, for every $X \in \fV$, there exists an open neighborhood $X \in U \subset \Omega$, such that $f$ extends to a bounded nc function on $U$.

\begin{remark}
We will see in Theorems \ref{thm:quotient_mult} and \ref{thm:mult_are_bounded_on_V}, that if $\Omega = \fB_d$, and if
$f : \fV \to \bM_1$ is an nc function that is bounded on $\fV$, then there exists a bounded nc holomorphic function $F$ on $\fB_d$ such that $f = F\big|_\fV$ (this also follows from results in \cite{BMV15b}).
\end{remark}

Given an open nc set $\Omega \subseteq \bM_d$, we define $H^\infty(\Omega)$ to be the algebra of bounded holomorphic functions on $\Omega$, and $A(\Omega)$ to be the algebra of bounded holomorphic functions that extend to uniformly continuous functions on  $\overline{\Omega} = \Omega \cup \partial \Omega$ (see Section \ref{sec:continuous} for more details).
We give $H^\infty(\Omega)$ and $A(\Omega)$ the obvious operator algebra structure, where the matrix norm of $F \in M_n(H^\infty(\fB_d))$ is given by
\[
\|F\| = \sup_{z \in \Omega} \|F(z)\| = \sup_{k \in \bN} \sup_{z \in \Omega(k)} \|F(z)\|_{M_{nk}}.
\]
Similarly, we define for a variety $\fV \subseteq \Omega$ the algebra $H^\infty(\fV)$ of bounded nc functions on $\fV$, and the algebra $A(\fV)$ of bounded nc functions on $\fV$ that continue to uniformly continuous functions on $\overline{\fV}$.

\subsection{Noncommutative reproducing kernel Hilbert spaces and multipliers}

In what follows, we let $B(X,Y)$ denote the space of all bounded linear maps between two normed spaces $X$ and $Y$.

Let $\Omega \subseteq \bM_d$, and let $\cA$ and $\cB$ be C*-algebras.
A {\em completely positive (cp) nc kernel} ({\em with values in $B(\cA, \cB)$}) on $\Omega$ is a function
\[
k : \Omega \times \Omega \to \sqcup_{m,n=1}^{\infty} B(M_{m\times n}(\cA), M_{m\times n}(\cB))
\]
such that
\begin{enumerate}
\item $k$ is graded, in the sense that
\[
Z \in \Omega_m, W \in \Omega_n \Rightarrow k(Z,W) \in B(M_{m\times n}(\cA), M_{m\times n}(\cB)) .
\]
\item $k$ respects intertwining, in the sense that
\[
Ak(Z,W)(P)B^* = k(Z',W')(APB^*)
\]
whenever $AZ = Z'A$ and $BW = W'B$, for all appropriately sized matrices with $Z,Z',W,W' \in \Omega$.
\item $k(Z,Z)$ is a cp map for all $Z \in \Omega$.
\end{enumerate}
This definition is a special case of the definition introduced and used in \cite{BMV15a,BMV15b}.

In this paper, we will be interested in the particular case where $\cA = \bC$ and $\cB = B(\cE)$, with $\cE$ a Hilbert space (we will say then that the kernel has values in $B(\cE) \cong B(\C,B(\cE))$). 
With these specifications, a main result of \cite{BMV15a} (Theorem 3.1 there) is that every cp nc kernel $k$ (with values in $B(\cE)$) is the ``reproducing kernel'' of an nc reproducing kernel Hilbert space (RKHS) --- a Hilbert space $\cH$ consisting of nc functions with values in $\cE \cong B(\C,\cE)$, in which ``point evaluation'' is bounded, and which is generated by the set of nc functions
\[
\left\{k_{W,v,y} : \Omega \to \sqcup_{n=1}^\infty M_n(B(\mathbb C, \mathcal E)) : W \in \Omega_n,  y \in \cE^n, v \in \C^n,  n \in \bN \right\}
\]
given by the formula
\[
k_{W,v,y}(Z) u = k(Z,W)(uv^*) y , \quad Z \in \Omega_m, u \in \C^m.
\]
Moreover, the kernel functions $k_{W,v,y}$ have the reproducing property that for every $h \in \cH$,
\begin{equation}\label{eq:rep_prop}
\langle h, k_{W,v,y} \rangle = \langle h(W)v, y \rangle.
\end{equation}

The {\em multiplier algebra} of an nc RKHS $\cH$ is the algebra of nc functions
\[
\mlt\cH = \{f : \Omega \to \sqcup_{n=1}^\infty M_n(B(\mathcal E))  :  fh \in \cH \,\, \textrm{ for all } h \in \cH\}.
\]
Every multiplier $f \in \mlt \cH$ determines a bounded multiplication operator $M_f : \cH \to \cH$  given by $M_f h = f h$ \cite{BMV15a}.

\begin{lem}\label{lem:adjoint}
Let $k$ be a cp nc kernel on $\Omega$ with values in $B(\cE)$.
If $f \in \mlt\cH$, then for every kernel function $k_{W,v,y}$,
\[
M_f^* k_{W,v,y} = k_{W,v,f(W)^*y}.
\]
\end{lem}
\begin{proof}
For every $h \in \cH$,
\begin{align*}
\langle h, M_f^* k_{W,v,y} \rangle &= \langle M_f h,  k_{W,v,y} \rangle \\
&= \langle f(W) h(W) v, y \rangle \\
&= \langle  h(W) v, f(W)^* y \rangle \\
&= \langle h,  k_{W,v,f(W)^*y} \rangle .
\end{align*}
\end{proof}

\begin{lem}\label{lem:facts_kernel}
Let $k$ be a cp nc kernel on $\Omega$ with values in $B(\cE)$.
\begin{enumerate}
\item $k_{W,v,y} + k_{W',v',y'} = k_{W\oplus W', v \oplus v', y \oplus y'}$.
\item $k_{SWS^{-1},v,y} = k_{W,S^{-1}v, S^*y}$
\end{enumerate}
\end{lem}
The proof is straightforward.

\begin{lem}\label{lem:operator}
Let $k$ be a cp nc kernel on $\Omega$ with values in $B(\cE)$.
Let $f : \Omega \to \sqcup_{n=1}^\infty M_n(B(\cE))$ be a graded function, and let $T$ be an operator on $\cH$ such that
\[
T k_{W,v,y} = k_{W,v,f(W)^*y}
\]
for all $W, v, y$.
Then $f \in \mlt \cH$ and $T = M_f^*$.
\end{lem}
\begin{proof}
Using Lemma \ref{lem:facts_kernel} we find that $f$ respects direct sums and similarities, thus $f$ is an nc function.
Next,
\begin{align*}
\langle T^* h, k_{W,v,y} \rangle &=
\langle h, T k_{W,v,y} \rangle \\
&= \langle h, k_{W,v,f(W)^* y} \rangle \\
&= \langle h(W) v, f(W)^* y \rangle \\
&= \langle f(W) h(W) v,  y \rangle .
\end{align*}
Therefore, the value of $T^* h$ on $W$ is $f(W) h(W)$.
It follows that $f$ is a multiplier and that $M_f = T^*$.
\end{proof}

\begin{lem} \label{prop:bound_wot_conv_is_point_conv}
Let $k$ be a cp nc kernel on $\Omega$ with values in $B(\cE)$, and assume that for every $n \in \bN$ and every $X \in \Omega(n)$, 
\[
\overline{\operatorname{span}}\{h(X)v : v \in \bC^n, h \in \cH(k)\} =\cE^n .
\]
Let $\{f_{\alpha} \}_{\alpha \in I}$ be a bounded net of multipliers on $\cH(k)$. 
Then $f_{\alpha}$ converges to a multiplier $f$ in the weak-operator topology if and only if it is pointwise convergent, i.e., $f_{\alpha}(X) \to f(X)$ in the weak-operator topology, for all $X \in \Omega$.
\end{lem}
\begin{proof}
Assume first that $\lim f_{\alpha} = f$ in the weak-operator topology. 
Then for every $X \in \Omega(n)$, $y \in \cE^n$ and $v \in\C^n$, we have that:
\[
\langle f_\alpha(X) h(X) v,y \rangle = \langle f_{\alpha} h , k_{X,v,y} \rangle \xrightarrow[\alpha]{} \langle f h, k_{X,v,y} \rangle = \langle f(X) h(X) v, y \rangle.
\]
By assumption, vectors of the form $h(X) v$ are dense in $\cE^n$, thus $f_{\alpha}(X)$ converges to $f(X)$ in the weak-operator topology on $M_n(B(\cE))$.

Conversely, assume that $f_{\alpha}(X) \xrightarrow{\textsc{wot}} f(X)$ for every $X \in \Omega$. 
Then, as above, we find that for every $X \in \Omega(n)$, $y \in \cE^n$ and $v \in \C^n$, and for every $h \in \cH(k)$, 
\[
\langle f_{\alpha} h , k_{X,v,y} \rangle  \xrightarrow[\alpha]{} \langle f h, k_{X,v,y} \rangle.
\]
Now note that by Lemma \ref{lem:facts_kernel}, a linear combination of kernel functions is again a kernel function, therefore the kernel functions are dense in $\cH(k)$.  
Making use of the assumption that $\{f_\alpha\}$ is bounded, it follows $f_{\alpha}$ converges in the weak-operator topology to $f$. 
\end{proof}

The following is the noncommutative analogue of \cite[Lemma 2.1]{MitPau10}.

\begin{cor} \label{cor:mult_wst_closed}
Let $\Omega \subset \M_d$ and $k$ a cp nc kernel on $\Omega$ as above. 
Set $\cH(k)$ to be the nc RKHS associated to $k$ and $\cM(k)$ the algebra of multipliers. 
Then $\cM(k)$ is weak-* closed and in particular it is a dual algebra.
\end{cor}
\begin{proof}
By the Krein--Smulian theorem, it suffices to prove that the unit ball of $\cM(k)$ is weak-* closed. 
Since on norm bounded sets the $\textsc{wot}$ and weak-* topologies coincide, it is enough to prove that for every $\textsc{wot}$ convergent net $\{f_{\lambda}\}$ in the ball, the limit is also in the ball. 
If $T = \lim_{\lambda}^{\textsc{wot}} f_{\lambda}$, then by Lemma \ref{prop:bound_wot_conv_is_point_conv} there exists a pointwise limit $f$ of $f_{\lambda}$ on $\Omega$. 
Now we apply Lemma \ref{lem:operator} to get that $T = M_f$ as desired.
\end{proof}

If $\cH(k)$ is an nc RKHS on $\Omega$ with values in $\cE$, and $\cF$ is a Hilbert space, then we can consider the space $\cH(k) \otimes \cF$ as an nc RKHS consisting of nc functions with values in $B(\cE \otimes \cF)$, i.e., graded, direct sum and similarity preserving functions from $\Omega$ to $\sqcup_n M_n(\cE \otimes \cF)$.
The reproducing kernel of $\cH(k) \otimes \cF$ is given by $\tilde{k} = k \otimes I_\cF$, where
\[
\tilde{k}(Z,W)(P) =(k \otimes I_\cF)(Z,W)(P) =  k(Z,W) (P) \otimes I_\cF ,
\]
and the kernel functions are
\[
\tilde{k}_{W,v,y\otimes f} = k_{W,v,y} \otimes f.
\]
Multipliers are then nc functions with values in $B(\cE \otimes \cF)$, that is, graded, direct sum and similarity preserving functions from $\Omega$ to $\sqcup_n M_n(B(\cE \otimes \cF))$.

Similarly to left multipliers, we may consider right multipliers as well. Given a scalar valued nc function $f : \Omega \to \M_1$ and an nc RKHS $\cH(k)$, we define the right multiplication operator
 \[
 R_f (h) (W) = h(W) f(W).
 \]
 We denote by $\mlt_r (\cH(k))$ the algebra of all nc functions such that $R_f$ gives rise to a well defined (bounded) operator on $\cH(k)$.
 
\subsection{The nc Szego kernel and the nc Drury--Arveson space}\label{subsec:ncDA}
 
We now define the nc Drury--Arveson space, which was introduced by Ball, Marx and Vinnikov in \cite{BMV15b}; this nc reproducing Hilbert space will play a central role in this paper. 

Recall that $\W_d$ denotes the free monoid on $d$ generators ($d \in \bN \cup \{\infty\}$).
The {\em nc Szego kernel} $K(Z,W)$ on the nc ball $\fB_d$ of $\M_d$ defined by
\[
K(Z,W)(T) = \sum_{k\in\W_d} Z^k T W^{*k} ,
\]
for $Z \in \fB_d(n)$, $W \in \fB_d(m)$ and $T \in M_{n \times m}(\C)$. 
We are using the plain notation $K$, because this kernel will be the only reproducing kernel to be discussed from this point on. 
Consider the nc function $K_{W,v,y}$ for $W \in \fB_d(m)$, $v, y \in \C^m$, defined for $Z \in \fB_d(n)$:
\[
K_{W,v,y}(Z)u = K(Z,W)(uv^*)y = \sum_{k \in \W_d} Z^k u v^* W^{*k} y = \sum_{k \in \W_d} \langle  y, W^k v \rangle Z^k u.
\]
Thus $K_{W,v,y}$ is an nc function given by the power series
\[
K_{W,v,y}(Z) = \sum_{k \in \W_d} \langle y, W^k v \rangle Z^k .
\]
The {\em nc Drury--Arveson space} $\cH^2_{d}$ is the RKHS determined by the nc Szego kernel $K$.
In the next section we will see that $\mlt \cH^2_d = H^\infty(\fB_d)$.

\section{The bounded holomorphic functions on the ball} \label{sec:szego_and_bounded_functions}

In noncommutative multivariable operator theory, there have been several candidates for the role of ``bounded analytic functions on the unit ball''.
Our goals in this section are (1) to carry out the task of demonstrating that all the natural approaches give rise to the same algebra; and (2) to collect some facts about these algebras to be used in the sequel. 
We will compare the algebra $H^\infty(\fB_d)$ of bounded holomorphic functions on the open matrix unit ball, with the {\em noncommutative analytic Toeplitz algebra} $\cL_d$ that was studied extensively by Davidson and Pitts, and by Popescu (who denoted it $\mathscr F^\infty_d$), and with Popescu's algebra of {\em bounded free holomorphic functions on the operatorial ball} \cite{Popescu06b}, denoted $H^\infty(B(\mathcal{X})^d_1)$.
Note that although similar words and even the same notations are used to describe these algebras, the definitions are different.
Nevertheless, these algebras are all isometrically isomorphic. 
We will also show that these algebras can be identified with the multiplier algebra $\mlt \cH^2_d$ of the nc Drury--Arveson space.

In this section, we work with any $d \in \bN \cup \{\infty\}$.
Let $\cE$ be a $d$ dimensional Hilbert space, and define the full Fock space to be
\[
\cF(\cE) = \C \oplus \cE \oplus (\cE \otimes \cE) \oplus \ldots .
\]
Define the shift operators $L_i \xi = e_i \otimes \xi$ ($i=1, \ldots, d$) where $\{e_1, \ldots, e_d\}$ is an orthonormal basis for $\cE$.
The tuple $L = (L_1, \ldots, L_d)$ is simply referred to as {\em the shift}.

Davidson and Pitts defined $\cL_d$ to be the closure of $\alg(L)$ in the weak-operator topology (more precisely, their underlying Hilbert space was $\ell_2(\mathbb W_d)$ and not $\cF(\cE)$, but these are clearly, up to a natural identification, the same Hilbert space).
They also showed that the weak-operator and weak-$*$ topologies coincide \cite{DavPitts1}.
The algebra $\cL_d$ is also Popescu's {\em noncommutative Hardy algebra} \cite{Popescu91}, which he denotes by $\mathscr F^\infty_d$.
In \cite{Popescu89} $\mathscr F_d^\infty$ was defined to be the algebra 
all $g \in \cF(\cE)$ for which $\sup_{p}\|g\otimes p\|<\infty$, where $p$ runs over all polynomials of norm one. However, it was shown that this is the same as the closure of $\alg(L)$ in the weak-operator topology.
This algebra is also a particular instance of Muhly and Solel's {\em Hardy algebra} \cite{MS04}, and would be denoted in their context as $H^\infty(\cE)$.

The {\em algebra of bounded holomorphic functions on the operatorial unit ball}, denoted by $H^\infty(B(\mathcal{X})^d_1)$ in Popescu's work, is defined to be the algebra of all functions $F$ which have a power series representation on $(B(\mathcal X)^d)_1$ and satisfy $\sup \| F(X_1, \dots,X_d) \| < \infty$ (where the supremum runs over all $(X_1,\dots,X_d) \in (B(\mathcal X)^d)_1$ for some infinite dimensional Hilbert space $\mathcal X$); see \cite{Popescu06b}.
This is not quite $H^\infty(\fB_d)$ as was defined above, since the arguments of functions are tuples of operators, and since the functions considered are a priori only those given by a power series.

\begin{thm}\label{thm:HinftyVSPopescu}
The algebras $H^\infty(B(\mathcal{X})_1^d)$, $\cF_d^{\infty}$ and $H^\infty(\fB_d)$ are completely isometrically isomorphic.
\end{thm}

\begin{proof}
In \cite[Theorem 3.1]{Popescu06b} Popescu showed that $H^\infty(B(\mathcal{X})_1^d)$ and $\cF_d^{\infty}$ are completely isometrically isomorphic. 

Now we proceed to prove the completely isometric isomorphism with $H^{\infty}(\fB_d)$. It is clear that every $F \in H^\infty(B(\mathcal{X})_1^d)$ defines a bounded holomorphic function $\hat{F}$ on $\fB_d$, and that the map $F \mapsto \hat{F}$ is a completely contractive homomorphism.

Let $f \in H^\infty(\fB_d)$.
By \cite[Theorem 5.1]{MS15}, $f$ has a power series representation $f(z) = \sum_\alpha c_\alpha z^\alpha$ with radius of convergence
\[
R = \left(\overline{\lim}_{k} \left| \sum_{|\alpha| = k} |c_\alpha|^2\right|^{\frac{1}{2k}}\right)^{-1} \geq 1.
\]
For every $r\in (0,1)$ define $f_r(z) = f(rz)$.
Then $f_r \in H^\infty(\fB_d)$ and $\|f_r\|_\infty \leq \|f\|_\infty$.
Moreover, since the power series of $f_r$ converges uniformly in the closed unit ball of $B(H)^d$ ($H$ is any Hilbert space), we can plug in the shift $L$.
Let us fix $\{e_j\}_{j=1}^d$ an orthonormal basis for $\cE$ and set $\cE_n$ to be the (closed) subspace of $\cE$ spanned by $e_1,\ldots,e_n$. Let $P_n$ be the sequence of projections onto
\begin{equation}\label{eq:Hn}
H_n = \C \oplus \cE_n \oplus \cdots \oplus \cE_n^{\otimes n} \subseteq \cF(\cE).
\end{equation}
If we view $H_n$ as a subspace of $\F_d$, then it corresponds to the subspace spanned by the monomial of degree at most $n$ in the first $n$ variables. Since $H_n$ are all co-invariant for $L$, we have that
\[
\|f_r(L)\| = \lim_n \|P_n f_r(L)P_n \| = \lim_n\|f_r(P_n L P_n) \| \leq \|f_r\|_\infty \leq \|f\|_\infty.
\]
But by \cite[Theorem 3.1]{Popescu06b}, this implies that $F = \sum_\alpha c_\alpha L_\alpha$ determines an element in $\mathscr F^\infty_d$ with $\lim_{r\to 1}\|f_r(L)\| = \lim_{r\to 1}\|F_r\| = \|F\|$.
Since $\hat{F} = f$, the mapping $F \mapsto \hat{F}$ is a surjective isometric isomorphism.

Finally, to show that $F \mapsto \hat{F}$ is completely isometric, let $f \in M_n(H^\infty(\fB_d)) = H^\infty(\fB_d) \otimes M_n$.
As above, there is an $F \in M_n(\mathscr F^\infty_d)$ such that $\hat{F} = f$, and we need to show that
\[
\|F\| \leq \|f\|_\infty = \sup_{X \in \fB_d}\|f(X)\| .
\]
But again $\|F\| = \lim_{r \to 1}\|F_r\|$ and $F_r = f_r(L)$.
Repeating the above computation with $P_n \otimes I_{\C^n}$ instead of $P_n$, the required inequality follows from
\[
\|f_r(L)\| = \lim_n \|(P_n \otimes I_{\C^n}) f_r(L)(P_n \otimes I_{\C^n}) \| = \lim_n\|f_r(P_n L P_n) \| \leq \|f_r\|_\infty \leq \|f\|_\infty.
\]
\end{proof}

One could also deduce the isomorphism between $H^\infty(B(\mathcal{X})_1^d)$ and $H^{\infty}(\fB_d)$ from \cite[Corollary\ 4.5.1]{BMV15b}. In the notations of \cite{BMV15b} one takes $\cU = \cY = \C$ and $Q(Z) = \begin{bmatrix} Z_1&Z_2&\dots&Z_d \end{bmatrix}$ (we will use a similar consideration in Proposition \ref{prop:bmv}). An earlier instance of such a result can be found in \cite{AlKV06} in the case of the noncommutative polydisc.


It is clear that the natural unitary transformation mapping $\cF(\cE)$ onto $\ell_2(\W_d)$ implements a unitary equivalence between
the $\textsc{wot}$-closed algebras $\cL_d$ and $\mathscr F_d^\infty$. In Proposition \ref{prop:mult=F^infty_d}, we will show that similarly, the unitary transformation mapping $\cF(\cE)$ onto $\cH_2^d$ implements a unitary equivalence between the $\textsc{wot}$-closed algebras $\mlt \cH_2^d$ and $\mathscr F_d^\infty$. Thus, the $\textsc{wot}$-closed algebras $\cL_d$, $\mathscr F_d^\infty$, and $\mlt \cH_2^d$ are all unitarily equivalent. Combining this with Theorem \ref{thm:HinftyVSPopescu} we conclude that all five  $\textsc{wot}$-closed algebras  $\cL_d$, $\mathscr F_d^\infty$, $\mlt \cH_2^d$, $H^\infty(B(\mathcal{X})_1^d)$, and $H^\infty(\fB_d)$ are completely isometrically isomorphic.

We start establishing these isomorphisms by showing that $\cH^2_d$ can be identified with the Fock space $\cF(\cE)$ over a $d$-dimensional Hilbert space $\cE$.

\begin{prop}\label{prop:identify_Fock_HK}
$\cH^2_{d}$ can be identified with the Fock space $\cF(\cE)$ (where $\cE$ is a $d$ dimensional Hilbert space), which in turn can be identified with the Hilbert space of formal power series in $d$ variables
\[
\cF_d = \{ F(z) = \sum_{k \in \W_d} a_k z^k : \|F\|^2 :=\sum_{k \in \W_d} |a_k|^2 < \infty \}.
\]
\end{prop}
\begin{proof}
The identification of $\cF(\cE)$ with the space of formal power series with square summable coefficients is clear.
The map $U : \cH^2_{d} \to \cF_d$ mapping the kernel function $K_{W,v,y}$ to the power series $\sum_{k \in \W_d} \langle y, W^k v \rangle z^k$ is readily seen to preserve inner products, thus it extends to an isometry.
On the other hand, it not hard to find, for every $m \in \W_d$, a kernel function $K_{W,v,y}$ that gets mapped to the monomial $z^m$ (e.g, using compression of $L$ to finite dimensional subspaces as in \eqref{eq:Hn}).
Thus, $U$ is unitary.
\end{proof}

\begin{remark}
Here is another point of view on the unitary equivalence of $\cH^2_{d}$ and $\cF_d$.
It is easy to see that $\cF_d$ is a Hilbert space of nc functions on $\fB_d$ in which point evaluation is bounded (just plug in a tuple into the formal power series and use Cauchy--Schwarz).
Now $\{z^k : k \in \W_d\}$ is an orthonormal basis for $\cF_d$.
By \cite[Theorems 3.3 and 3.5]{BMV15a}, the space $\cF_d$ is associated to a kernel $\tilde{K}$ which is given by $\tilde{K}(Z,W)(P) = \sum_{k \in \W_d} Z^k P W^{*k}$ (this is an nc analogue to a familiar result in the classical theory of RKHS).
Thus the Hilbert spaces $\cH^2_{d}$ and $\cF_d$ actually give rise to the same reproducing kernel Hilbert spaces on $\fB_d$.
\end{remark}

Recall that the multiplier algebra of $\cH^2_{d}$ is the algebra $\mlt\cH^2_{d}$ of all nc functions from $\fB_d$ to $\bM_d$ to $\M_1$ such that $fh \in \cH^2_{d}$ for all $h \in \cH^2_{d}$.

\begin{prop}\label{prop:mult_in_H}
$\mlt \cH^2_{d} \subseteq \cH^2_{d}$.
\end{prop}
\begin{proof}
The constant function $\fB_d(n) \ni Z \mapsto I_n$ is in $\cH^2_{d}$, since the constant formal power series $1$  is clearly in $\cF_d$.
From this the containment $\mlt \cH^2_{d} \subseteq \cH^2_{d}$ follows immediately.
\end{proof}

\begin{prop} \label{prop:mult=F^infty_d}
The unitary $U$ of Proposition \ref{prop:identify_Fock_HK} maps $\mlt \cH^2_{d}$ onto $\mathscr F^\infty_d$, when considered as subspaces of $\cH^2_d$ and $\cF(\cE)$, respectively. Moreover, this unitary implements a completely isometric isomorphism via conjugation between these $\textsc{wot}$-closed algebras, when considered as subalgebras of $B(\cH^2_d)$ and $B(\cF(\cE))$, respectively. 
In fact, these two maps coincide.
\end{prop}
\begin{proof}
Recall that Popescu's original definition of $\mathscr F_d^\infty$ is as the subspace of the Fock space $\cF(\cE)$, of a $d$-dimensional Hilbert space $\cE$, consisting of all $g \in \cF(\cE)$ for which $\sup_{p}\|g\otimes p\|<\infty$, where $p$ runs over all polynomials of norm one \cite{Popescu91}. It is easy to see that this space is, in fact, the space of all $g\in \cF(\cE)$ for which $g \otimes f \in \cF(\cE)$ for all $f \in \cF(\cE)$; see e.g. Popescu's observation in \cite{Popescu06b} after equation (3.1).
It now follows by Propositions \ref{prop:mult_in_H} and \ref{prop:identify_Fock_HK}  that the unitary $U$ from the latter proposition maps $\mlt \cH^2_{d}$ onto $\mathscr F^\infty_d$.

We now turn to the unitary equivalence. 
When thinking of $\mlt \cH^2_{d}$  as a $\textsc{wot}$-closed subalgebra of $B(\cH^2_d)$, we identify $f$ with the multiplication operator $M_f$; when thinking of $\mathscr F^\infty_d$  as a $\textsc{wot}$-closed subalgebra of $B(\cF(\cE))$, we identify $f$ with $f(L) = f(L_1,\dots,L_d)$, where $L_1,\dots,L_d$ are the creation operators on the associated Fock space (see \cite[Corollary 3.5]{Popescu91}). 
Then for every $1\leq i \leq d$, $W \in \fB_d(m)$, $v,y \in \mathbb C^m$, and $n \in \mathbb W_d$ we have
\[
\begin{split}
\langle UM_{z_i}^* K_{W,v,y} , z^n \rangle &= \langle UK_{W,v,W_i^*y} , z^n \rangle
                                                                         = \sum_{k\in \mathbb W_d} \langle W_i^*y, W^k v \rangle \langle z^k , z^n \rangle \\
                                                                       &= \langle W_i^*y, W^n v \rangle
                                                                         = \langle y, W_i W^n v \rangle \\
                                                                       &= \sum_{k\in \mathbb W_d} \langle y, W^k v \rangle \langle z^k , L_i z^n \rangle
                                                                         = \langle L_i^* U K_{W,v,y} , z^n \rangle.
\end{split}
\]
Thus, $UM_{z_i}U^*=L_i$. 
Next, if $f$ is a polynomial then
$UM_{f}U^*=f(UM_{z}U^*)=f(L)$. 
Since every $f \in \mathscr F^\infty_d$ is the $\textsc{wot}$-limit of a bounded polynomial net and $\mlt \cH^2_d$ is $\textsc{wot}$-bounded closed, we have that $UM_{f}U^*=f(UM_{z}U^*)=f(L)$ for every $f \in \mathscr F^\infty_d$. 
As $U(\mlt \cH^2_d) = \mathscr F_d^\infty$, we are done.

Finally, it is clear that on the level of formal power series the map sending $f$ to $Uf$ and the map sending $f$ to the formal power series associated with $f(S)$ agree. 
Therefore, they coincide on $\mlt \cH^2_{d}$.
\end{proof}

\begin{cor} \label{cor:unitary_equiv}
The $\textsc{wot}$-closed algebras $\cL_d$, $\mathscr F_d^\infty$, $H^\infty(B(\cX)^d_1)$, $\mlt \cH^2_d$ and $H^\infty(\fB_d)$ are all completely isometrically isomorphic.
\end{cor}

\begin{cor} \label{cor:mult=H^infty}
The completely isometric isomorphism $H^\infty(\fB_d) \to \mlt \cH^2_{d}$ can be chosen to be the identity map. In particular $H^\infty(\fB_d)$ and $\mlt \cH^2_{d}$ are the same set.
\end{cor}
\begin{proof}
This follows by applying the composition of isomorphisms
\[
\mlt \cH^2_{d} \to \mathscr F^\infty_d \to H^\infty(B(\cX)^d_1)  \to H^\infty (\fB_d)
\]
on the level of formal power series on which it is easily seen to be the identity.
\end{proof}

\begin{remark}
The equality $\mlt \cH^2_{d} = H^\infty(\fB_d)$ also follows from \cite[Theorem 5.5]{BMV15b} (the equivalence of conditions (1) and (1') when taking $a(Z) = I_n$).
\end{remark}

\begin{remark}
We have seen that the various algebras $H^\infty(\fB_d)$, $\cL_d$ (or $\mathscr F^\infty_d$ in Popescu's notation), $H^\infty(B(\mathcal{X})^d)$, and $\mlt \cH^2_d$, are all the same.
Henceforth, these algebras will be identified.
In particular, the free shift $L = (L_1, \ldots, L_d)$ on $\cF(\bC^d)$ will be identified with the tuple $M_z = (M_{z_1}, \ldots, M_{z_d})$ of multiplication by coordinate operators, which will be identified with $z = (z_1, \ldots, z_d)$.
\end{remark}

Using \cite[Theorem 1.2]{DavPitts1} we obtain
\begin{cor}\label{cor:commutant}
$\mlt \cH^2_{d}$ and $\mlt_r\cH^2_{d}$ are unitarily equivalent and are mutual commutants: $\mlt \cH^2_{d} = (\mlt_r \cH^2_{d})'$ and $\mlt_r \cH^2_{d} = (\mlt \cH^2_{d})'$.
\end{cor}
\begin{proof}
The algebra $\mlt_r \cH_d^2$ is the algebra of nc functions $f$ on $\fB_d$, such that for every $g \in \cH_d^2$, we have $g f \in \cH_d^2$. Following \cite{DavPitts1} we note that the map that sends the monomial $z^{\alpha}$ to the reversed monomial $z^{\tilde{\alpha}}$ extends to the power series representing functions in $\cH_d^2$. 
In fact it, is a unitary and implements an isomorphism between $\mlt \cH_d^2$ and $\mlt_r \cH_d^2$. The second claim is precisely \cite[Theorem 1.2]{DavPitts1}.
\end{proof}

The works of Davidson and Pitts \cite{DavPitts2,DavPittsPick,DavPitts1}, Arias and Popescu \cite{AriasPopescu} and Popescu \cite{Popescu95b,Popescu96,Popescu06b} introduce the algebra $H^{\infty}(\fB_d)$ as the $\textsc{wot}$-closure of the multiplication operators by $\F_d$ on $\cH_d^2$. It was proved in \cite{DavPitts1} that in fact the weak-* and $\textsc{wot}$ topologies coincide on $H^{\infty}(\fB_d)$. Hence in particular the unit ball of $H^{\infty}(\fB_d)$ is $\textsc{wot}$-compact.

A final useful fact about $H^\infty(\fB_d)$ is the following. 
\begin{thm}
Let $f \in H^\infty(\fB_d)$. Then $f$ has a Taylor series $f = \sum_{n=0}^\infty f_n$ that converges pointwise, converges Ces\`{a}ro in the weak-$*$ topology, and $f \mapsto f_n$ is bounded.
\end{thm}
\begin{proof}
This follows from \cite[pp. 405--406]{DavPitts1}.
\end{proof}

\section{The complete Pick property} \label{sec:complete_pick}

For a set $\Omega \subseteq \fB_d$ we let
\[
\cH_\Omega = \ol{\spn}\{K_{W,u,y} : W \in \Omega(m), u,y \in \C^m, m \in \N\}.
\]
Let $\cI_\Omega$ denote the space of all functions in $\cH^2_{d}$ that vanish on $\Omega$.
For a set $S \subseteq \cH^2_{d}$, let $V(S)$ denote the variety determined by $S$, that is, the intersection of the zero sets of all functions in $S$. 
Finally, let $\cJ_\Omega$ denote the two sided ideal in $H^\infty(\fB)$ that consists of functions vanishing on $\Omega$.

The following lemma shows that varieties determined by functions in $\cH^2_{d}$ are the same as varieties determined by functions in $\mlt \cH^2_{d}$.
\begin{lem}\label{lem:HvsMltVarieities}
Let $\fV$ be an nc variety determined by functions in $\cH^2_{d}$; that is, there exists a family $E \subseteq \cH^2_{d}$ such that
\[
\fV = \{W \in \fB_d :  f(W) = 0 \,\, \textup{ for all } \,\, f \in E\}.
\]
Then $\fV$ is determined by a family of functions in $\mlt \cH^2_{d}$; that is, there exists a family $F \subseteq \mlt\cH^2_{d}$ such that
\[
\fV = \{W \in \fB_d :  f(W) = 0 \,\, \textup{ for all } \,\, f \in F\}.
\]
\end{lem}
\begin{proof}
We modify the argument in \cite[Theorem 9.27]{AM_Book} to our setting.
For every $h$, let $V(h) = \{W \in \fB_d: h(W) = 0\}$.
It suffices to prove that for all $h \in \cH^2_{d}$ and all $X \notin V(h)$, there exists $f \in \mlt \cH^2_{d}$ such that $f\big|_{V(h)} = 0$ and $f(X) = h(X)$.

To show the existence of such a multiplier $f$, we attempt to define a bounded operator $T$ on $\cH_{V(h) \cup \{X\}}$ by $T K_{X,v,y}= K_{X,v,h(X)^*y}$ for all $v,y \in \C^n$ and $T K_{W,v,y} = 0 = K_{W,v,h(W)^*y}$ for all $W \in V(h)$.
Clearly, this defines a bounded operator on $\cH_{V(h)}$.
It also defines a bounded, well defined operator on $\cH_{\{X\}}$, because being a convergent power series of $n \times n$ matrices in $X$, $h(X)$ is actually equal to a polynomial in $X$, say $h(X) = p(X)$ for some $p \in \F_d$.
Therefore, $K_{X,v,y} \mapsto K_{X,v,h(X)^*y}$ agrees with the action of the adjoint of the bounded multiplier $M_p$ on $\cH_{\{X\}}$.

What is not yet clear, is that the operator $T$ that we defined is well defined on the intersection $\cH_{V(h)} \cap \cH_{\{X\}}$.
We need to show that if $K_{X,v,y} \in \cH_{V(h)}$, then $K_{X,v,h(X)^*y} = 0$.
But if $K_{X,v,y} \in \cH_{V(h)}$, then $K_{X,v,y}$ is the limit of linear combinations of the kernel functions $K_{W,u,x}$ for $W \in V(h)$ and thus is orthogonal to $hq$ for any polynomial $q$;
indeed
\[
\langle hq, K_{W,u,x} \rangle = \langle  h(W)q(W) u, x \rangle = 0
\]
for all $W \in V(h)$ and all $u,x$ of appropriate size.
Therefore $\langle h(X) q(X) v, y \rangle = \langle hq, K_{X,v,y} \rangle = 0$ for every polynomial $q$.
Thus
\begin{align*}
\|K_{X,v,h(X)^*y}\|^2
&= \langle K(X,X)(v v^*) h(X)^*y, h(X)^* y \rangle \\
&= \langle \sum_{k \in \W_d} X^k (v v^*) X^{k*} h(X)^*y, h(X)^* y \rangle \\
&= \sum_{k \in \W_d} \langle v^* X^{k*} h(X)^* y, v^* X^{k*} h(X)^* y \rangle \\
&= \sum_{k \in \W_d} |\langle  h(X)X^k v, y \rangle |^2 = 0.
\end{align*}

Since $T$ defined above is a bounded operator, such that $T K_{W,v,y} = K_{W,v,h(W)^*y}$ (for $W \in V(h) \cup \{X\}$), by Lemma \ref{lem:operator} there exists $f \in \mlt \cH_{V(h) \cup \{X\}}$ such that $f\big|_{V(h)} = 0  = h\big|_{V(h)}$ and $f(X) = h(X)$.
By Theorem \ref{thm:quotient_mult} below (which does not depend on the lemma we are now proving), $f$ extends to a multiplier in $\mlt \cH^2_{d} = H^\infty(\fB_d)$.
\end{proof}

\begin{remark}
By Lemma \ref{lem:facts_kernel}, a linear combination of kernel functions is again a kernel function. 
We  also note that the kernel functions are not necessarily independent. 
For every $X \in \fB_d$ and every $f \in H^{\infty}(\fB_d)$, the matrix $f(X)$ is an element of the algebra generated by the coordinates of $X$. 
Thus, if $X \in \fB_d(n)$ is generic (i.e. the algebra generated by the coordinates of $X$ is $M_n(\C)$) and $v \neq 0$, then $\{ f(X) v \mid f \in H^{\infty}(\fB_d)\} = \C^n$ and thus $K_{X,v,y} = 0$ if and only if $y = 0$. On the other hand if the coordinates of $X$ have a non-trivial joint invariant subspace $U \subset \C^n$, then we take $0 \neq v \in U$ and $0 \neq y \in U^{\perp}$ and obtain that $K_{X,v,y} = 0$.
\end{remark}

For $\Omega \subseteq \fB_d$ we denote by $\Omega'$ the {\em full nc envelope} of $\Omega$ in $\fB_d$.
Recall from \cite{BMV15b} that this means that $\Omega'$ is the smallest subset of $\fB_d$ that is closed under direct sums and left injective intertwiners, in the sense that if $Z \in \Omega'$, if $\cI$ is an injective matrix, and if $\cI \tilde{Z} = Z \cI$, then $\tilde{Z}$ is also in $\Omega'$.
The zero set of an nc function is clearly closed under direct sums and injective left intertwiners, thus if an nc function vanishes on $\Omega$, then it also vanishes on $\Omega'$.

The following lemma is an analogue of \cite[Proposition 2.2]{DRS15}.
\begin{lem}\label{lem:HOmega}
For $\Omega \subseteq \fB_d$,
\[\cH_{\Omega}  = \cH_{\Omega'} = \cH_{V(\cI_\Omega)} = \cH_{V(\cJ_\Omega)}.\]
\end{lem}

\begin{proof}
Clearly $\cH_{\Omega} \subseteq \cH_{\Omega'}$.
If $f \in \cH^2_{d}$ vanishes on $\Omega$, then it also vanishes on $\Omega'$, hence $\cH_{\Omega}^\perp \subseteq \cH_{\Omega'}^\perp$.
Thus $\cH_{\Omega} = \cH_{\Omega'}$.

Since $\Omega \subseteq V(\cI_\Omega) \subseteq V(\cJ_\Omega)$, we have
\[\cH_{\Omega} \subseteq \cH_{V(\cI_\Omega)} \subseteq \cH_{V(\cJ_\Omega)} .\]
On the other hand, if $h \in \cH_{\Omega}^\perp$, then $h \big|_\Omega = 0$ so $h \in \cI_\Omega$.
This means that $h(W) = 0$ for all $W \in V(\cI_\Omega)$, therefore $h \in \cH_{V(\cI_\Omega)}^\perp$.
It follows that $\cH_{\Omega} = \cH_{V(\cI_\Omega)}$.

To prove the last equality, we note that $V(\cI_\Omega) = V(\cJ_\Omega)$.
To see this, recall that $V(\cI_\Omega)$ is the smallest nc variety determined by functions in $\cH^2_{d}$ that contains $\Omega$, and $V(\cJ_\Omega)$ is the smallest nc variety determined by functions in $\mlt \cH^2_{d}$ that contains $\Omega$.
Now invoke Lemma \ref{lem:HvsMltVarieities}.
\end{proof}

\begin{lem}\label{lem:exist_poly}
Let $\cE$ be a finite dimensional Hilbert space, and suppose that $Z \in \fB_d(n)$ and $W \in M_n(B(\cE))$.
Then $W$ belongs to the unital algebra $\alg (Z) \otimes B(\cE)$ if and only if the map $Z \mapsto W$ extends to an nc function on $\Omega'$ --- the full nc envelope $\fB_d(n)$ of the singleton $\Omega = \{Z\}$.
The nc function can be chosen to be a polynomial with coefficients in $B(\cE)$.
\end{lem}

\begin{proof}
If $W \in \alg (Z) \otimes B(\cE)$, then $W = p(Z)$ for some operator coefficient polynomial $p$.
Such a polynomial $p$ is clearly an nc function $\bM_d \to \sqcup_{k=1}^\infty M_k(B(\cE))$ extending $Z \mapsto W$.

Suppose that $Z \mapsto W$ extends to an nc function $f_0$ on the full nc closure $\Omega'$ of $\Omega = \{Z\}$ relative to $\fB_d$.

\noindent{\bf Claim.} If $M \subseteq \C^n \otimes \cE$ is an invariant subspace for $\alg(Z) \otimes B(\cE)$, then it is also an invariant subspace for $W$.

Indeed, such an invariant subspace must be of the form $L \otimes \cE$, where $L$ is an invariant subspace for $Z$.
Now, if for all $i$
\[
Z_i = \begin{pmatrix} \tilde{Z}_i & * \\ 0 & * \end{pmatrix} ,
\]
then $Z_i \cI = \cI \tilde{Z_i}$ for $\cI = \left(\begin{smallmatrix} I \\ 0  \end{smallmatrix}\right)$, and so $\tilde{Z} \in \Omega'$.
Thus $f_0$ extends to $\tilde{Z}$, and since it is an nc function it respects intertwiners, so $W \cI = f_0(Z) \cI = \cI f_0(\tilde{Z})$.
Therefore
\[
W = \begin{pmatrix} f_0(\tilde{Z}) & * \\ 0 & * \end{pmatrix}.
\]
This proves the claim.

Now, assume for contradiction that $W \notin \alg(Z) \otimes B(\cE)$.
Then there exists a linear functional $\psi : M_n \otimes B(\cE) \to \C$ such that $\psi( p(Z)) = 0$ for all $p \in \F_d \otimes B(\cE)$, while $\psi(W) \neq 0$.
We may find some $k \in \bN$ and $u,v \in (\C^{n} \otimes \cE)^{k}$ such that
$\psi(T) = \langle T^{(k)}u, v \rangle$ for all $T \in M_n \otimes B(\cE)$, where $T^{(k)}$ denotes the $k$-fold ampliation of $T$, i.e., $T^{(k)} = T \oplus T \oplus \cdots \oplus T$ ($k$ times).
Denote $M = [\alg(Z^{(k)}) \otimes B(\cE)] u$.
Since $\langle p(Z^{(k)}) u, v \rangle = 0$ for all $p \in \F_d \otimes B(\cE)$, $M$ is a nontrivial invariant subspace for $\alg(Z^{(k)}) \otimes B(\cE)$, and $v \in M^\perp$.
Applying the claim to $Z^{(k)}$ and $W^{(k)}$ in place of $Z$ and $W$, we find that $M$ is invariant under $W^{(k)}$ as well.
In particular, $\psi(W) = \langle W^{(k)} u, v \rangle = 0$, a contradiction.
\end{proof}

\begin{lem}\label{lem:mu_iota}
Let $\Omega \subseteq \fB_d$, and let $\cJ_\Omega$ be the ideal in $H^\infty(\fB_d)$ consisting of all functions vanishing on $\Omega$.
Then $\overline{\cJ_\Omega\cH^2_{d}} =  \cH^\perp_\Omega$.
\end{lem}

\begin{proof}
Put $\cG = \overline{\cJ_\Omega\cH^2_{d}}$.
The space $\cG$ is invariant under $\mlt \cH^2_{d}$ and $\mlt_r \cH^2_{d}$.
To prove that $\cG = \cH^\perp_\Omega$, we invoke the correspondence between subspaces of $\cF(\C^d)$ invariant under the left and right shift operators, and two-sided ideals in $\cL_d$ (developed in \cite[Section 2]{DavPitts2}), together with the identifications made in Theorem \ref{thm:HinftyVSPopescu} and Proposition \ref{prop:identify_Fock_HK}.

By \cite[Theorem 2.1]{DavPitts2}, the map $\mu$ that sends a weak-operator closed ideal $J \triangleleft H^\infty(\fB_d)$ to its closure in $\cH^2_d$, is invertible, with inverse $\iota : M \mapsto M \cap H^\infty(\fB_d)$.
The linear space $\cH_\Omega^\perp$ is equal to the space of all functions in $\cH^2_{d}$ that vanish on $\Omega$.
It is therefore invariant under left and right multiplications.
The set $\cN = \cH^{\perp}_\Omega \cap H^\infty(\fB_d) = \iota(\cH^\perp_{\Omega})$ is therefore a weak-operator closed two sided ideal, and by definition it is equal to all functions in $H^\infty(\fB_d)$ that vanish on $\Omega$.
Hence $\cN = \cJ_\Omega = \iota(\cG)$, and we conclude that $\cG = \cH_\Omega^\perp$.
\end{proof}

Let $k$ be an nc kernel on a set $\widetilde{\Omega} \subseteq \bM_d$.
Then $k$ is said to have the {\em complete Pick property} if whenever $\cE$ is a finite dimensional Hilbert space, $\Omega \subset \widetilde{\Omega}$, and we are given an nc function $f_0:\Omega \to \M_1 \otimes B(\cE) := \sqcup_n M_n(B(\cE))$ that extends to an nc function defined on the full nc envelope ${\Omega'}$ of $\Omega$ (in $\widetilde{\Omega}$), then there exists a multiplier $f \in \mlt (\cH(k) \otimes \cE) = \mlt \cH(k \otimes I_\cE)$ such that $f\big|_{\Omega} = f_0$ and $\|M_f\|_{\mlt \cH(k) \otimes \cE} \leq 1$, if and only if  the kernel
\begin{equation}\label{eq:pick}
{K}^{dBR}(Z,W)(P) = k(Z,W)(P) \otimes I_\cE - f_0(Z) \left[k(Z,W)(P)\otimes I_\cE \right] f_0(W)^*
\end{equation}
is a cp nc kernel on $\Omega$.

If $k$ has the complete Pick property, then it is called a {\em complete Pick kernel}, and $\cH(k)$ is said to be a {\em complete Pick space}.

\begin{remark}\label{rem:dbR}
It follows from the definitions of nc reproducing kernel Hilbert space, that when $\Omega = \widetilde{\Omega}$, then the positivity of the kernel $K^{dBR}$ as in \eqref{eq:pick} on $\widetilde{\Omega}$ is a necessary and sufficient condition for $f_0$ to be a multiplier of norm less than or equal to $1$.
Thus, for any kernel (not necessarily a complete Pick kernel), if $f$ is a multiplier of norm less than or equal to $1$ and $f_0 = f \big|_{\Omega}$, then $K^{dBR}$ is a completely positive nc kernel on $\Omega$.
The special feature of kernels with the complete Pick property, is that positivity of $K^{dBR}$ on $\Omega$ is sufficient for the existence of a contractive multiplier extending $f_0$.
\end{remark}

\begin{thm}\label{thm:complete_Pick}
$\cH^2_{d}$ is a complete Pick space.
\end{thm}
\begin{remark}
The theorem follows from \cite[Corollary 5.6]{BMV15b} as a special case.
Variants of this theorem also appeared before (see \cite{DavPittsPick, AriasPopescu, MS04}), but not quite in this form. 
We will give here a proof using the methods of \cite{DavPittsPick} --- in particular the distance formula --- to spell out how these operator algebraic techniques apply to this nc function theoretic problem. 
The referee has pointed out that \cite[Theorem 2.1]{AriasPopescu}, that is in turn based on Popescu's commutant lifting theorem \cite{Popescu92}, can be also be used as a basis for a proof of the complete Pick property. 
\end{remark}
\begin{proof}
We need to show that for every finite dimensional Hilbert space $\cE$, every $\Omega \subset \fB_d$, and every nc function $f_0:\Omega \to \M_1\otimes B(\cE)$ that extends to an nc function defined on the full nc envelope $\Omega'$ of $\Omega$, the following holds:

\begin{center} There exists a multiplier $f \in \mlt(\cH^2_{d} \otimes \cE)$ such that $f\big|_{\Omega} = f_0$ and $\|f\|_{\mlt (\cH^2_{d} \otimes \cE)} \leq 1$, if and only if  the kernel $K^{dBR}$ associated with $K$ as in  \eqref{eq:pick}
is a cp nc kernel on $\Omega$.
\end{center}

By Remark \ref{rem:dbR}, we need only prove that positivity of $K^{dBR}$ implies the existence of a multiplier $f$.
We will prove the result for a finite set $\Omega$.
The result for arbitrary sets follows by Corollary \ref{cor:mult_wst_closed} and a compactness argument.

Let $\Omega = \{Z_1, \ldots, Z_n\}$ and suppose that $f_0(Z_i) = W_i$ for $i=1, \ldots, n$.
By taking direct sums and using the assumption that $f_0$ extends to $\Omega'$, we may assume that $\Omega = \{Z\}$ and $f_0(Z) = W \in M_n$. 
By the assumption that $Z \mapsto W$ extends to an nc function on $\Omega'$, together with Lemma \ref{lem:exist_poly}, there is some polynomial $p \in \F_d \otimes B(\cE)$ (that is, a polynomial with matrix coefficients) such that
$p(Z) = W$.

Assuming that $K^{dBR}$ defined as in \eqref{eq:pick} is a cp nc kernel on $\Omega$, we need to find a contractive multiplier $f$ satisfying $f(Z) = W$.
Let $\cJ = \cJ_\Omega$ be the ideal in $H^\infty(\fB_d)$ consisting of all functions vanishing on $\Omega$.
If we can find $g \in \cJ \otimes B(\cE)$ such that $\|p+g\|\leq 1$, then putting $f = p+g$ we will be done.
Since $\cJ$ is weak-$*$ closed, it suffices to prove that $\dist(p, \cJ \otimes B(\cE)) \leq 1$.

Let $\cG = \overline{\cJ\cH^2_{d}}$.
The space $G$ is invariant under $\mlt \cH^2_{d}$ and $\mlt_r \cH^2_{d}$.
By Davidson and Pitts's distance formula\cite[Corollary 2.2]{DavPittsPick}, we have,
\[
\dist(p,\cJ \otimes B(\cE)) = \|(P_{\cG^\perp} \otimes I_{\cE}) M_p (P_{\cG^\perp} \otimes I_{\cE})\| ,
\]
so we are led to compute the norm of this compression.

By Lemma \ref{lem:mu_iota}, $\cG^\perp = \cH_\Omega$.
Letting $h = \sum_i K_{Z,u_i,y_i} \otimes \epsilon_i = \sum_i K_{Z,u_i,y_i \otimes \epsilon_i}$ be an element of $\cG^\perp \otimes \cE = \cH_\Omega \otimes \cE$, we write $v_i = y_i \otimes \epsilon_i$ and calculate (using $p(Z) = f_0(Z)$)
\begin{align*}
& \|h\|^2 - \|M_p^* h\|^2
= \\
&\sum_{i,j} \Big\langle \left[\Big(K(Z,Z)(u_j u_i^*) - f_0(Z) K(Z,Z)(u_j u_i^*)f_0(Z)^*\Big) \otimes I_\cE \right] v_j, v_i \Big\rangle \\
&\geq 0,
\end{align*}
(by the assumption that \eqref{eq:pick} is a cp nc kernel on $\Omega$).
This computation shows that $\|M_p^* \big|_{\cG^\perp \otimes \cE}\| \leq 1$.
But since $\|M_p^* \big|_{\cG^\perp \otimes \cE}\| = \|(P_{\cG^\perp} \otimes I_{\cE}) M^*_p (P_{\cG^\perp} \otimes I_{\cE})\|$,
it follows that
\[
\dist(p,\cJ \otimes B(\cE)) = \|(P_{\cG^\perp} \otimes I_{\cE}) M_p (P_{\cG^\perp} \otimes I_{\cE})\| \leq 1,
\]
as required.
\end{proof}

\section{Quotients of $\cH^2_{d}$ and $H^\infty(\fB_d)$} \label{sec:quotients}

For a family $S$ of nc functions on $\fB_d$, we define
\[
V_{\fB_d}(S) = \{Z \in \fB_d :  f(Z) = 0 \,\, \textrm{ for all } f \in S\}.
\]
Given an nc set $\mathfrak{S}$, we define
\[
I(\mathfrak{S}) = \{p \in \F_d :   p(Z) = 0 \,\, \textrm{ for all } Z \in \mathfrak{S}\}.
\]
Recall that we previously defined
\[
\cJ_\mathfrak{S} = \{f \in H^\infty(\fB_d) :   f(Z) = 0 \,\, \textrm{ for all } Z \in \mathfrak{S}\}.
\]

Henceforth, when we speak of an nc variety, we shall mean a set $\fV \subseteq \fB_d$ which is the joint zero set of a family of multipliers in $\mlt \cH^2_{d} = H^\infty(\fB_d)$.
By Lemma \ref{lem:HvsMltVarieities}, this is the same as the joint zero set of a family of functions in $\cH^2_{d}$.
Letting $\cJ_\fV \triangleleft H^\infty(\fB_d)$ denote the ideal of bounded nc functions vanishing on $\fV$, we have $\fV = V_{\fB_d}(\cJ_\fV)$.
The ideal $\cJ_\fV$ is the kernel of the restriction map $f \mapsto f\big|_\fV$.
We will write $P_\fV$ for the orthogonal projection $P_\fV : \cH^2_{d} \to \cH_\fV$.

\begin{lem} \label{lem:quotient_mult}
For every $f \in \mlt \cH^2_{d}$, the restriction $f\big|_\fV \in \mlt \cH_{\fV}$.
Moreover,
\[
M_{f\big|_\fV} = P_{\fV} M_f = P_{\fV} M_f P_{\fV} .
\]
\end{lem}
\begin{proof}
The space $\cH_\fV$ is invariant for $M_f^*$.
Indeed, by Lemma \ref{lem:adjoint} the action of $M_f^* \big|_{\cH_\fV}$ on kernel functions is $K_{W,v,y} \mapsto K_{W,v, f(W)^* y}$.
Using Lemma \ref{lem:operator}, we see that $f\big|_\fV$ is a multiplier and that $M_{f\big|_\fV}^* = M_f^*\big|_{\cH_\fV}$.
\end{proof}

\begin{thm} \label{thm:quotient_mult}
Let $\fV \subseteq \fB_d$ be an nc variety.
The map $f \mapsto f\big|_\mathfrak{V}$ is a completely contractive and surjective homomorphism, which induces a completely isometric isomorphism $\mlt \cH^2_{d} / \cJ_\fV \to \mlt \cH_\fV$. 
In particular, $\mlt\cH_\mathfrak{V} = \mlt \cH^2_{d}\big|_\mathfrak{V}$, and for every $g \in \mlt \cH_\mathfrak{V}$ there exists $f \in H^\infty (\fB_d)$, such that $f\big|_\mathfrak{V} = g$ and $\|f\|_\infty = \|g\|_{\mlt\cH_\mathfrak{V}}$.
\end{thm}
\begin{proof}
The map $f \mapsto f\big|_\fV$ is completely contractive by Lemma \ref{lem:quotient_mult}, and it is readily seen that the kernel is $\cJ_\fV$. 
To see that this map is surjective and induces a completely isometric isomorphism $\mlt \cH^2_{d} / \cJ_\fV \to \mlt \cH_\fV$, we use the complete Pick property.

Let $g$ be a matrix valued multiplier in $\mlt (\cH_\fV \otimes \C^k)$.
Suppose without loss of generality that $\|g\|_{\mlt \cH_\fV \otimes \C^k} \leq 1$.
Then we have that
 \[
{K}^{dBR}(Z,W)(P) = K(Z,W)(P) \otimes I_k - g(Z) \left[K(Z,W)(P)\otimes I_k \right] g(W)^*
\]
is a cp nc kernel on $\fV$ (recall Remark \ref{rem:dbR}).
By Theorem \ref{thm:complete_Pick}, $g$ extends to a multiplier $f \in \mlt (\cH^2_{d} \otimes \C^k)$ of norm less than or equal to $1$.
In particular, the restriction map $f \mapsto f\big|_\fV$ is surjective.
The restriction map therefore induces a completely contractive linear isomorphism $\Phi: \mlt \cH^2_{d} / \cJ_\fV \to \mlt \cH_\fV$.
$\Phi$ is in fact a complete isometry, since given $g \in M_k(\mlt \cH_\fV)$ we just observed that one can find a $f \in M_k(\mlt \cH^2_{d})$ with $\|f\| = \|g\|$.
This implies that $\|(\Phi_n)^{-1}(g)\| = \|g\|$.
\end{proof}

Recall the Bunce--Frazho--Popescu dilation theorem \cite{Bunce,Frazho,Popescu89}, which says that if $T = (T_1, \ldots, T_d)$ is a pure row contraction on a Hilbert space $H$, then there is a Hilbert space $\cE$ of dimension $d$, and an isometry $V: H \to \cF(\cE) \otimes \cE$ such that $V H$ is a co-invariant subspace for the free shift $L \otimes I_\cE$, and such that
\[
T_i = V^* (L_i \otimes I_\cE) V \quad , \quad i=1,2, \ldots, d.
\]
Identifying $H^\infty(\fB_d)$ with $\cL_d$, this gives rise to a functional calculus: for every pure row contraction $T$, there is a weak-operator continuous, unital, completely contractive homomorphism
\[
\Phi_T : H^\infty(\fB_d) \to \overline{\alg}^{\textup{wot}}(T), 
\]
given by $\Phi_T(f) = V^* (f(L) \otimes I_\cE) V$ (where $f(L)$ is the image of $f$ in $\cL_d$ under the isomorphism $H^\infty(\fB_d)  \cong \cL_d$). 
If $T$ is a strict contraction ($\|T\| := \|T\|_{\textup{row}}<1$) then it is not hard to see that $\Phi_T$ becomes the evaluation at $T$, that is
\[
\Phi_T\left(\sum_{k \in \W_d} a_k z^k\right) = \sum_{k \in \W_d} a_k T^k.
\]
We obtain a functional calculus for multiplier algebras on nc varieties, versions of which were observed in \cite{Popescu06a,ShalitSolel}.
\begin{cor}\label{cor:functional_calculus}
Let $\fV \subseteq \fB_d$ be an nc variety.
Let $T$ be a pure row contraction.
If $\cJ_\fV \subseteq \ker \Phi_T$ (in particular, if $T \in \fV$), then there is a weak-operator continuous, unital completely contractive homomorphism from $\mlt \cH_\fV$ to $\overline{\alg}^{\textup{wot}}(T)$ mapping $M_z$ to $T$.
\end{cor}
\begin{proof}
By \cite[2.3.5]{BlecherLeMerdy}, the map $\Phi_T: H^\infty(\fB_d) \to \overline{\alg}^{\textup{wot}}(T)$ induces a unital completely contractive homomorphism $\Psi_T: H^\infty(\fB_d) / \cJ_\fV \to \overline{\alg}^{\textup{wot}}(T)$ satisfying $\Psi_T \circ q_{\cJ_\fV} = \Phi_T$ where $q_{\cJ_\fV}$ is the natural quotient map.
Since by Theorem \ref{thm:quotient_mult} and the identification in Corollary \ref{cor:unitary_equiv}, $H^\infty(\fB_d) / \cJ_\fV$ is completely isometrically isomorphic to $\mlt \cH_\fV$, we obtain the desired map.
Therefore, it remains to show that this map is weak-operator continuous.
Due to Davidson--Pitts \cite{DavPitts1} and Corollary \ref{cor:unitary_equiv},
the weak-$*$ topology and the weak-operator topology coincide in $H^\infty (\fB_d)$. In particular, the weak-$*$ closed ideal $\cJ_\fV$ is also $\textsc{wot-}$closed.
As $\Phi_T$ is $\textsc{wot-}$continuous, the map $\Psi_T$ induced on the quotient must be $\textsc{wot-}$continuous as well.
\end{proof}


\begin{thm} \label{thm:mult_are_bounded_on_V}
Let $\fV \subseteq \fB_d$ be an nc variety.
Then $\mlt\cH_\mathfrak{V} = H^\infty(\mathfrak{V})$ completely isometrically.
\end{thm}

To prove the above theorem, we need to recall a result of Ball--Marx--Vinnikov \cite[Theorem 3.1]{BMV15b} in a very specific case.
\begin{prop}\label{prop:bmv}
Let $f\in M_n(H^\infty(\mathfrak{V}))$ satisfying $\|f\|_{\infty} \leq 1$.
Then $f$ has an extension $\tilde f \in M_n(H^\infty(\fB_d))$ with $\|f\|_{\infty} \leq 1$.
\end{prop}

\begin{proof}
This follows from the implication $(1') \implies (1)$ in \cite[Theorem 3.1]{BMV15b}. With the notations of the latter reference, if we are given a set of points $\Omega \subseteq \mathbb D_Q$ (here $\mathbb D_Q = \{ X \in \M_d : \|Q(X)\|<1\}$ where $Q$ is some matrix valued nc polynomial) and two nc functions
$a\in \cT(\Omega', \cL(\cY,\cE)_{\rm{nc}})$, $b \in \cT(\Omega', \cL(\cU,\cE)_{\rm{nc}})$,
where $\Omega'$ is the $\mathbb D_Q$-relative full nc envelope of $\Omega$, then the inequality $a(Z)a(Z)^*-b(Z)b(Z)^* \geq 0$ for all $Z \in \Omega'$ implies there exists an $S:\mathbb D_Q \to \cL(\cU,\cY)_{\rm{nc}}$ in the nc Schur--Agler class satisfying $a(Z)S(Z)=b(Z)$ for all $Z \in \Omega$.

Letting $Q(z):= \begin{bmatrix} z_1&z_2&\dots&z_d \end{bmatrix}$ and $\Omega:=\fV$ we obtain that $\mathbb D_Q=\fB_d$ and $\Omega'=\Omega$. Letting $\cY=\cU=\cE=\bC^n$, $a(Z)=I_n$ for all $Z\in \fV$, and  $b(Z)=f(Z)$ for all $Z \in \fV$, we note that
\[
a(Z)a(Z)^*-b(Z)b(Z)^*=1-f(Z)f(Z)^* \geq 0 \,\,, \,\, \textrm{ for all } Z\in \fV.
\]
Thus, by the above result, there exists an nc Schur--Agler class function $\tilde f:=S: \fB_d \to \sqcup_{m=1}^{\infty} (M_n^d)^{m \times m}$ satisfying $\tilde f(Z) = f(Z)$ for all $Z \in \fV$.
The fact that $\tilde f$ is a Schur--Agler class function implies that $\tilde f \in M_n(H^\infty(\fB_d))$ and $\|\tilde f \|_\infty \leq 1$.
\end{proof}

\begin{proof}[Proof of Theorem \ref{thm:mult_are_bounded_on_V}]

Suppose $f\in M_n(H^\infty(\mathfrak{V}))$ and $\|f\|_{\infty} \leq 1$. By Proposition \ref{prop:bmv},
there exists $\tilde f \in M_n(H^\infty(\mathfrak{B}_d))$
such that $\tilde f|_\fV=f$.
Theorem \ref{thm:quotient_mult} implies that $f \in M_n(\mlt \cH_{\fV})$ and $\|f\|_{\mlt \cH_\fV} \leq 1$.
The converse direction follows immediately by Theorem \ref{thm:quotient_mult}.
\end{proof}

In Section \ref{sec:homog_case}, we will give an alternative proof of Theorem \ref{thm:mult_are_bounded_on_V} in the case of homogeneous nc varieties (that proof will not require invoking the machinery of \cite{BMV15b}).

\section{Isomorphisms and isometric isomorphisms} \label{sec:isomorphisms}

\subsection{The space of finite dimensional representations}

The finite dimensional, unital, completely contractive representations of $H^\infty(\fB_d) = \cL_d$ have been worked out in \cite[Section 3]{DavPitts2} (for the case $d = \infty$, one should also see the erratum of that paper).
Let us denote by $\Rep_{k}(\cA)$ the space of all unital completely contractive representations of an operator algebra $\cA$ on $\C^k$.

\begin{thm}[Davidson--Pitts \cite{DavPitts2}] \label{thm:DavPitts_reps}
For all $d \in \bN \cup \{\infty\}$ and $k \in \bN$, there is a natural continuous projection $\pi_{d,k}$ of $\Rep_{k}(H^\infty(\fB_d))$ onto the closed unit ball $\overline{\fB_d(k)}$, given by
\[
\pi_{d,k}(\Phi) = (\Phi(z_1), \ldots, \Phi(z_d)).
\]
For every $T \in \fB_d(k)$, there is a unique weak-operator continuous representation $\Phi_T \in \pi_{d,k}^{-1}(T)$, given by Popescu's functional calculus \cite{Popescu95}.
If $d < \infty$ and $T \in \fB_d(k)$, then $\pi_{d,k}^{-1}(T)$ is the singleton $\{\Phi_T\}$.
\end{thm}

\begin{remark}\label{rem:MS_cnc}
By Corollary 5.7 in \cite{MS11}, a tuple $T \in \ol{\fB}_d$ gives rise to a weak-$*$ continuous unital representation, mapping $z_i$ to $T_i$, if and only if $T$ is completely non-coisometric (c.n.c.), meaning that there is no vector $v$ in the space on which $T$ acts such that $\|\sum_{|\alpha|=n}(T^{\alpha})^* v \| = \|v\|$ for all $n$.
If $\|T\|<1$, then clearly $T$ is c.n.c.
Additionally, if $\|T\| = 1$ and $T$ is {\em pure} (meaning that $\sum_{|\alpha|=n}(T^{\alpha})^*$ converges $\textsc{sot}$ to $0$ as $n \to \infty$) then $T$ is also c.n.c.
In fact, since the unit sphere of a finite dimensional space is compact, it follows that $T \in \ol{\fB}_d$ is c.n.c. if and only if it is pure.
Thus we define for every pure $T \in \ol{\fB}_d$ the unique weak-$*$ continuous representation $\Phi_T$ that maps $z_i$ to $T_i$.
For every $f\in H^\infty(\fB_d)$ and every pure $T$, one can evaluate $f$ at the point $T$, by $f(T) = \Phi_T(f)$.
\end{remark}

For $\fV \subseteq \fB_d$, we denote by $\ol{\fV}^{p}$ the set of all pure $T$ such that $\cJ_\fV \subseteq \ker \Phi_T$.
This allows us to get a handle on the finite dimensional representations of $H^\infty(\fV)$.

\begin{thm} \label{thm:finite_dim_reps}
Let $\fV \subseteq \fB_d$ be an nc variety.
For every $k \in \bN$, there is a natural continuous projection $\pi_{d,k}$ of $\Rep_{k}(H^\infty(\fV))$ into the closed unit ball $\overline{\fB_d(k)}$, given by
\[
\pi_{d,k}(\Phi) = (\Phi(z_1), \ldots, \Phi(z_d)).
\]
For every $T \in \ol{\fV}^{p}$, there is a unique weak-$*$ continuous representation $\Phi_T \in \pi_{d,k}^{-1}(T)$, and these are the only weak-$*$ continuous elements in $\Rep_{k}(H^\infty(\fV))$.
If $d < \infty$ and $T \in \fV$, then $\pi_{d,k}^{-1}(T)$ is the singleton $\{\Phi_T\}$.
Moreover, if $d < \infty$, then
\[
\pi_{d,k}(\Rep_{k}(H^\infty(\fV)))\cap \fB_d(k) = \fV(k) .
\]
\end{thm}
\begin{proof}
Every representation $\Phi \in \Rep_{k}(H^\infty(\fV))$ can be thought of as an element of the space $\Rep_{k}(H^\infty(\fB_d))$ as well: for each $f \in H^{\infty}(\fB_d)$ the map $f \mapsto \Phi(f|_{\fV})$ is indeed a unital completely contractive representation of $\fB_d$ on $\mathbb C^k$.
Thus $\pi_{d,k}$ from Theorem \ref{thm:DavPitts_reps} maps $\Rep_{k}(H^\infty(\fV))$ into the closed unit ball $\overline{\fB_d(k)}$.

Now let $T\in \ol{\fV}^{p}$. By Remark \ref{rem:MS_cnc}, as $T$ is pure, there exists
a weak-$*$ continuous representation $\Phi_T$ such that $\pi_{d,k}(\Phi_T)=T$. Since these are the unique weak-$*$ continuous elements of $\Rep_{k}(H^\infty(\fB_d))$, they are the unique weak-$*$ continuous elements of $\Rep_{k}(H^\infty(\fV))$ as well.

The penultimate assertion follows from the last statement of Theorem \ref{thm:DavPitts_reps}.
As for the last assertion, if $\Phi \in \Rep_{k}(H^\infty(\fV))$ and $T:=\pi_{d,k}(\Phi) \in \fB(k)$, then by last statement of Theorem \ref{thm:DavPitts_reps}, $\Phi_T(f)=\Phi(f|_\fV)$ for all $f \in H^{\infty}(\fB_d)$. In particular, for every $f \in \mathcal J_\fV$ we have that $f(T)=\Phi_T(f)=\Phi(0)=0$.
This, together with fact that $T \in \fB_d(k)$, implies that $T\in \fV(k)$.
\end{proof}


\subsection{Completely contractive homomorphisms}

Let $\fV \subseteq \fB_d$ and $\fW \subseteq \fB_{e}$ be nc varieties.
Every completely contractive unital homomorphism  $\alpha : H^\infty(\fV) \to H^\infty(\fW)$ induces a graded map
\[
\alpha^* : \sqcup_k \Rep_{k}(H^\infty(\fW)) \to \sqcup_k \Rep_{k}(H^\infty(\fV))
\]
by $\alpha^* (\Phi) = \Phi \circ \alpha$.
If $\alpha$ is weak-$*$ continuous then $\alpha^*$ maps weak-$*$ continuous representations to weak-$*$ continuous representations.
We obtain an nc map $G_\alpha : \ol{\fW}^p \to \overline{\fB}_d$ given by
\[
G_\alpha(W) = \pi_{d,k} (\alpha^*(\Phi_W)) \quad, \quad W \in \ol{\fW}^p(k).
\]

\begin{prop}\label{prop:cc_homo}
Let $\fV \subseteq \fB_d$ and $\fW \subseteq \fB_{e}$ be nc varieties, and let $\alpha : H^\infty(\fV) \to H^\infty(\fW)$ be a completely contractive unital homomorphism.
Then there exists an nc map $G: \fB_e \to \ol{\fB}_d$ such that $G \big|_{\fW} = G_\alpha \big|_{\fW}$.
If $\alpha$ is also assumed to be weak-$*$ continuous, then $G$ maps $\fW$ into $\ol{\fV}^p$ and  implements $\alpha$:
\[
\alpha(f) = f \circ G \quad, \quad f \in H^\infty(\fV).
\]
\end{prop}
\begin{proof}
For $i=1, \ldots, d$, let us define $g_i = \alpha(z_i)$, where $z_i$ denotes the nc coordinate function.
Then $g_i \in H^\infty(\fW)$ for all $i$.
We define the nc map $G: \fW \to \M_d$ by
\[
G(W) = (g_1(W), \ldots, g_d(W)) \quad, \quad W \in \fW.
\]
Then for every $i=1, \ldots, d$, and every $W \in \fW(k)$,
\[
\alpha^*(\Phi_W) (z_i) = \Phi_W (\alpha(z_i)) = g_i(W),
\]
which shows that $G(W) = G_\alpha(W)$ for all $W \in \fW$.
In particular, this means that $G(W) \in \ol{\fB}_d$ for all $W \in \fW$, so $G \in H^\infty(\fW) \otimes \C^d$ has norm less than or equal to $1$.
By Theorem \ref{thm:quotient_mult}, we can therefore extend $G$ from $\fW$ to $\fB_e$ to obtain a function (which we still call $G$) in $H^\infty(\fB_e) \otimes \C^d$ with the same norm.

Now, if $\alpha$ is weak-$*$ continuous, then $\alpha^*$ preserves weak-$*$ continuous representations.
Thus, for every $W \in \fW$, $\alpha^*(\Phi_W)$ is determined completely by $\pi_{d,k}(\alpha^*(\Phi_W))$, therefore it is $\Phi_{G(W)}$.
So
\[
\alpha(f) (W) = \Phi_W \circ \alpha(f) = \Phi_{G(W)}(f) = f \circ G(W).
\]
\end{proof}

\subsection{Completely isometric isomorphisms}

An {\em automorphism} of the nc ball $\fB_d$ is an nc holomorphic map $\varphi : \fB_d \to \fB_d$ with an nc holomorphic inverse.
We let $\operatorname{Aut}(\fB_d)$ denote the group of automorphisms of $\fB_d$.
If $\fV, \fW \subseteq \fB_d$, and $\fW = \varphi(\fV)$ for some $\varphi \in \operatorname{Aut}(\fB_d)$, then we say that {\em $\fV$ and $\fW$ are conformally equivalent}.

\begin{prop}\label{prop:autballs}
Every $\varphi \in \operatorname{Aut}(\fB_d)$ is determined uniquely by its restriction to $\fB_d(1) = \bB_d$.
In particular, $\operatorname{Aut}(\fB_d) \cong \operatorname{Aut}(\bB_d)$.
\end{prop}

\begin{proof}
This automorphism group has been touched upon several times in the literature (e.g., \cite[Theorem 4.11]{DavPitts2}, \cite{Harris}, \cite{McT16}, \cite{MS08} or \cite[Theorem 2.8]{Popescu10}). In \cite{Popescu10}, Popescu proved that $\operatorname{Aut}(B(\mathcal{X})_1^d) \cong \operatorname{Aut}(\bB_d)$. In a way similar to the proof of Theorem \ref{thm:HinftyVSPopescu} one can check that every function in Popescu's $\operatorname{Aut}(B(\mathcal{X})_1^d)$ corresponds naturally to a function in $\operatorname{Aut}(\fB_d)$ and vice versa. For $d<\infty$, one may also apply Theorems 7 and 13 from \cite{McT16} to obtain the result (it seems that their argument can be adapted to the case $d=\infty$ as well).

Let us give an alternative proof in the case $d<\infty$ that is based on the well known structure of the automorphism group of matrix balls.
Let $\fB_d(n)$ be the matrix unit ball in $M_n(\C)^d$. 
Recall from \cite[pp. 273]{Satake80} that $\fB_d(n)$ is a bounded symmetric domain, and its holomorphic automorphisms are given by:
\[
\varphi(Z) = (AZ + B)(C Z + D)^{-1} ,
\]
where the matrix $T = \begin{pmatrix} A & B \\ C & D \end{pmatrix}$ belongs to $\operatorname{SU}(n,dn)$.
Here we think of $Z$ as an $n \times dn$ matrix $Z = (Z_1, \ldots, Z_d)$, and we have that $A \in M_n(\C)$, $B \in M_{n,dn}(\C)$, $C \in M_{dn,n}(\C)$ and finally $D \in M_{dn}(\C)$.

Now assume that an automorphism $\varphi \in \operatorname{Aut}(\fB_d(n))$ arises as the restriction to $\fB_d(n)$ of an nc map.
Then, in particular, for every $U \in U_n$ we have
\[
U \varphi(U^*Z U^{\oplus d}) U^{(\oplus d) *} =\varphi(Z).
\]
In other words, $\varphi$ is a fixed point for the action of $U_n$ on $\operatorname{Aut}(\fB_d(n))$.
Writing it out explicitly we get that:
\begin{multline*}
U \varphi(U^*Z U^{\oplus d}) U^{(\oplus d) *} = U (A U^* Z U^{\oplus d} + B)( C U^* Z U^{\oplus d} + D)^{-1} U^{(\oplus d)*} = \\
(U A U^* Z + U B U^{(\oplus d) *})U^{\oplus d} ( U^{\oplus d} C U^* Z U^{\oplus d} + U^{\oplus d} D)^{-1} = \\
(U A U^* Z + U B U^{(\oplus d) *})( U^{\oplus d} C U^* Z  + U^{\oplus d} DU^{(\oplus d) *})^{-1}
\end{multline*}
Hence for every $U \in U_n$, the matrix $T$ and the matrix $U^{\oplus (d+1)} T U^{ (\oplus (d+1)) *}$ induce the same holomorphic automorphism on $\fB_d(n)$.
The holomorphic automorphisms of $\fB_d(n)$ are isomorphic to $\operatorname{SU}(n,dn)/Z(\operatorname{SU}(n,dn))$.
This implies that for every $U \in U_n$, the matrix $T^{-1} U^{\oplus (d+1)} T U^{ (\oplus (d+1)) *}$ is in the center of $\operatorname{SU}(n,dn)$.
Now note that the map 
\[
U \mapsto T^{-1} U^{\oplus (d+1)} T U^{ (\oplus (d+1)) *}
\]
is continuous on $U_n$.  Since $U_n$ is connected and the center is finite, we see that the map is constant and thus $T = U^{\oplus (d+1)} T U^{ (\oplus (d+1)) *}$  for every $U \in U_n$.
Since $U_n$ is Zariski dense in $\operatorname{GL}_n$ we can conclude that $T = S^{\oplus (d+1)} T (S^{\oplus (d+1)})^{-1}$, for every $S \in \operatorname{GL}_n$.
Thus $A = a I_n$ is a scalar matrix and if we write $B = (B_1 \cdots B_d)$, then each $B_j$ is scalar, i.e., there exists a row vector $v \in M_{1,d}(\C)$, such that $B = v \otimes I_n$ and similarly $C = w \otimes I_n$ and $D = X \otimes I_n$, where $w \in M_{d,1}(\C)$ and $X \in M_d(\C)$.
Note that the assumption that $T \in \operatorname{SU}(n,dn)$ implies that $\begin{pmatrix} a & v \\ w & X \end{pmatrix} \in \operatorname{SU}(1,d)$ and thus $\varphi$ is induced by an automorphism of the commutative ball $\fB_d(1) = \bB_d$.
\end{proof}

\begin{prop}\label{prop:extend_auto}
Every automorphism $\phi \in \operatorname{Aut}(\fB_d)$ extends to an automorphism
$\Phi \in \operatorname{Aut}(\fB_{d+e})$.
\end{prop}
\begin{proof}
This follows from Proposition \ref{prop:autballs} together with the commutative result \cite[Section 2.2.8]{Rudin}.
\end{proof}

The following is a version of Cartan's uniqueness theorem in the nc setting.
This too, has been considered, from different perspectives (e.g., \cite{McT16} or \cite{Popescu10}).

Let $f \colon U \to \M_d$ be an nc holomorphic function, where $U$ is an nc domain.
We will write $f = (f_1,\ldots,f_d)$ and for every point $Y \in U$ we will write the first order nc derivative of $f$ at $Y$ as:
\[
\Delta f(Y,Y)(Z) = \left(\Delta f_1(Y,Y)(Z), \ldots, \Delta f_d(Y,Y)(Z)\right).
\]
Similarly for higher order nc derivatives (for the notion of nc derivatives, see \cite{KVBook}).
\begin{thm} \label{thm:nc_cartan}
Let $G \subset \M_d$ be a uniformly bounded nc domain. 
Let $f \colon G \to G$ be an nc holomorphic function.
Assume that $Y \in G$ is such that $f(Y) = Y$ and $\Delta f(Y,Y) = I$. 
Then $f(Z) = Z$ for every $Z \in G$.
\end{thm}
\begin{proof}
Since $f$ is an analytic function of bounded domains in $M_s(\C)^d$ and the derivative of $f$ at $Y$ can be identified with $\Delta f (Y,Y)$, we can apply the classical Cartan's uniqueness theorem \cite[Theorem 2.1.1]{Rudin} to get that $f$ is the identity on level $s$.
Now writing out the Taylor--Taylor power series for $f$ around $Y$ (see \cite{KVBook}), we get that, in fact, $f$ is the identity on a noncommutative ball with center $Y$.
Using the uniqueness theorem we can conclude that $f$ is the identity on every level that is multiple of $s$. Using the fact that $f$ is an nc function and $G$ is an nc domain, if $X \in G \cap M_n(\C)^d$, then
\[
X^{\oplus s} = f(X^{\oplus s}) = f(X)^{\oplus s}.
\]
Conclude that $f(X) = X$ and we are done.
\end{proof}

By \cite[Theorem 4.1]{DavPitts2} and the identification in Corollary \ref{cor:unitary_equiv}, there exists a homomorphism $\tau:\operatorname{Aut}(H^\infty(\fB_d)) \to \operatorname{Aut}(\bB_d)$ that has a continuous section $\kappa$ carrying $\operatorname{Aut}(\bB_d)$ to the subgroup $\operatorname{Aut}_u(H^\infty(\fB_d))$ of unitarily implemented automorphisms of $H^\infty(\fB_d)$. That is, for every $\phi \in \operatorname{Aut}(\bB_d)$, there is a unitary $U_\phi \in B(\cF(\C^d))$, such that $\kappa(\phi)$ is a completely isometric automorphism of $H^\infty(\fB_d)$ of the form
\[
\kappa(\phi)(M_f)=U_\phi^*M_fU_\phi \quad\text {for all }f\in H^\infty(\fB_d).
\]
We record this:

\begin{prop}[Davdison--Pitts \cite{DavPitts2}]
Every $\phi \in \operatorname{Aut}(\bB_d)$ gives rise to a unitary on $\cH^2_d$ that implements a completely isometric automorphism of $H^\infty(\fB_d)$.
\end{prop}

We easily obtain a sufficient condition for two algebras $H^\infty(\fV)$ and $H^\infty(\fW)$ to be completely isometrically isomorphic.
\begin{prop}\label{prop:biholo_auto}
Let $\fV$ and $\fW$ be two nc varieties in $\fB_d$ and $\fB_e$, respectively.
Suppose that there exists nc holomorphic functions $G : \fB_e \to \fB_d$ and $H : \fB_d \to \fB_e$ such that  $G\circ H \big|_\fV = {\bf id}_{\fV}$ and $H \circ G \big|_\fW = {\bf id}_{\fW}$.
Then $H^\infty(\fV)$ and $H^\infty(\fW)$ are completely isometrically isomorphic.

In particular, if there exists $\varphi \in \operatorname{Aut}(\fB_d)$ such that $\varphi(\fW) = \fV$, then $H^\infty(\fV)$ and $H^\infty(\fW)$ are completely isometrically isomorphic.
\end{prop}
\begin{proof}
It is easy to check that the map $\alpha : H^\infty(\fV) \to H^\infty(\fW)$ given by $\alpha(f) = f \circ G$ is a completely isometric isomorphism.
\end{proof}

\begin{remark}
We conjecture that in the case $d=e$, if there exist $G$ and $H$ as in the first part of the theorem, then there exists an automorphism $\varphi$ as in the second part of the theorem.
We have not been able to prove this.
In the next section, we will prove that if the varieties under consideration are homogeneous, then this is indeed true.
\end{remark}

We now generalize the maximum modulus principle for nc holomorphic functions mapping domains which are invariant under unitary conjugation and containing $0$ into $\mathbb M_d$.

\begin{lem}[maximum principle]\label{lem:maximum_principle}
Let $\Omega \subseteq \mathbb M_e$ be an nc domain invariant under unitary conjugation and containing the origin, and let $G : \Omega \to \mathbb M_d$ be an nc holomorphic function. Suppose there exists $W_0\in \Omega$ such that 
\[
\|G(W_0)\|=\max_{W\in \Omega}\|G(W)\|
\]
where $\|\cdot\|$ is the row operator norm on $\mathbb M_d$. Then $G$ is constant.
\end{lem}
\begin{proof}
 Let $G : \Omega \to \mathbb M_d$ be an nc holomorphic function such that
$\|G(W_0)\|=\max_{W\in \Omega}\|G(W)\|$ for some $W_0\in \Omega(n)$.
We may assume that $\|G(W_0)\|=1$, so that $G(\Omega) \subseteq \ol{\fB}_d$.
Let $\varphi$ be a (unitary) automorphism of the $d$-dimensional ball mapping $G(0)$ to $(\|G(0)\|,0,\dots,0)$, and set $G^\varphi:=\varphi \circ G$. As $G(W_0)\in \partial \fB_d(n)$, we have that $G^\varphi(W_0) \in \partial \fB_d(n)$ as well. So there exists a unit vector $u \in \mathbb C^{nd}$  such that $v:=G^\varphi(W_0)u$ is a unit vector in $\mathbb C^n$. 
Now consider the function $\psi : \Omega(n) \to \ol{\bD}$ given by $\psi(W) = \langle G^\varphi(W)u,v\rangle$.
Then $|\psi(W)|\leq 1$ for all $W \in \Omega(n)$, but $\psi(W_0)=1$, so by the maximum modulus principle $\psi$ is constant on $\Omega(n)$.
By the Cauchy Schwarz inequality, the function $\phi: \Omega(n) \to \C^n$ given by $\phi : W \mapsto G^\varphi(W)u$ must be constantly equal to $v$.
Write 
\[
u=\begin{bmatrix} u_1 \\ \vdots \\ u_d \end{bmatrix}\quad \text{and} \quad G^\varphi=(G^\varphi_1,\dots,G^\varphi_d).
\] 
So for every $W\in \Omega(n)$ we have $v=G^\varphi(W) u = G^\varphi(0) u = u_1$. Thus, $u_2=\dots=u_d=0$, so $G^\varphi_1(W)u_1=u_1$ for all $W \in \Omega(n)$. But $G^\varphi$ is an nc function, so for every unitary $U\in M_n$ and $W \in \Omega(n)$ we have $G^\varphi_1(W)Uu_1=UG^\varphi_1(U^*WU)u_1= Uu_1$ .
Thus $G^\varphi_1(W)=I_n$ for all $W\in \Omega(n)$ so that $G^\varphi(W)=(I_n,0,\dots,0)$, and consequently we have $G(W)=G(0)$ for all $W \in \Omega(n)$.

In fact, $G$ must be the constant $G(0)$ on {\em all} levels of $\Omega$. 
To see this, note that for each level $m$ of $\Omega$, the zero element is mapped to $G(0) \in \partial \fB_d(m)$, so the previous argument implies that $G$ must then be equal to the constant $G(0)$ on $\fB_e(m)$.
\end{proof}

\begin{thm}\label{thm:isomorphism}
Let $\fV \subseteq \fB_d$ and $\fW \subseteq \fB_{e}$ be nc varieties, and let $\alpha : H^\infty(\fV) \to H^\infty(\fW)$ be a completely isometric isomorphism. Assume that $d$ and $e$ are finite or that $\alpha$ is weak-$*$ continuous.
Then there exists an nc map $G: \fB_e \to \fB_d$ such that $G \big|_{\fW}  = G_\alpha|_{\fW}$ maps $\fW$ bijectively onto $\fV$, which implements $\alpha$ by the formula
\[
\alpha(f) = f \circ G \quad, \quad f \in H^\infty(\fV).
\]
\end{thm}
\begin{proof}
By Proposition \ref{prop:cc_homo}, there is an nc map $G=(G_1,\dots,G_d): \fB_e \to \ol{\fB}_d$ such that $G \big|_{\fW}  = G_\alpha|_{\fW}$. We first show that the injectivity of $\alpha$ implies that $G(\fB_e) \subseteq \fB_d$. Assuming the opposite, the maximum principle (Lemma \ref{lem:maximum_principle}) implies that $G$ is constant of norm $1$.
By the construction of $G$, we have that $G_i = \alpha(z_i)$ where $z_i$ denotes the coordinate nc function $z \mapsto z_i$ restricted to $\fW$. Thus, $z_i - G_i(0) \in \ker \alpha$. As $\alpha$ is injective, we obtain that $z_i$ is the constant $G_i(0)$ for all $ i =1,2,\dots$. Since $\|G(0)\|=1$, we conclude that $\fW$ must be empty, which is of course a contradiction. Thus, $G(\fB_e) \subseteq \fB_d$.

Finally, we prove that $\alpha$ is implemented by composition with $G$. 
In the case that $d$ is finite, then as $G(\fW) \subseteq \fB_d$, Theorem \ref{thm:finite_dim_reps} implies that $\pi_{d,k}^{-1}(G(W))$ is the singleton $\{\Phi_{G(W)}\}$ for every $k\in \mathbb N$ and $W \in \fW(k)$. Thus, $\alpha^*(\Phi_W)=\Phi_{G(W)}$ for all $W \in \fW$.
In the case that $\alpha$ is weak-$*$ continuous, then $\alpha$ preserves weak-$*$ continuous representations. Since evaluations by elements of $\ol{\fV}^p$ are the only weak-$*$ continuous elements of $\Rep_{k}(H^\infty(\fV))$ and since $G(\fW) \subseteq \fB_d$, we must have that $\alpha^*(\Phi_W)=\Phi_{G(W)}$ for all $W \in \fW$. In any case we conclude that $G(\fW) \subseteq \fV$ and
\[
\alpha(f)(W) = \Phi_W(\alpha(f))=\alpha^*(\Phi_W)(f)=\Phi_{G(W)}(f)=(f \circ G)(W)
\]
for every $f\in H^\infty(\fV)$ and $W \in \fW$. Replacing the roles of $\alpha$ and $\alpha^{-1}$ (which must be weak-$*$ continuous if $\alpha$ is weak-$*$ continuous) yields an nc map $H:\fB_d \to \fB_e$ mapping $\fW$ into $\fV$ such that   $G\circ H \big|_\fV = {\bf id}_{\fV}$ and $H \circ G \big|_\fW = {\bf id}_{\fW}$.
\end{proof}

Theorem \ref{thm:isomorphism} shows that in the case where $d,e \in \bN$, every completely isometric isomorphism is implemented by a composition with a biholomorphism. 
If follows easily (using Lemma \ref{prop:bound_wot_conv_is_point_conv}) that such an isomorphism is weak-$*$ continuous. 
We record this:

\begin{cor}\label{cor:isomorphism_auto_weakstar}
Let $\fV \subseteq \fB_d$ and $\fW \subseteq \fB_{e}$ be nc varieties with $d,e \in \bN$.
Then every completely isometric isomorphism from $H^\infty(\fV)$ onto $H^\infty(\fW)$ is automatically weak-$*$ continuous.
\end{cor}

The following corollary to Theorem \ref{thm:isomorphism} should be read with the previous one in mind. 

\begin{cor}\label{cor:isomorphism}
Let $\fV \subseteq \fB_d$ and $\fW \subseteq \fB_{e}$ be nc varieties.
Then $H^\infty(\fV)$ and $H^\infty(\fW)$ are completely isometrically isomorphic via a weak-$*$ continuous map, if and only if $\fV$ and $\fW$ are biholomorphically equivalent, in the sense that there exists an nc holomorphic map $G: \fB_e \to \fB_d$ and an nc holomorphic map $H: \fB_d \to \fB_e$ such that $G\big|_\fW = (H\big|_\fV)^{-1}$. 
\end{cor}

\begin{remark}
The above corollary is our noncommutative generalization of \cite[Theorem 4.4]{DRS15}.
Note that something remains to be desired, since, unlike in the commutative setting, we are not able to show that $G$ and $H$ can be chosen to be automorphisms of the nc ball.
In Section \ref{sec:homog_isom} we will remedy this, under the assumption that the varieties under consideration are homogeneous.
\end{remark}

\section{Homogeneous varieties and a homogeneous Nullstellensatz} \label{sec:homog_case}

In this section, unless stated otherwise, we always assume $d<\infty$.
An ideal $I \triangleleft \F_d$ or $I \triangleleft H^\infty(\fB_d)$ is said to be {\em homogeneous} if for every $f \in I$, every homogeneous component $f_n$ of $f$ is in $I$.

\begin{prop}\label{prop:homogeneous}
$I$ is a homogeneous ideal in $\F_d$ if and only if for every polynomial $p\in I$ and for all $t \in \bC$ the polynomial $z \mapsto p(tz)$ is also in $I$.
Likewise, $J$ is a homogeneous ideal in $H^\infty(\fB_d)$ if and only if for every function $f\in J$ and for all $t \in \ol{\bD}$, the function $z \mapsto f(tz)$ is also in $J$.
\end{prop}
\begin{proof}
Omitted (see \cite[Proposition 6.3]{DRS11} for a similar result in the commutative setting).
\end{proof}

A subset $\fS \subseteq \fB_d$ is said to be {\em homogeneous} if $t\fS \subseteq \fS$ for all $t \in \ol{\bD}$.
A variety $\fV \subseteq \fB_d$ that is homogeneous is called a {\em homogeneous variety}.

\begin{prop}\label{prop:homogeneous2}
If a set $\fS \subseteq \fB_d$ is homogeneous, then both $I(\fS)$ and $\cJ_\fS$ are homogeneous ideals.
If an ideal $I$ in $\F_d$ or $H^\infty(\fB_d)$ is homogeneous, then $V_{\fB_d}(I)$ is a homogeneous variety.
\end{prop}

\begin{proof}
Clear from the definitions and Proposition \ref{prop:homogeneous}.
\end{proof}

\begin{thm}\label{thm:null_poly}
Let $J \triangleleft \F_d$ be a homogeneous ideal.
Then
\[I(V_{\fB_d}(J)) = J. \]
\end{thm}

\begin{proof}
By definition, $J \subseteq I(V_{\fB_d}(J))$.
For the converse, note that by Propositions \ref{prop:homogeneous} and \ref{prop:homogeneous2}, $I(V_{\fB_d}(J))$ is also a homogeneous ideal.
Let $p \notin J$ be a homogeneous polynomial.
We will find $X \in V_{\fB_d}(J)$ such that $p(X) \neq 0$.

Identifying $J$ as a subspace of $\cH^2_d$, we consider the compression of the shift $M_z = (M_{z_1}, \ldots, M_{z_d})$ to $J^\perp = \cH^2_{d} \ominus J$:
\[
S = P_{J^\perp} M_z P_{J^\perp}.
\]
A homogeneous polynomial $q$ satisfies $q(S) = 0$ if and only if $q \in J$.
Therefore, having chosen $p \notin J$, we have $p(S) \neq 0$ \cite[Lemma 7.6]{ShalitSolel}.
Let $P_n$ denote the orthogonal projection onto the polynomials of degree less than or equal to $n$.
Since $J$ is homogeneous, $P_n$ commutes with the projection of $\cH^2_{d}$ onto $J^\perp$.
It follows that:
\begin{enumerate}
\item for every $n$ and every $q \in J$, $q(P_n S P_n) = P_n q(S) P_n = 0$; and
\item for some $n$, $p(P_n S P_n) = P_n p(S) P_n  \neq 0$.
\end{enumerate}
Letting $n$ be as in (2) above, we pick some $t \neq 0$ in the open unit disc.
Then we have that $X:=t P_n S P_n \in V_{\fB_d}(J)$ while $p(t P_n S P_n) \neq 0$.
Thus $p \notin I(V_{\fB_d}(J))$, and the proof is complete.
\end{proof}

\begin{remark}
Note that the result as stated is false for nonhomogeneous ideals, as the example $J = \langle xy - yx - 1 \rangle \triangleleft \F_2$ shows.
This example shows that the result is false if one replaces $\fB_d$ with $\M_d$, or even with $B(H)^d$ for some Hilbert space $H$.
Thus, a perfect Nullstellensatz $J = I(V(J))$ does not hold  without some further assumptions.
On the other hand, we will see below in Corollary \ref{cor:free_com_NSTZ} that a perfect free Nullstellensatz does hold in the free commutative case.
\end{remark}

\begin{remark}
For a version of the Nullstellensatz that works in the noncommutative and nonhomogeneous setting, see Amitsur's Nullstellensatz \cite{Amitsur}.
Although our result is rather simple minded in comparison, it does seem to contain independent information.
For a nonhomogeneous Nullstellensatz closer in spirit to our result, see \cite[Theorem 6.3]{HM2004}.
For a closely related homogeneous Nullstellensatz, where the variety consists of {\em operator} row contractions satisfying the relations, see \cite[Theorem 7.7]{ShalitSolel}.
\end{remark}

\begin{prop}\label{prop:hom_poly_closed}
For a weak-$*$ closed ideal $J \triangleleft H^\infty(\fB_d)$, the following are equivalent.
\begin{enumerate}
\item $J$ is homogeneous.
\item $V_{\fB_d}(J)$ is homogeneous.
\item $J$ is the weak-$*$ closure of a homogeneous ideal $I \triangleleft \F_d$.
\end{enumerate}
In fact, if $J \triangleleft H^\infty(\fB_d)$ is a weak-$*$ closed homogeneous ideal, then
\[
J = \ol{J \cap \F_d}^{w*} ,
\]
and $J \cap \F_d$ is the unique ideal in $\F_d$ with closure equal to $J$.
\end{prop}
\begin{proof}
Omitted.
\end{proof}

\begin{cor}
If $\fV$ is a homogeneous holomorphic nc variety in $\fB_d$, then $\fV$ is in fact an algebraic variety: there exists $I \triangleleft \F_d$ such that
$\fV = V_{\fB_d}(I)$.
\end{cor}
\begin{proof}
$\fV = V_{\fB_d}(\cJ_\fV)$.
Now take $I = \cJ_\fV \cap \F_d$.
\end{proof}

The significance of the following result is that it shows that, in the context of homogeneous varieties and ideals, it does not matter whether our starting point is a variety or an ideal, since there is a bijective correspondence between homogeneous weak-$*$ closed ideals in $H^\infty(\fB_d)$ and homogeneous varieties.

\begin{thm}\label{thm:homoclosedweakstar}
If $I \triangleleft \F_d$, is homogeneous, then
\[
\cJ_{V_{\fB_d}(I)} = \ol{I}^{w*}.
\]
If $J \triangleleft H^\infty(\fB_d)$ is a homogeneous weak-$*$ closed ideal, then
\[
\cJ_{V_{\fB_d}(J)} = J.
\]
\end{thm}
\begin{proof}
It suffices to prove the second assertion, since $V_{\fB_d}(I) = V_{\fB_d}(\ol{I}^{w*})$.
Now, for a homogeneous weak-$*$ closed ideal $J \triangleleft H^\infty(\fB_d)$,
\[
\cJ_{V_{\fB_d}(J)} \cap \F_d = I({V_{\fB_d}(J)}) = I({V_{\fB_d}(J \cap \F_d)}) = J \cap \F_d,
\]
where for the last equality we used Theorem \ref{thm:null_poly}.
By Proposition \ref{prop:hom_poly_closed} we find that $\cJ_{V_{\fB_d}(J)} = J$.
\end{proof}

If $\fV$ is a homogeneous variety in $\fB_d$, and $d < \infty$, then we may present an alternative proof of Theorem \ref{thm:mult_are_bounded_on_V}, and there is no need to invoke \cite[Theorem 3.1]{BMV15b}. First we need a couple of lemmas:

\begin{lem} \label{lem:circ_act}
Let $\Omega \subset \M_d$ be a open nc set containing $0$, such that for every $\lambda \in \overline{\D}$ and every $Z \in \Omega$, we have $\lambda Z \in \Omega$ as well. For $\theta \in [0,2\pi]$ we define an action $\alpha_{\theta}$ on the nc holomorphic functions on $\Omega$ via $\alpha_{\theta}(f)(Z) = f(e^{ i \theta} Z)$. Then the following is true:
\begin{itemize}
\item[(i)] The following function is $n$ homogeneous on $\Omega$:
\[
f_n(Z) = \frac{1}{2\pi} \int_0^{2\pi} \alpha_{\theta}(f)(Z) e^{- i n \theta}d\theta.
\]
Furthermore, $f_n$ is a polynomial, and if $f$ is $n$ homogeneous, then $f_n = f$ and $f_m = 0$ for $m \neq n$.

\item[(ii)] If $\Omega = \fB_d$, then $\alpha_{\theta}$ induces a unitary action of the circle on $\cH^2_d$. 
If $\fV \subseteq \fB_d$ is a homogeneous variety, then $\cH_{\fV}$ and $H^{\infty}(\fV)$ are both invariant under $\alpha_{\theta}$ for all $\theta \in [0,2\pi]$.

\end{itemize}
\end{lem}
\begin{proof}
For $0 < r < 1$ and $\tau \in [0,2\pi]$ we have:
\[
f_n(r e^{i \tau} Z) = \frac{1}{2\pi} \int_0^{2\pi} f(re^{i (\theta + \tau)} Z) e^{- i n \theta}d\theta = \frac{r^n e^{i n \tau}}{2\pi} \int_{|\lambda| = r} f(\lambda Z) \frac{d \lambda}{\lambda^n}.
\]
Since $f(\lambda Z) \frac{1}{\lambda^n}$ is an analytic matrix valued function on $\overline{\D} \setminus \{0\}$, we get that the integral is independent of $r$, hence $f_n$ is $n$-homogeneous. The second statement of $(i)$ follows from the Taylor expansion of $f$ in a neighborhood of the origin.

For $(ii)$ we note that it is immediate that $f \mapsto \alpha_{\theta}(f)$ is a unitary operator on $\cH^2_d$ and it preserves homogeneous components of the Taylor expansion around the origin. 
This induces a unitarily implemented automorphism of $H^{\infty}(\fB_d)$. 
If $\fV \subset \fB_d$ is a homogeneous variety, then its ideal is homogeneous and thus is invariant under $\alpha_{\theta}$, hence $\cH_{\fV}$ is also invariant under this action of the circle. 
We also note that each $\alpha_{\theta}$ commutes with the projection on $\cH_{\fV}$ and thus it induces a unitarily implemented automorphism on $H^{\infty}(\fV)$.

\end{proof}

\begin{lem} \label{lem:homog_func_is_poly}
Let $\fV \subset \fB_d$ be a homogeneous variety. Let $f$ be an nc holomorphic function on $\fV$, such that for every $Z \in \fV$ and every $\lambda \in \D$ we have $f(\lambda Z) = \lambda^n f(Z)$. 
Then,
\begin{itemize}
 \item[(i)] There exists an $n$-homogeneous polynomial $p \in \F_d$, such that $p|_{\fV} = f$;
 
 \item[(ii)] $f \in \mlt \cH_{\fV}$ and $\|f\|_{\infty} = \|f\|_{\mlt} = \|f\|_{\cH_{\fV}}=\|p\|_{\infty}= \|p\|_{\mlt}=\|p\|_{\cH^2_d}$.
\end{itemize}
\end{lem}
\begin{proof}
To prove $(i)$ we extend, by definition, $f$ to an nc holomorphic function $\tilde{f}$ on a ball of radius $\epsilon$ around $0$ and denote this ball by $B$. Since $f$ is $n$-homogeneous we can take $\tilde{f}$ to be a homogeneous polynomial of degree $n$, since for every $Z \in \fV \cap B$:
\[
\tilde{f}_n(Z) = \frac{1}{2\pi} \int_0^{2\pi} \tilde{f}(e^{i \theta} Z) e^{-in\theta}d\theta = f(Z).
\]
Now for every $Z \in \fV$ we can find $r > 0$, such that $rZ \in \fV \cap B$ and since both $\tilde{f}$ and $f$ are homogeneous we get that $\tilde{f}(Z) = f(Z)$. Thus we have found a homogeneous polynomial of degree $n$, such that $\tilde{f}|_{\fV} = f$.

We first prove $(ii)$ for homogeneous polynomials on all of $\fB_d$. Let $p = \sum_k a_k z^{\alpha_k}$, where for every $k$ we have $|\alpha_k| = n$. 
Now $\|p\|_{\cH^2_d}^2 = \sum_k |a_k|^2$. 
Since $M_{z^{\alpha_k}}$ are isometries with orthogonal ranges, we obtain that for every $g \in \cH^2_d$:
\[
\|M_p g\|^2 = \sum_k |a_k^2| \|M_{z^{\alpha_k}} g\|^2 = \|g\|^2 \sum_k |a_k|^2.
\]
Now since the supremum norm coincides with the multiplier norm we have our equality in the case of the entire ball.

To prove $(ii)$ in general, we note first that since $f$ is a restriction of a polynomial it is bounded on $\fV$ and furthermore, it is a multiplier on $\cH_{\fV}$. Now since $\fV$ is homogeneous, it is cut out by an ideal generated by polynomials that we shall denote by $I_{\fV}$. Let us identify $\cH_{\fV} = I_{\fV}^{\perp}$ inside $\cH^2_d$, thus we may choose the polynomial obtained in $(i)$ to lie in $\cH_{\fV}$. 
Then we have:
\[
\|f\|_{\mlt \cH_{\fV}} \geq \|f\|_{\cH_{\fV}} = \|p\|_{\cH^2_d} = \|p\|_{\mlt \cH^2_d} \geq \|f\|_{\mlt \cH_{\fV}}.
\]
Hence the inequalities are in fact equalities. 
For the supremum norm we first note that $\|f\|_{\infty} \leq \|f\|_{\mlt}$; indeed, with $p$ as above, 
\[
\sup_{X \in \fV} \|f(X)\| \leq \|p\|_\infty = \|p\|_{\mlt \cH^2_d} = \|f\|_{\mlt \cH_{\fV}}.
\]
It remains to prove the reverse inequality. 
To this end we note that:
\[
\|f\|_{\mlt \cH_{\fV}} = \|f\|_{\cH_{\fV}} = \|f \cdot 1\|_{\cH_{\fV}} = \|p(L)1\|_{\cH_{\fV}}=\|\lim_{k\to \infty} p(P_k L P_k) 1\|_{\cH_{\fV}} \leq \|f\|_{\infty}
\]
Here $P_k L P_k$ is the compression of the shifts to the finite dimensional space of polynomials of degree less than or equal to $k$. 
(We can plug $P_k L P_k$ into $p$ since it is a polynomial.)
\end{proof}

\begin{proof}[Proof of Theorem \ref{thm:mult_are_bounded_on_V}, for a homogeneous variety $\fV \subseteq \fB_d$, $d<\infty$]
We need to show that if $f \in M_n(H^\infty(\mathfrak{V}))$, then $f$ is a multiplier and that $\|f\|_{\mlt \cH_\fV}\leq \|f\|_{\infty}$ (the reverse inequality follows immediately by Theorem \ref{thm:quotient_mult}). 

Let $f\in M_n(H^\infty(\fV))$. For every $n \in \mathbb N$ set
\[
f_n(Z) = \frac{1}{2\pi} \int_0^{2 \pi} \alpha_\theta(f)(Z) e^{-i n \theta} d \theta, \quad \forall Z \in \fV.
\]
Clearly, each $f_n$ is in $H^\infty(\fV)$ with $\|f_n\|_\infty \leq \|f\|_\infty$, and for every $Z \in \fV$ and every $\lambda \in \D$ we have $f_n(\lambda Z) = \lambda^n f_n(Z)$. Thus, by Lemma \ref{lem:homog_func_is_poly}, the nc functions $f_n$ are restrictions of polynomials $p_n$, and for every $n$ we have  $\|f\|_{\infty} = \|f\|_{\mlt \cH_{\fV}} = \|f\|_{\cH_{\fV}}=\|p\|_{\infty}= \|p\|_{\mlt \cH^2_d}=\|p\|_{\cH^2_d}$.

Now, let $0<r<1$. 
As $\sum_n r^n \|p_n\|_* \leq \frac{\|f\|_\infty}{1-r}$, where $\| \cdot \|_*$ stands for the norm of either $\mlt\cH^2_d$, $\cH^2_d$, or $H^\infty(\fB_d)$, the series $g_r = \sum_n r^n p_n$ defines a function which is in $\mlt\cH^2_d = H^\infty(\fB_d)$. 
In addition, a simple computation shows that $g_r(Z)=f_r(Z):=f(rZ)$ for all $Z \in \fV$.

Define $S = P_{\cH_\fV} L P_{\cH_\fV}$, that is, $S$ is the compression of the shift $L$ to $\cH_\fV$. 
Then recalling Lemma \ref{lem:quotient_mult}, we have 
\[
\begin{split}
\| f_r\|_{\mlt \cH_\fV}    &=        \|P_\fV M_{g_r} P_\fV\|=\|g_r(P_\fV L P_\fV)\|          \\
                                       &=        \|g(rS)\|=\lim_{k\to \infty} \|f(rP_k S P_k)\|               \\
                                       &\leq    \|f\|_\infty. 
\end{split}
\]
Since $f$ is the bounded pointwise limit of $f_r$, letting $r \to 1$, we conclude that $f$ is a multiplier and that $\|f\|_{\mlt \cH_\fV} \leq \|f\|_\infty$.
\end{proof}

\section{The isomorphism problem for homogeneous varieties}\label{sec:homog_isom}

For every $n$, we write $\bM_d(n) = M_n^d = \bC^d \otimes M_n$.
Given a subset $\cX \subseteq \bM_d$, we define its {\bf {\em matrix span}} $\mspn{\cX}$ to be the graded set $\mspn{\cX} = \sqcup_n \mspn{\cX}(n)$ given by
\[
\mspn{\cX}(n) = \spn\{[I_d \otimes T] (X) : X \in \cX(n) \,\, , \,\, T \in \cL(M_n)\}.
\]
(Here, $\cL(M_n)$ denotes the linear maps on $M_n$.)

\begin{lem}\label{lem:Delta_id}
Let $F : \fB_d \to \bM_d$ be an nc map, and let $\cX \subseteq \bM_d$.
If $\Delta F(0,0)$ is the identity on $\cX$, then $\Delta F(0,0)$ is equal to the identity on $\mspn{\cX}$.
\end{lem}
\begin{proof}
If $F : \fB_d \to \bM_d$ is an nc map, then using \cite[Proposition 2.15]{KVBook} we find that $\Delta F(0,0)$ acts as $A \otimes I$ on $\bC^d \otimes M_n$, where $A$ is a linear map on $\bC^d$.
Because $\Delta F(0,0)$ is also an nc holomorphic function, the linear map $A$ does not depend on $n$.
It follows that if we fix $n$, $\Delta F(0,0)\big|_{M_n^d}$ commutes with every operator of the form $I_d \otimes T$, where $T$ is a linear map on $M_n$.
The result follows.
\end{proof}

\begin{lem}\label{lem:matspan}
Let $\cX \subseteq \bM_d$ be an nc set.
Then for all $n$ there exists a subspace $V_n \subseteq \C^d$ such that $\mspn{\cX}(n) = V_n \otimes M_n$.
There exists a minimal subspace $V \subseteq \bC^d$ such that $\mspn{\cX}(n) \subseteq V \otimes M_n$ for all $n$, and if $d< \infty$,
then $\mspn{\cX}(n) = V \otimes M_n$ for all sufficiently large $n$.
\end{lem}
\begin{proof}
Denote $\tilde{\cX} = \mspn{\cX}$.
Fix $n$.
For every linear $f : \bC^d \to \bC$, the space $(f \otimes \id_{M_n}) \tilde{\cX}_n$ is invariant under $\cL(M_n)$.
Therefore, either $(f \otimes \id_{M_n}) \tilde{\cX}_n = 0$ or $(f \otimes \id_{M_n})\tilde{\cX}_n = M_n$.
If the latter holds for every $f \in (\bC^d)^*$, then $\tilde{\cX}_n = \bC^d \otimes M_n$.
Let
\[
\Phi_n = \{f \in (\bC^d)^* : \|f\|=1 \textrm{ and } (f \otimes \id_{M_n}) \tilde{\cX}_n = 0\} .
\]
Put
\[
V_n = \{x \in \C^d : f(x) = 0 \,\, \textrm{ for all } \,\, f \in \Phi_n\} ,
\]
(interpreted as $\C^d$ in case $\Phi_n = \emptyset$).
The fact that $\tilde{\cX}$ is an nc set containing zero, implies that $\Phi_{n+1} \subseteq \Phi_n$.
It follows that $V_n \subseteq V_{n+1}$ for all $n$.

We will now show that $\tilde{\cX}_n = V_n \otimes M_n$.
Obviously $\tilde{\cX}_n \subseteq V_n \otimes M_n$.
To prove the converse, we first make some elementary observations.
Suppose that $\sum x_i \otimes A_i \in \tilde{\cX}_n$, where $A_1, A_2, \ldots$ are linearly independent.
Then from the definition of $\tilde{\cX}_n$, $x_i \otimes A_i \in \tilde{\cX}_n$ for all $i$, and therefore --- again, from the definition of $\tilde{\cX}_n$ --- it follows that $x_i \otimes A \in \tilde{\cX}_n$ for all $A \in M_n$.

Now let
\[
W = \{v \in \C^d : v \otimes A \in \tilde{\cX}_n \, \textrm{ for some nonzero } \, A \in M_n\}.
\]
Then by the above observations $W$ is a subspace, and $W \otimes M_n \subseteq \tilde{\cX}_n \subseteq V_n \otimes M_n$.
Now if $f \in (\C^d)^*$ satisfies $f(W) = 0$, then $f \in \Phi_n$.
This implies that $W=V_n$.

Define $V = \ol{\bigcup_n V_n}$. 
If $\Phi_n = \emptyset$ for some $n$, then $V = \C^d$ and $\tilde{\cX}_n = \bC^d \otimes M_n$ for all sufficiently large $n$.
Otherwise, if $d< \infty$, then the spaces $V_n$ form an increasing sequence of subspaces of $\C^d$ and therefore must stabilize.
\end{proof}


\begin{lem} \label{lem:homog_nc_der}
Let $\fV \subset \fB_d$ be a homogeneous variety and let $f \colon \fB_d \to \fB_d$ be an nc holomorphic function, such that $f|_{\fV}$ is the identity.
Then for every $X \in \fV$, 
\[
\Delta f (0,0)(X) = X.
\]
\end{lem}
\begin{proof}
Recall that by the nc difference differential relation we have that for every $Y \in \fB_d$ and every $t \in \C^{\times}$ the following relation holds:
\[
f(tY) = f(tY) - f(0) = \Delta f(tY,0)(tY - 0) = t \Delta f(tY,0)(Y) .
\]
Therefore, for $X \in \fV$ we have that:
\[
X = \frac{1}{t}f(tX) = \Delta f(tX,0)(X).
\]
Since the above equality holds for every $t \in \C^{\times}$ and furthermore by \cite[Theorem 7.46]{KVBook} $\Delta f(\cdot,\cdot)$ is an nc holomorphic function of order $1$, we can take the limit as $t$ goes to $0$ to obtain the desired result.
\end{proof}

\begin{thm}\label{thm:isomorphism_homo}
Let $\fV \subseteq \fB_d$ and $\fW \subseteq \fB_{e}$ be homogeneous nc varieties, and let $\alpha : H^\infty(\fV) \to H^\infty(\fW)$ be a completely isometric isomorphism. Assume that $d$ and $e$ are finite or that $\alpha$ is weak-$*$ continuous.
Then $\fV$ and $\fW$ are conformally equivalent, in the sense that one may assume that there is some $k$ such that $\fV, \fW \subseteq \fB_k$, and that under this assumption there exists an automorphism $G \in \operatorname{Aut}(\fB_k)$ such that $G(\fW)  = \fV$, and such that
\[
\alpha(f) = f \circ G \quad, \quad f \in H^\infty(\fV).
\]
\end{thm}
\begin{proof}
By Lemma \ref{lem:matspan}, there are increasing sequences of subspaces $V_n \subseteq \C^d$ and $W_n \subseteq \C^e$ such that for every $n$,
\[
\mspn{\fV}(n) = V_n \otimes M_n \,\, \textrm{ and } \,\, \mspn{\fW}(n) = W_n \otimes M_n .
\]
Put $V = \overline{\bigcup_n V_n}$ and $W = \overline{\bigcup_n W_n}$. 
Since $\fV \subseteq \sqcup_n  (V \otimes M_n) \cap \fB_d$ and $\fW \subseteq \sqcup_n  (W \otimes M_n) \cap \fB_e$, we may as well assume that $\sqcup_n (V \otimes M_n) \cap \fB_d = \fB_d$ and $\sqcup_n (W \otimes  M_n) \cap \fB_e = \fB_e$ (otherwise, we restrict attention to these sub-balls). 

By Theorem \ref{thm:isomorphism}, we find that there are nc holomorphic maps $G : \fB_e \to \fB_d$ and $H : \fB_d \to \fB_e$ such that $\alpha(f) = f \circ G$ for $f \in H^\infty(\fV)$, and $\alpha^{-1}(g) = g \circ H$ for $g \in H^\infty(\fW)$.
We need to show that $G \circ H$ and $H \circ G$ are equal to the identity on $\fB_e$ and $\fB_d$, respectively. 

Define $F = G \circ H$.
As $F$ fixes the homogeneous variety $\fV$,
Lemma \ref{lem:homog_nc_der} says that the derivative $\Delta F(0,0)$ fixes every element in $\fV$.
By Lemma \ref{lem:Delta_id}, $\Delta F (0,0)$ fixes every point of $(\mspn{\fV})(n) = V_n \otimes M_n$, for all $n$. 
We claim that $\Delta F(0,0)$ is the identity. 
Indeed, in Lemma \ref{lem:Delta_id} we noted that there is some linear $A$ such that $\Delta F(0,0)$ acts as $A \otimes I_n$ on $\bC^d \otimes M_n$, so $A\big|_{V_n}$ is the identity. 
Since $\Delta F(0,0)$ is continuous, it follows that $A$ is the identity on $V$, and $\Delta F(0,0)$ is the identity as claimed.   
Since $F(0_n) = 0_n$,  Theorem \ref{thm:nc_cartan} implies that $F$ is the identity on $\fB_d$.

In the same way, we obtain that $H \circ G$ is the identity on $\fB_e$. 
This shows that $d=e$, and that $G,H$ are automorphisms of $\fB_d$, as required.

\end{proof}

\begin{remark}
We cannot obtain that $\fV$ and $\fW$ are related by an automorphism, without first embedding them in some $\fB_k$: consider $\fV = \fB_\infty$ and
\[
\fW = \{Z = (Z_1, Z_2, \ldots) \in \fB_\infty : Z_1 = 0\}.
\]
Then clearly $H^\infty(\fV) = H^\infty(\fB_\infty) \cong H^\infty(\fW)$, but there is no automorphism of $\fB_\infty$ that takes $\fB_\infty$ onto $\fW$. 
Of course, after restricting attention to the matrix spans, the problem disappears. 
\end{remark}

\begin{cor}\label{cor:isomorphism_homo}
Let $\fV, \fW \subseteq \fB_d$ be two homogeneous varieties.
Then $H^\infty(\fV)$ and $H^\infty(\fW)$ are completely isometrically isomorphic via a weak-$*$ continuous map if and only if $\fV$ and $\fW$ are conformally equivalent in the sense of Theorem \ref{thm:isomorphism_homo}.
Furthermore, every weak-$*$ continuous completely isometric isomorphism is implemented by composition with an automorphism of $\fB_d$. 
(Recall that if $d \in \bN$, then every completely isometric isomorphism is automatically weak-$*$ continuous.)
\end{cor}

We shall now show that if two homogeneous nc varieties $\fV,\fW \subseteq \fB_d$ are conformally equivalent, then $\fV$ is the image of $\fW$ under an invertible linear transformation.
We start with the following lemma.

\begin{lem}\label{lem:DiscToDisc}
Let $\fV, \fW \subseteq \fB_d$ be two conformally equivalent homogeneous varieties.
If $0$ is not mapped to $0$, then there exist
two discs $D_1 \subseteq \fV(1)$ and $D_2 \subseteq \fW(1)$, both containing $0$, such that $D_1$ is mapped by the conformal equivalence onto $D_2$.
\end{lem}
\begin{proof}
Let $G\in\operatorname{Aut}(\fB_d)$ be an automorphism mapping $\fV$ onto $\fW$. 
If $0$ is not mapped to $0$, then both $V:=\fV(1)$ and $W:=\fW(1)$ are non-trivial homogeneous varieties which are conformally equivalent. Since automorphisms of the commutative ball $\mathbb B_d:= \fB_d(1)$ map affine sets to affine sets, $G$ maps affine discs to affine discs. Thus the disc $D_1 \subseteq V$ spanned by $G^{-1}(0)$ is mapped to a disc $D_2$ containing $0$. As $0 \neq G(0) \in W\cap D_2$, $D_2$ is the disc spanned by $G(0)$ and therefore must be contained in $W$.
\end{proof}

\begin{prop}\label{prop:biholo=>0-biholo}
Let $\fV,\fW \subseteq \fB_d$ be two conformally equivalent homogeneous varieties. Then there exists a conformal equivalence $F$ of $\fV$ onto $\fW$ that maps $0$ to $0$.
\end{prop}

\begin{proof}
We import the ``disc trick" from \cite{DRS11} to the current setting (see also \cite[Lemma 5.9]{SalomonShalit}).
Let $G$ be a conformal equivalence mapping $\fV$ onto $\fW$.
If $0$ is mapped by $G$ to $0$, we are done.
Assume that $G(0) \neq 0$.
We will prove that there exists an automorphism $F$, mapping $\fV$ onto $\fW$, such that $F(0) = 0$.

Lemma \ref{lem:DiscToDisc} implies there exist two discs $D_1 \subseteq \fV(1)$ and $D_2 \subseteq \fW(1)$ such that
$G(D_1) = D_2$.
Define
\[
\cO(0;\fV):=\{z \in D_1 : z=F(0) \text{ for some automorphism $F$ of $\fV$} \},
\]
and
\[
\cO(0;\fV,\fW):=\left\{z \in D_2~ :~
\begin{minipage}{0.51\linewidth}
\text{$z=F(0)$ for some conformal equivalence}\\
\text{$F$ of $\fV$ onto $\fW$}
\end{minipage}
\right\}.
\]
Since homogeneous varieties are invariant under multiplication by complex numbers, it is easy to check that these sets are circular, that is, for every $\mu \in \cO(0;\fV)$ and $\nu \in \cO(0;\fV,\fW)$, it holds that $C_{\mu,D_1}:=\{z \in D_1: |z|=|\mu|\} \subseteq \cO(0;\fV)$ and $C_{\nu,D_2}:=\{z \in D_2: |z|=|\nu|\} \subseteq \cO(0;\fV,\fW)$.

Now, as $G(0)$ belongs to $\cO(0;\fV,\fW)$, we obtain that $C := C_{G(0),D_2}\subseteq \cO(0;\fV,\fW)$.
Therefore, the circle $G^{-1}(C)$ is a subset of $\cO(0;\fV)$; note that this circle passes through the point $0 = G^{-1}(G(0))$. 
As $\cO(0;\fV)$ is circular, every point of the interior of the circle $G^{-1}(C)$ lies in $\cO(0;\fV)$.
Thus, the interior of the circle $C$ must be a subset of $\cO(0;\fV,\fW)$. 
But the interior of $C$ contains $0$. 
We conclude that $0 \in \cO(0;\fV,\fW)$.
\end{proof}

\begin{cor} \label{cor:equivalence_is_linear_homog}
Let $\fV, \fW \subseteq \fB_d$ be two conformally equivalent homogeneous varieties.
Then there is a unitary transformation which maps $\fV$ onto $\fW$.
\end{cor}

\begin{proof}
Suppose that there exists $\psi \in \operatorname{Aut}(\fB_d)$ such that $\fW = \psi(\fV)$.
By Proposition \ref{prop:biholo=>0-biholo}, $\fV$ and $\fW$ are conformally equivalent via a $0$ preserving map $\varphi \in \operatorname{Aut}(\fB_d)$.
An automorphism $\varphi \in  \operatorname{Aut}(\fB_d) = \operatorname{Aut}(\bB_d)$ such that $\varphi(0) = 0$ is a unitary transformation \cite[Theorem 2.2.5]{Rudin}.
Alternatively, the free version of Cartan's uniqueness theorem \cite[Theorem 7]{McT16} says that if there exists a $0$ preserving nc automorphism of a circular bounded nc domain, then it is the restriction of a unitary linear map. 
\end{proof}

\begin{cor}\label{cor:iso_homo_weak}
Let $\fV \subseteq \fB_d$ and $\fW \subseteq \fB_e$ be two homogeneous varieties.
Then $H^\infty(\fV)$ and $H^\infty(\fW)$ are completely isometrically isomorphic via a weak-$*$ continuous map, if and only if there is an embedding $\fV, \fW \subseteq \fB_k$ and a unitary transformation which maps $\fV$ onto $\fW$ 
(in case $d<\infty$, then the weak-$*$ continuity requirement is superfluous). 
\end{cor}

In Theorem \ref{thm:iso_homo_cont_older}, we will show that when considering norm closed analogues of the multiplier algebra, or when $d<\infty$, the condition ``completely isometrically isomorphic" can be weakened ``isometrically isomorphic".

\section{Algebras of continuous functions} \label{sec:continuous}

In this section, we consider algebras of continuous multipliers on subvarieties of the noncommutative ball. First, we require a few definitions.
Let $\Omega \subset \M_d$ be an nc subset and let $f \colon \Omega \to \M_1$ be an nc function. 
We say that $f$ is {\em uniformly continuous} on $\Omega$ if for every $\epsilon > 0$ there exists a $\delta > 0$, such that for every $n \in \N$ we have that if $X, Y \in \Omega(n)$ are such that $\|X - Y \| < \delta$, then $\|f(X) - f(Y)\| < \epsilon$
(recall that for $A \in \Omega(n)$ we let $\|A\|$ denote the norm of the row operator $(A_1, \ldots, A_d) : (\bC^n)^d \to \bC^n$).

Let $\Omega \subset \M_d$ be an nc set and $f \colon \Omega \to \M_1$ a uniformly continuous nc function.
Given an $\epsilon > 0$, the $\delta$ that we obtain from uniform continuity actually satisfies, that for every $X \in \Omega(n)$ and $Y \in \Omega(m)$, such that $\|X^{\oplus \ell/n} - Y^{\oplus \ell/m}\| < \delta$, where $\ell$ is a common multiple of $n$ and $m$, we have that $\|f(X)^{\oplus \ell/n} - f(Y)^{\oplus \ell/m}\| < \epsilon$.

The proof of the following lemma is standard and we state it for the sake of completeness.
\begin{lem} \label{lem:unif_cont_facts}
Let $\Omega \subset \M_d$ be an nc set, then:
\begin{enumerate}
\item[(i)] A linear combination of nc functions that are uniformly continuous on $\Omega$ is uniformly continuous on $\Omega$.
\item[(ii)] A product of two nc functions that are bounded and uniformly continuous on $\Omega$ is uniformly continuous on $\Omega$.
\item[(iii)] If a sequence of bounded and uniformly continuous nc functions on $\Omega$ converges in the supremum norm, then the limit is also bounded and uniformly continuous on $\Omega$.
\end{enumerate}
\end{lem}

\begin{cor} \label{cor:poly_unif_cont}
Let $\Omega \subset \M_d$ be a bounded nc set. 
Then the polynomials are uniformly continuous on $\Omega$.
\end{cor}
This leads us to define the following two algebras.
Let $\fV \subseteq \fB_d$.
Let $A(\fV)$ denote all functions in $H^\infty(\fV)$ which continue uniformly continuously to $\overline{\fV}$, and let $\fA_\fV$ denote the norm closure of the image of the polynomials under the quotient map $H^{\infty}(\fB_d) \to H^{\infty}(\fV)$.
It is clear from the Lemma \ref{lem:unif_cont_facts} and Corollary \ref{cor:poly_unif_cont} that $\fA_{\fV} \subseteq A(\fV)$.

In this section, we treat these algebras, concentrating mostly on homogeneous varieties.
We will prove that $A(\fV) = \fA_\fV$ in the case that $\fV$ is a homogeneous variety.
We then obtain a classification of the algebras of continuous holomorphic functions $A(\fV)$ analogous to the classification of algebras of bounded holomorphic functions.
We will also obtain a homogeneous Nullstellensatz here in the context of algebras of continuous functions.

\subsection{The equality $A(\fV) = \fA_\fV$ for homogeneous varieties}

For a homogeneous variety $\fV \subseteq \fB$, we denote by $\ol{\fV}$ its closure in $\ol{\fB_d}$. The following notion is a weaker version of uniform continuity that later will turn out to be equivalent in the case of homogeneous varieties.

Let $\fV$ be a homogeneous variety in $\fB_d$, and let $f \in H^\infty(\fV)$.
We will say that $f$ is {\em radially uniformly continuous} on $\fV$ if $f$ for every $\epsilon > 0$, there exists a $\delta > 0$, such that for every $r,s \in (0,1)$ and every $X \in \fB_d$, if $|r - s| < \delta$, then $\|f(rX) - f(sX)\| < \epsilon$.
Every radially uniformly continuous function in $H^\infty(\fV)$ extends to an nc function $f: \ol{\fV} \to \M_1$, which is radially continuous at every point of the boundary (we will see below that such a function is in fact uniformly continuous on $\ol{\fV}$).
It is immediate that linear combinations, products, and uniform limits of functions that are radially uniformly continuous is also radially uniformly continuous.

\begin{prop} \label{prop:approximation_by_polynomials}
Let $\fV$ be a homogeneous variety in $\fB_d$, and let $f \in H^{\infty}(\fV)$.
Then $f \in \fA_\fV$ if and only if $f$ is radially uniformly continuous on $\fV$.
\end{prop}
\begin{proof}
In the proof we will use the following fact repeatedly:
\[
\|f\| = \sup \left\{ \|f(X)\| \mid X \in \fV\right\}.
\]
Assume that $f \in H^{\infty}(\fV)$ is radially uniformly continuous on $\fV$.
Let $f_r(Z)=f(rZ)$ for all $Z \in \fV$. Applying the methods of \ref{lem:homog_func_is_poly}, we see that for $d<\infty$, $f_r$ is a norm converging series of $n$-homogeneous polynomials. If $d=\infty$, a similar argument shows that $f_r$ is a norm converging series of $n$-homogeneous {\em holomorphic nc-functions}, each --- due to the Fock structure of $\cH^2_d$ --- is in the norm closure of the monomials of degree $n$. In any case, $f_r \in \fA_\fV$. In addition, the net $f_r$ $\textsc{sot}$-converges to $f$ as $r \to 1$.
We need to show that the convergence is in fact in norm.
It suffices to show that for every $\epsilon > 0$, there exists $\delta > 0$, such that if $1 - \delta < r < 1$, then $\|f_r - f\| < \epsilon$.
Given $\epsilon > 0$, let us choose $\delta$ from uniform continuity that corresponds to $\frac{\epsilon}{2}$.
Now for every $r \in (1-\delta,1)$, we can choose $X \in \fV$, such that $\|f_r - f\| < \|f_r(X) - f(X)\| + \frac{\epsilon}{2}$.
Since $\|r X - X\| = 1 - r < \delta$ we get that $\|f_r - f\| < \epsilon$, and this concludes the proof.
The converse is trivial since every $f \in \fA_{\fV}$ is uniformly continuous on $\ol{\fV}$.
\end{proof}

\begin{cor} \label{cor:levelwise_continuous}
Let $\fV \subseteq \fB_d$ be a homogeneous variety.
Every radially uniformly continuous multiplier on $\cH_\fV$ extends to a uniformly continuous function on $\ol{\fV}$.
In particular, $A(\fV) = \fA_\fV$.
\end{cor}
\begin{proof}
As we have already observed, $\fA_{\fV} \subseteq A(\fV)$.
On the other hand, clearly every uniformly continuous function is radially uniformly continuous.
\end{proof}

\subsection{Nullstellensatz and quotients}
In accordance with parts of the literature, we let  $\fA_d$ denote the closure of polynomials in the sup norm, that is
\[
\fA_d = \fA_{\fB_d} = A(\fB_d).
\]
Following Popescu \cite{Popescu96}, the algebra $\fA_d$ is called the {\em noncommutative disc algebra};
the discussion above shows what a suitable designation this is. 

If $\fV \subseteq \fB_d$, we put
\[
\cI_{\fV} = \{f \in \fA_d : f(X) = 0 \,\, \textrm{ for all } \,\, X \in \fV\}.
\]

\begin{thm} \label{thm:cont_nullss}
Let $d \in \bN$.
If $I \triangleleft \F_d$, is homogeneous, then
\[
\cI_{V_{\fB_d}(I)} = \ol{I}^{\|\cdot\|}.
\]
If $J \triangleleft \fA_d$ is a homogeneous norm closed ideal, then
\[
\cI_{V_{\fB_d}(J)} = J.
\]
\end{thm}

\begin{proof}
Let us write $\fV = V_{\fB_d}(I)$ and $\cI_{\fV} = \cI_{V_{\fB_d}(I)}$ for simplicity. To prove the first part note that for every $f = \sum_{n=0}^{\infty} f_n \in \cI_{\fV}$ we have that $f_n \in I$ and the Ces\`{a}ro sums of the Taylor expansion converges in norm to $f$. 
Alternatively, we may --- as in the proof of Theorem \ref{thm:homoclosedweakstar} --- content ourselves with proving the second assertion, since the variety cut out by $I$ equals the variety cut out by $\ol{I}^{\|\cdot\|}$.

As for the second assertion, let $\fV$ be the variety cut out by $J$. 
Clearly, $J \subseteq \cI_{\fV}$ and we only need to prove the other inclusion. 
Since both ideals are homogeneous, we have that $\ol{J \cap \F_d}^{\|\cdot\|} = J$ and similarly for $\cI_{\fV}$. 
Therefore, the proof of the proposition reduces to the polynomial case. 
But Theorem \ref{thm:null_poly} implies that
\[
\cI_{V_{\fB_d}(J)} \cap \F_d = I({V_{\fB_d}(J)}) = I({V_{\fB_d}(J \cap \F_d)}) = J \cap \F_d.
\]
Thus, $\cI_{V_{\fB_d}(J)}=J$.
\end{proof}

\begin{lem}
Let $\fV \subseteq \fB_d$ be a homogeneous variety.
Let $\cI_\fV$ and $\cJ_\fV$ denote, respectively, the ideals of functions in $A(\fB_d)$ and $H^\infty(\fB_d)$, respectively, that vanish on $\fV$.
Then the natural map $A(\fB_d)/\cI_{\fV}\to H^{\infty}(\fB_d)/\cJ_\fV$ given by $f + \cI_\fV \mapsto f + \cJ_\fV$ is completely isometric.
\end{lem}
\begin{proof}
Given $f \in M_k(A(\fB_d))$, we need to prove that
\[
d(f, M_k(\cI_\fV)) := \inf \{\|f - g\| : g \in M_k(\cI_\fV)\} = \inf\{\|f-g\| : g \in M_k(\cJ_\fV)\}.
\]
Fixing $\epsilon$, let $r>0$ be such that $\|f-f_r\| < \epsilon$.
Then for every $g \in M_k(\cJ_\fV)$, we have that $g_r \in M_k(\cI_\fV)$ --- here we use homogeneity.
Applying the methods of Lemma \ref{lem:homog_func_is_poly}, and noting that $f_r$ is obtained from $f$ by integrating against the Poisson kernel
\[
f_r(Z) = \frac{1}{2\pi} \int_0^{2\pi} f(e^{i\theta}Z) \left(\sum_{n=0}^\infty (r e^{-i\theta})^n\right) d\theta, 
\]
we get that $f \mapsto f_r$ is a complete contraction. 
Therefore, we see that for $g \in M_k(\cJ_\fV)$,
\[
\|f - g\|\geq \|f_r - g_r\| \geq \|g_r - f\| - \|f_r - f\| \geq d(f,M_k(\cI_\fV)) - \epsilon.
\]
\end{proof}

\begin{prop}\label{prop:quot_cont}
Let $\fV \subseteq \fB_d$ be a homogeneous variety.
Then $A(\fV)$ is completely isometrically isomorphic to $A(\fB_d)/\cI_{\fV}$.
\end{prop}
\begin{proof}
Since the restriction of a function $f \in A(\fB_d)$ to $\fV$ will clearly result in a function that extends to a uniformly continuous function on $\ol{\fV}$, we have a well defined map $A(\fB_d) \to A(\fV)$ that factors through $A(\fB_d)/\cI_{\fV}$.
In other words, we have a map $A(\fB_d)/\cI_{\fV} \to A(\fV)$, and our goal is to show that it is completely isometric.

Now let $\cJ_{\fV}$ be the ideal of functions in $H^\infty(\fB_d)$ that vanish on $\fV$.
By Theorem \ref{thm:quotient_mult}, the restriction map $f \mapsto f\big|_\fV$ induces a completely isometric isomorphism $H^{\infty}(\fB_d)/{\cJ_\fV} \to H^{\infty}(\fV)$.
Thus, for every $f \in M_k(A(\fB_d))$, we have by the previous lemma
\[
d(f,M_k(\cI_\fV)) = d(f,M_k(\cJ_\fV)) = \|f\big|_\fV\|.
\]
\end{proof}

\subsection{Classification up to completely isometric isomorphism}

Let us consider the completely contractive finite dimensional representations of $\fA_d$. 
If $\rho \colon \fA_d \to M_n(\C)$ is completely contractive, then $X_\rho:=(\rho(z_1),\ldots,\rho(z_d))$ is a row contraction and thus a point in $\ol{\fB_d}(n)$.
On the other hand, using Popescu's functional calculus, we note that every point $X \in \ol{\fB_d}(n)$ induces a unique completely contractive representation $\rho_X \colon \fA_d \to M_n(\C)$. 
Hence, the finite dimensional completely contractive representations of $\fA_d$ are in one-to-one correspondence with points of $\ol{\fB_d}$.
Furthermore, given a homogeneous nc variety $\fV \subseteq \fB_d$ and the corresponding homogeneous ideal $\cI_\fV \subset \fA_d$, the finite dimensional completely contractive representations of $\fA_d/\cI_\fV$ are in one-to-one correspondence with the variety $\ol{\fV} := \ol{V(\cI_\fV)}$ cut out by $\cI_\fV$ in the closed ball $\ol{\fB_d}$.
Hence every unital completely contractive homomorphism $\varphi \colon A(\fV) = \fA_d/\cI_\fV \to A(\fW) = \fA_e/\cI_\fW$ induces a map $\rho \mapsto \rho\circ\varphi$ from $\ol{\fW}$ to $\ol{\fV}$.
We will let $\varphi^*$ denote the map $\rho \mapsto \rho\circ\varphi$.

Now, set $g_j = \varphi(z_j)$, $j=1, \ldots,d$.
Then
\[
X_{\varphi^* (\rho)} = (\rho (\varphi(z_1)), \ldots, \rho(\varphi(z_d))) = (\rho(g_1),\ldots,\rho(g_d)) .
\]
Hence we can consider $\varphi^*$ as a map that takes the point $X_{\rho} \in \ol{\fW}$ to the point 
\[
(g_1(X_\rho),\ldots,g_d(X_\rho)) \in \ol{\fV}.
\]
In other words $\varphi^*$ defines an nc map from $\overline{\fW}$ to $\overline{\fV}$.
This discussion leads us to the following proposition, in which we describe the completely contractive maps homomorphisms in the case of the norm closed algebras. 
It is interesting to contrast with the case of full multipliers (Proposition \ref{prop:cc_homo}). 

\begin{prop} 
Let $\fV \subseteq \fB_d$ and $\fW \subseteq \fB_{e}$ be homogeneous nc varieties, and $\varphi : A(\fV) \to A(\fW)$ a unital completely contractive homomorphism. 
For every $\epsilon > 0$, there exists a continuous nc map $G \colon \overline{\fB_e} \to (1+\epsilon)\overline{\fB_d}$ such that $G|_{\overline{\fV}} = \varphi^*$, and such that $G$ implements $\varphi$:  
\[
\varphi(f) = f\circ G. 
\]
\end{prop}
\begin{proof}
Set $g_j = \varphi(z_j)$, $j=1, \ldots,d$ as above.
Invoking Proposition \ref{prop:quot_cont}, we lift the maps $g_j$ to $G_j \in \fA_d$ and set $G = (G_1,\ldots,G_d)$. 
Now for all $f \in A(\fV)$ and $W \in \fW$, 
\[
\varphi(f) (W) = \Phi_W \circ \varphi(f) = \Phi_{G(W)}(f) = f \circ G(W).
\]
\end{proof}

Our next goal is to prove a norm closed counterpart of Theorem \ref{thm:isomorphism_homo}, namely, to show that two homogeneous varieties $\fV$ and $\fW$ are conformally equivalent if and only if the norm closed algebras $A(\fV)$ and $A(\fW)$ are completely isometric isomorphic. To achieve this, we first need to show that every completely isometric isomorphism between $A(\fV)$ and $A(\fW)$ is implemented as a precomposition with an nc map from one nc ball into the other mapping one variety onto the other. 
\begin{lem}\label{lem:isomorphism_continuous}
Let $\fV \subseteq \fB_d$ and $\fW \subseteq \fB_{e}$ be homogeneous nc varieties, and let $\alpha : A(\fV) \to A(\fW)$ be a completely isometric isomorphism. 
Then there exists an nc map $G: \fB_e \to \fB_d$ such that $G \big|_{\fW}  = G_\alpha|_{\fW}$ maps $\fW$ bijectively onto $\fV$, which implements $\alpha$ by the formula
\[
\alpha(f) = f \circ G \quad, \quad f \in A(\fV).
\]
\end{lem}
\begin{proof}
Since $\alpha$ is completely isometric, it takes completely contractive representations to completely contractive representations.
Hence $\alpha^*$ and $(\alpha^{*})^{-1}$ take $\ol{\fW}$ to $\ol{\fV}$ and vice versa.
By Theorem \ref{thm:quotient_mult}, we can lift $\alpha^*$ and $(\alpha^{*})^{-1}$ to nc maps $G \colon \fB_e \to \ol{\fB_d}$ and $F \colon \fB_d \to \ol{\fB_e}$, such that $G|_{\fW} =\alpha^*$ and $F|_{\fV} = (\alpha^{*})^{-1}$. The maximum principle (Lemma \ref{lem:maximum_principle}) and the injectivity of $\alpha$ and $\alpha^{-1}$ imply that $G(\fB_e) \subseteq \fB_d$ and $F(\fB_d)\subseteq \fB_e$. Since point evaluations are the only completely contractive representations of $A(\fV)$ and $A(\fW)$, we deduce --- as in the proof of Theorem \ref{thm:isomorphism} --- that $\alpha(f)=f \circ G$ for all $f\in A(\fV)$ and $\alpha^{-1}(g)=g\circ F$ for all $g\in A(\fV)$.
\end{proof}

Following the lines of the proof of Theorem \ref{thm:isomorphism_homo} --- using Lemma \ref{lem:isomorphism_continuous} instead of Theorem \ref{thm:isomorphism} ---  we obtain the counterpart of Theorem \ref{thm:isomorphism_homo} for the norm closed algebras $A(\fV)$ and $A(\fW)$.
\begin{thm}\label{thm:isomorphism_homo_cont}
Let $\fV \subseteq \fB_d$ and $\fW \subseteq \fB_{e}$ be homogeneous nc varieties, and let $\alpha : A(\fV) \to A(\fW)$ be a completely isometric isomorphism.
Then $\fV$ and $\fW$ are conformally equivalent, in the sense that one may assume that there is some $k$ such that $\fV, \fW \subseteq \fB_k$, and that under this assumption there exists an automorphism $G \in \operatorname{Aut}(\fB_k)$ such that $G(\fW)  = \fV$, and such that
\[
\alpha(f) = f \circ G \quad, \quad f \in A(\fV).
\]
\end{thm}

From the above theorem, together with Theorems \ref{thm:isomorphism_homo}, \ref{thm:isomorphism}, Lemma \ref{lem:isomorphism_continuous}, and Corollary \ref{cor:iso_homo_weak} we get the following:
\begin{cor}\label{cor:iso_homo_cont}
Let $\fV \subseteq \fB_d$ and $\fW \subseteq \fB_e$ be two homogeneous varieties. 
Then the following are equivalent:
\begin{enumerate}[(i)]
\item $A(\fV)$ and $A(\fW)$ are completely isometrically isomorphic;
\item $H^{\infty}(\fV)$ and $H^{\infty}(\fW)$ are completely isometrically isomorphic via a weak-$*$-continuous map (weak-$*$ is automatic when $d< \infty$);
\item  $\fV$ and $\fW$ are biholomorphic; 
\item $\fV$ and $\fW$ are conformally equivalent (perhaps after finding a new embedding $\fV, \fW \subseteq \fB_k$); 
\item there is a unitary transformation which maps $\fV$ onto $\fW$ (perhaps after finding a new embedding $\fV, \fW \subseteq \fB_k$).
\end{enumerate}
\end{cor}

In Theorem \ref{thm:iso_homo_cont_older}, we will see that in the finite dimensional case ($d<\infty$) the condition ``completely isometrically isomorphic" can be weakened ``isometrically isomorphic".

\subsection{An example (radial continuity versus uniform continuity)}
Let $\fC \fB_d$ denote the commutative nc unit ball, that is
\[
\fC \fB_d = \{X \in \fB_d : X_i X_j = X_j X_i \textrm{ for all } i,j=1,\ldots, d\}.
\]
$\fC \fB_d$ is a homogeneous variety in $\fB_d$.
We will give an example of a function in $H^\infty(\fC\fB_d)$ that extends to a function on $\ol{\fC\fB}_d$ which is radially continuous at each level $\fC\fB_d(n)$, but is not in $\fA_{\fC\fB_d}$ (and hence, not radially uniformly continuous).

First, we need a preliminary result. 
\begin{prop}
For every $T \in \fC \fB_d$, there exists a constant $C_T > 0$, such that
\[
\|p(T)\| \leq C_T \sup_{z\in {\bB_d}} |p(z)|
\]
for all $p \in \bC[z_1, \ldots, z_d]$.
\end{prop}
\begin{proof}
Fix $T \in \fC \fB_d$.
Then the joint spectrum of $T$, $\sigma(T)$, is contained in the closed unit ball $\ol{\bB}_d$.
Now the joint spectrum of a tuple of commuting $n \times n$ matrices $(T_1, \ldots, T_d)$ is nothing but the points appearing on the diagonals of the matrices when put in upper triangular form.
In other words, $\sigma(T)$ is simply the $n$ points in $\bC^d$ obtained as $(\langle T_1 v_i, v_i \rangle, \ldots, \langle T_d v_i, v_i\rangle)$, $i =1, \ldots, n$, where $\{v_1, \ldots, v_n\}$ is an orthonormal basis of $\bC^n$ with respect to which $T_1, \ldots, T_d$ are simultaneously upper triangular.

As $T$ is a row contraction, every point in the spectrum that is on the boundary of $\ol{\bB}_d$ corresponds to a direct summand.
Thus we have $T = N \oplus T'$, where $N$ is a normal tuple with $\sigma(N) \subseteq \partial \bB_d$, and $\sigma(T') \subseteq \bB_d$.
Thus we may assume $\sigma(T) \subseteq \bB_d$, because the spectral theorem for commuting normal tuples implies that $\|p(N)\| = \sup_{z \in \sigma(N)}|p(z)|$.

But if $\sigma(T) \subseteq \bB_d$, then $\sigma(T) \subseteq r\bB_d$ for some $r<1$.
By the continuity of the holomorphic functional calculus  for commuting operators (see, e.g., \cite{Hlm81} or \cite{Tay70}), there is a constant $C_T$ such that
\[
\|f(T)\| \leq C_T \sup_{z\in {\bB_d}} |f(z)|
\]
for every function $f$ analytic in $\bB_d$, as required.
\end{proof}
Let $A(\ol{\bB_d})$ denote the closure in the sup norm of the polynomials in $H^\infty(\bB_d)$.
$A(\ol{\bB_d})$ is called {\em the ball algebra}.
\begin{cor}
Every $T \in \fC\fB_d(n)$ gives rise to a functional calculus
\[
\Phi_T : A(\ol{\bB}_d) \to M_n,
\]
which we denote by $f \mapsto f(T)$.
The functional calculus is a bounded homomorphism, satisfying $\|f(T)\| \leq C_T \|f\|_\infty$,  for all $f \in A(\ol{\bB}_d)$.
\end{cor}

\begin{example}\label{ex:cont_mult}
When the variety is the set $\fC \fB_d$ consisting of all commuting row contractions, then it is common to use the notation $\cM_d = H^\infty(\fC\fB_d)$ and $\cA_d = \fA_{\fC\fB_d}$.
Note that $\cM_d$ is just the multiplier algebra of the Drury--Arveson space, and that $\cA_d$ is the norm closure of the polynomials in $\cM_d$ (in Section \ref{sec:connections_to_comm} we will elaborate further on these identifications).

Let $\psi \in \cM_d \cap C(\ol{\bB}_d)$ be such that $\psi \notin \cA_d$ (the existence of such a multiplier was noted in \cite[Section 5.2]{ShalitSurvey} by invoking \cite{FX11}).
Then $\psi$ is, in particular, in the ball algebra $A(\ol{\bB_d})$.
By the above corollary, we have that $\|\psi(rX) - \psi(sX)\| = \|(\psi_r - \psi_s)(X)\| \leq C_X \|\psi_r-\psi_s\|_\infty$ for every $X \in \fB_d$ and every $r,s \in \ol{\bD}$.
Thus $\psi$ gives rise to a function that is defined on $\ol{\fB_d(n)}$, continuous on $\fB_d$, and radially continuous on $\ol{\fB_d(n)}$, for every $n$.
However, as $\psi \notin \cA_d$, $\psi$ is not in the closure of the polynomials.

We conclude that a function may be continuous on every $\fB_d(n)$, with radial limits everywhere on $\fB_d(n)$, holomorphic and uniformly bounded on $\fB_d$, while not being in the closure of polynomials.
We do not know whether the function $\psi$ above is uniformly continuous on each level $\fB_d(n)$.
\end{example}

\begin{quest}
Let $n \in \bN$ be fixed.
Does there exist a constant $C_n$ (depending implicitly on $d$) such that
\[
\|p(T)\| \leq C_n \sup_{z\in {\bB_d}} |p(z)|
\]
for all $T \in \fC\fB_d(n)$ and all $p \in \bC[z_1, \ldots, z_d]$?
\end{quest}

Of course, by the well known incomparability of the multiplier and supremum norms in $\cM_d$ \cite[Section 3.7]{ShalitSurvey}, if such constants $C_n$ exist then they must satisfy $C_n \to \infty$ (when $d>1$).
If the answer to the previous question is affirmative, then so is the answer to the following question, by making use of the same function from the example above.

\begin{quest}
Does there exist a bounded holomorphic function, that is levelwise uniformly continuous on every $\fB_d(n)$, which is not uniformly continuous on $\fB_d$?
\end{quest}

\section{Connection to subproduct systems} \label{sec:subproduct}
Our results connect well to works on structure and classification of operator algebras associated with subproduct systems.

A {\em subproduct system} is a family $X = \{X(n)\}_{n \in \bN}$ of a Hilbert spaces, such that $X(0) = \bC$, $X(n) \subseteq X(1)^{\otimes n}$, and
\[
X(m+n) \subseteq X(m) \otimes X(n),
\]
for all $m,n \in \bN$.

Subproduct systems were introduced in \cite{ShalitSolel} as a technical tool for the analysis of semigroups of completely positive maps on von Neumann algebras (recently, they have been used to study semigroups on C*-algebras too \cite{ShalitSkeideBig}).
In fact, one also looks at subproduct systems over more general semigroups, and it is useful to allow fibers that are Hilbert W*-correspondences, and not just Hilbert spaces, but such generality is beyond the scope of the present work.
Subproduct systems give rise to a class of natural operator algebras, and in recent years these algebras have been investigated by several researchers \cite{And15,DRS11,DorMar14,DorMar15,Hartz12,KakSh15,Vis11,Vis12}.
We will now explain how algebras of bounded analytic functions on homogeneous varieties are operator algebras associated with subproduct systems, and indicate points of intersection with previous works.

Assume that $X$ is a subproduct system, and that $X(1) = \C^d$, with $d< \infty$ (the assumption $d< \infty$ is mainly for simplicity).
We identify the free algebra $\F_d$ with a dense subspace of the Fock space $\cF(\C^d)$.
Then we can define a homogeneous ideal $I^X$ by saying that a homogeneous polynomial $p$ of degree $n$ is in $I^X$ if and only if $p \in (\C^d)^{\otimes n} \ominus X(n)$.
It is straightforward to check that $I^X$ is really a homogeneous (two sided) ideal in $\F_d$.

Conversely, given a homogeneous (two sided) ideal $I \triangleleft \F_d$, we define a subproduct system $X_I$ by
letting $X_I(n)$ be the orthogonal complement in $(\C^d)^{\otimes n}$ of the homogeneous polynomials in $I$ that have degree $n$.

In \cite[Proposition 7.2]{ShalitSolel}  it was observed that the map $X \mapsto I^X$ is a bijective correspondence (with inverse $I \mapsto X_I$) between subproduct subsystems of $\{(\C^d)^{\otimes n}\}_{n \in \bN}$ and proper homogeneous ideals in $\F_d$.

Let $X$ be a subproduct system with $X(1) = \C^d$.
The {\em $X$-Fock space} is the direct sum
\[
\cF_X = X(0) \oplus X(1) \oplus X(2) \oplus \ldots
\]
Fix an orthonormal basis $\{e_1, \ldots, e_d\}$ for $\C^d$.
The {\em $X$-shift} is the $d$-tuple of operators $S^X = (S^X_1, \ldots, S^X_d)$, given by
\[
S^X_i \eta = e_i \otimes \eta \,\, , \,\, \eta \in X(n).
\]
The {\em tensor algebra} associated with $X$ is the unital, norm closed operator algebra $\cA_X$ generated by $S^X$.
The {\em noncommutative Hardy algebra} associated with $X$ is the unital, weak-operator closed operator algebra $\cL_X$ generated by $S^X$.

Let $I \triangleleft \F_d$ be a homogeneous ideal, let $\fV = V_{\fB_d}(I)$ and let $\cJ = \cJ_{\fV}$ be the weak-$*$ closed ideal consisting of multipliers in $H^\infty(\fB_d)$ that vanish on $\fV$.
By Theorem \ref{thm:homoclosedweakstar}, $\cJ = \ol{I}^{w*}$.
By Lemma \ref{lem:mu_iota}, $\cH_\fV = (\cJ \cH^2_d)^\perp$.

On the other hand, if we identify $\cF_X$ as a subspace of $\cF(\C^d) = \cH^2_d$, we see that
$\cF_X = I^\perp = (\cJ \cH^2_d)^\perp$.
Thus, using Corollary \ref{cor:levelwise_continuous}, we make the identifications
\begin{equation}\label{eq:algebras}
\cA_X = A(\fV) \quad , \quad \cL_X = H^\infty(\fV).
\end{equation}

In \cite{DRS11,DorMar14} the algebras $\cA_X$ and $\cL_X$ were classified in terms of the structure of the subproduct systems.

\begin{dfn}
Two subproduct systems $X = \{X(n)\}_{n \in \bN}$ and $Y = \{Y(n)\}_{n \in \bN}$ are said to be {\em isomorphic} if there exists a family of unitaries $\{U_n : X(n) \to Y(n)\}_{n \in \bN}$ such that
\[
U_{m+n} P_{X(m) \otimes X(n) \to X(m+n)} = P_{Y(m) \otimes Y(n) \to Y(m+n)} (U_m \otimes U_n) ,
\]
for all $m,n \in \bN$.
\end{dfn}

In \cite[Proposition 7.4]{ShalitSolel} (see also \cite[Proposition 3.1]{DRS11}) it was shown that if $X$ and $Y$ are subproduct subsystems of $(\C^d)^{\otimes n}_{n \in \bN}$, then $X$ and $Y$ are isomorphic, if and only if $I^X$ is obtained from $I^Y$ by unitary change of variables.
Using this, we can now prove the following geometric characterization of subproduct system isomorphism.

\begin{prop}\label{prop:subproduct_varieties}
Let $X$ and $Y$ be subproduct subsystems of $(\C^d)^{\otimes n}_{n \in \bN}$, for $d \in \N$.
$X$ and $Y$ are isomorphic if and only if $V_{\fB_d}(I^X)$ and $V_{\fB_d}(I^Y)$ are conformally equivalent, and this happens if and only if there is a unitary map $U$ of $\C^d$ such that
\[
U V_{\fB_d}(I^X) = V_{\fB_d}(I^Y) .
\]
\end{prop}
\begin{proof}
First let us assume that $X$ and $Y$ are isomorphic as subproduct systems. 
Then, the associated Fock spaces $\cF_X$ and $\cF_Y$ are in particular unitarily equivalent. This equivalence induces an isomorphism between the multiplier algebras. Applying Corollary \ref{cor:iso_homo_weak}, we obtain that there exists a unitary $U$ on $\C^d$ that satisfies $U V_{\fB_d}(I^X) = V_{\fB_d}(I^Y)$.

Conversely, assume that there exists a unitary $U$ on $\C^d$, such that $U V_{\fB_d}(I^X) = V_{\fB_d}(I^Y)$. We note that $U$ is in fact a coordinate change and extends to a map $\Gamma(U) \colon \F_d \to \F_d$, which acts on the $n$-th graded component (which we view as $(\C^d)^{\otimes n}$) as $U^{\otimes n}$. 
By the homogeneous Nullstellensatz (Theorem \ref{thm:null_poly}), this coordinate change maps $I^X$ onto $I^Y$. 
As $I^X$ is obtained from $I^Y$ by a unitary change of variables the subproduct systems are isomorphic.
\end{proof}

Using the above characterization of subproduct system isomorphism, we can now recognize that Corollaries \ref{cor:iso_homo_weak} and \ref{cor:iso_homo_cont} were obtained (for finite $d$) in \cite[Theorems  4.8 and 11.2]{DRS11}. 
In \cite{DRS11}, the general \cite[Theorem 9.7]{ShalitSolel} on isomorphisms of subproduct systems was invoked, to obtain a stronger statement, with ``completely isometric" replaced by ``isometric". 
Having the dictionary set up between homogeneous nc varieties and subproduct systems, we reformulate these results as follows. 

\begin{thm}\label{thm:iso_homo_cont_older}
Let $d< \infty$, and let $\fV, \fW \subseteq \fB_d$ be two homogeneous varieties.
Then $A(\fV)$ and $A(\fW)$ are isometrically isomorphic if and only if $H^{\infty}(\fV)$ is isometrically isomorphic to $H^{\infty}(\fW)$, and this happens if and only if $\fV$ and $\fW$ are conformally equivalent, which is the case if and only if there is a unitary map $U$ of $\C^d$ such that
\[
U V_{\fB_d}(I(\fV)) = V_{\fB_d}(I(\fW)) .
\] 
\end{thm}
\begin{proof}
This follows from \cite[Theorems  4.8 and 11.2]{DRS11} together with Proposition \ref{prop:subproduct_varieties}. 

Alternatively, if $A(\fV)$ and $A(\fW)$ are isometrically isomorphic, then by the above discussion the corresponding algebras $\cA_X$ and $\cA_Y$ (as in Equation \eqref{eq:algebras}) are isometrically isomorphic. 
By an application of the disc trick (as in Proposition \ref{prop:biholo=>0-biholo}), there exists a vacuum preserving isometric isomorphism $\cA_X \to \cA_Y$. 
By \cite[Theorem 9.7]{ShalitSolel} this means that $X \cong Y$, which, by Proposition \ref{prop:subproduct_varieties}, means that there exists a unitary as stated. 
The converse is already taken care of by Corollary \ref{cor:iso_homo_cont} together with Proposition \ref{prop:subproduct_varieties}. 

The case of $H^\infty(\fV) \cong H^\infty(\fW)$ is handled in a similar manner. 
\end{proof}
Note that in Theorems \ref{thm:isomorphism_homo} and \ref{thm:isomorphism_homo_cont} we obtain additional information regarding the form of non-zero-preserving isomorphisms, and moreover we also handle the case of $d = \infty$. 

In fact, the connection between subproduct systems and nc varieties hold also in the setting of $d = \infty$, but when discussing the connection between ideals and varieties there might occur ideals which are not ideals of polynomials in the classical sense (for example, if $a = (a_i)_{i=1}^\infty \in \ell^2$, then the function $z \mapsto \langle z, a \rangle = \sum a_i z_i$ is not a polynomial in the classical sense, but such functions may naturally be thought of polynomials of degree one). 
In this setting one has a bijective correspondence between homogeneous norm closed ideals $\cI$ in $A(\fB_\infty) = \fA_\infty$ and subproduct systems $X$ with $X(1)$ a separable Hilbert space. 
Let us write the bijection $\cI \leftrightarrow X_\cI$ and $X \leftrightarrow \cI^X$. 

In this setting it still holds that two subproduct systems $X$ and $Y$ are isomorphic if and only if $\cI^X$ and $\cI^Y$ are related by a unitary change of variables. 
However, we do not know whether or not homogeneous normed closed ideals in $A(\fB_\infty)$ are in bijective correspondence with homogeneous nc varieties in $\fB_\infty$. 
The issue is that we do not know whether the homogeneous Nullstellensatz (Theorem \ref{thm:cont_nullss}) holds in the case $d = \infty$. 

With these comments in mind, we use the results of this paper to contribute to the completely isometric isomorphism problem for tensor algebras of subproduct systems in the case of $d = \infty$, something that was left open in \cite{DRS11}. 
The following result still leaves much to be desired. 

\begin{prop}\label{prop:iso_homo_dinfty}
Let $X$ and $Y$ be two subproduct systems whose fibers $X(n)$ and $Y(n)$ are separable Hilbert spaces for all $n$. 
Suppose that there exist homogeneous nc varieties such that $\cI^X = \cI_\fV$ and $\cI^Y = \cI_\fW$. 
Then $\cA_X$ and $\cA_Y$ are completely isometrically isomorphic if and only if $X$ and $Y$ are isomorphic. 
\end{prop}
\begin{proof}
By assumption, there exist homogeneous varieties $\fV$ and $\fW$ such that $\cA_X = A(\fV)$ and $\cA_Y = A(\fW)$. 
By Corollary \ref{cor:iso_homo_cont}, after perhaps finding a new embedding $\fV, \fW \subseteq \fB_k$, there is a unitary $U$ such that $\fW = U \fV$. 
As above, this implies that $X$ and $Y$ are isomorphic. 
\end{proof}

We can now also recognize that \cite[Theorem 9.2]{KakSh15} treated the classification of the operator algebras of the form $A(\fV)$, where $\fV$ is the zero set of an ideal generated by monomials. 
In \cite[Theorem 9.2]{KakSh15}, due to the particularity of the ideals under investigation, additional rigidity was present: completely isometric isomorphism was actually shown to be equivalent to algebraic isomorphism.

Finally, we mention that in \cite[Theorem 3.4]{KakSh15} (following work done in \cite{DorMar14}) it was shown that operator algebras $\cA_X$ and $\cA_Y$ (arising from subproduct systems $X$ and $Y$) are boundedly isomorphic if and only if $X$ and $Y$ are {\em similar}. 
We leave it for future work to parse what this means in terms of bounded isomorphisms between algebras of the form $A(\fV)$, where $\fV \subset \fB_d$ is a homogeneous variety.

\section{Connection to the commutative case} \label{sec:connections_to_comm}

In this section, we show how our study connects to previous works on algebras of bounded analytic functions on commutative analytic varieties.
As we shall see, the nc setting not only generalizes some of the results, it also clarifies some of the results (as well as some non-results) that were obtained for commutative algebras.

\begin{remark}
We will be using somewhat confusing terminology, as we will be considering ``commutative noncommutative varieties".
The word ``noncommutative" here means that we will be considering subvarieties of the nc ball $\fB_d$, that is, varieties consisting of $d$-tuples of matrices of arbitrary size.
The word ``commutative" here refers to the fact that the varieties under consideration will all lie in the commuting variety $\fC\fB_d \subset \fB_d$, that is, the tuples of matrices are assumed to commute with one another.
Perhaps an alternative way of saying ``commutative nc variety" would be ``free commutative variety".
In any case, now that the reader is warned, there should be no confusion.
\end{remark}

\subsection{The isomorphism problem in the commutative case}

We start by recalling that the Drury--Arveson space $H^2_d$ is the reproducing kernel Hilbert space (in the usual, commutative function-theoretic sense) on the unit ball $\bB_d \subseteq \C^d$, with reproducing kernel $k_w(z) = k(z,w) = \frac{1}{1-\langle z,w \rangle}$ (see \cite{ShalitSurvey}).
Let $\cM_d$ denote the multiplier algebra (in the usual, commutative function-theoretic sense) of $H^2_d$.
Note that if we put $\Omega = \bB_d = \fB_d(1)$, and if we denote $\fC \fB_d$ the part of $\fB_d$ consisting of all commuting tuples, then using Lemma \ref{lem:HOmega} we see that
\[
H^2_d \cong \cH_\Omega = \cH_{V(\cJ_\Omega)} = \cH_{\fC\fB_d};
\]
this is because $\fC \fB_d$ is the smallest variety in $\fB_d$ that contains $\bB_d$.
Thus $\cM_d$ can be identified with $\mlt \cH_{\fC\fB_d} = H^\infty(\fC \fB_d)$ (further explanation will be given in Proposition \ref{prop:commutingcase} below).

We call a subset $V \subseteq \bB_d$ a {\em variety} if $V$ is the joint zero set of a family of functions in $\cM_d$.

In \cite{DHS14,DRS11,DRS15,Hartz12,Hartz16,KerMcSh13,RamseyThesis} (see also the survey paper \cite{SalomonShalit}), the following problem was investigated.
For a variety $V \subseteq \bB_d$, consider the Hilbert space $\cF_V = \spn\{k_v : v \in V\}$, and define $\cM_V = \mlt \cF_V$.
By \cite[Proposition 2.6]{DRS15},
\[
\cM_V = \{f\big|_V : f \in \cM_d\} .
\]

\begin{prop}\label{prop:commutingcase}
Let $V \subseteq \bB_d$.
Let $\fV = V(\cI_V) = V(\cJ_V)$ be the smallest nc variety that contains $V$.
Then $\cM_V$ is completely isometrically isomorphic and unitarily equivalent to $H^\infty(\fV) = \mlt \cH^2_\fV$.
\end{prop}
\begin{proof}
First, $V(\cI_V) = V(\cJ_V)$ thanks to Lemma \ref{lem:HvsMltVarieities} (here we are using the notation of that lemma).
As above, we apply Lemma \ref{lem:HOmega} to obtain
\[
\cH_V = \cH_{\fV} ,
\]
as subspaces of $\cH^2_d$, and $\cH_V$ is clearly unitarily equivalent to $\cF_V$, via the identity map.

Let us concentrate first on the case $V = \bB_d$.
In this case, $\fV = \fC \fB_d$.
To see this, observe that $\cJ_V$ is simply the weak-operator closed ideal generated by the nc functions $z_i z_j - z_j z_i$ ($i,j \in \{1, \ldots, d\}$), thus $V(\cJ_V) = \fC\fB_d$.
Now, $\cM_d$ and $H^\infty(\fC\fB_d)$ have different interpretations as function algebras, but both $\cM_d$ and $H^\infty(\fC\fB_d)$ are the operator algebra obtained by compressions of $\cL_d = H^\infty(\fB_d)$ to the co-invariant subspace $H^2_d = \cH_{\fC\fB_d}$.
Thus these operator algebras coincide.

Now let $V \subseteq \bB_d$ be a variety, and let $\fV = V(\cJ_V)$ be the smallest nc variety in $\fB_d$ that contains it.
The algebra $\cM_V$ is obtained from $\cM_d$ by compressing to $\cF_V$ \cite[Proposition 2.6]{DRS15}, and by Theorem 7.2, $H^\infty(\fV)$ is obtained from $H^\infty(\fB_d)$ by compressing to $\cH_\fV$, and since $\cH_\fV \subseteq \cH_{\fC\fB_d}$, $H^\infty(\fV)$ is the compression of $H^\infty(\fC\fB_d)$ to that subspace.
Thus, $\cM_V$ and $H^\infty(\fV)$ coincide.
\end{proof}

\begin{remark}\label{rem:mult_sup_norm}
It is well known that $\cM_d \subsetneq H^\infty(\bB_d)$, and that the multiplier norm and supremum norm $\|f\|_\infty = \sup_{z \in \bB_d}|f(z)|$ are not comparable.
The noncommutative framework allows to view the multiplier norm as a supremum norm: for every multiplier $f \in \cM_d$
\[
\|f\|_{\mlt} = \sup_{Z\in \fC\fB_d}\|f(Z)\| .
\]
Another thing that the noncommutative framework helps to clarify, is the issue of continuous multipliers.
As pointed out in Example \ref{ex:cont_mult}, the algebra $\cA_d$, obtained as the norm closure of polynomials in $\cM_d$, is strictly smaller than the algebra $\cM_d \cap C(\ol{\bB}_d)$ consisting of  multipliers that extend to (uniformly) continuous functions on $\ol{\bB}_d$.
When looked at from the noncommutative point of view, we see that $\cA_d = A(\fC\fB_d)$ --- the algebra of multipliers that extend to uniformly continuous functions on $\ol{\fB}_d$.
Thus, the urge to call $\cA_d$ the algebra of ``continuous multipliers" need not be suppressed.
\end{remark}

In \cite{DRS15}, the point of departure was a radical homogeneous ideal $I \triangleleft \C[z_1, \ldots, z_d]$.
Let $V$ be the affine variety corresponding to $I$ (i.e., the zero locus of $I$).
Let $\cA_V$ denote the norm closure of polynomials in $\cM_V$.
This algebra is also, in some sense, the universal unital operator algebra generated be a commuting row contraction that satisfies the relations in the ideal $I$.
In \cite{DRS15}, $\cA_V$ was denoted by $\cA_I$, to highlight the role of the ideal $I$.
In fact, one naturally defines the universal operator algebras $\cA_I$ and $\cM_I$ and  for a not-necessarily radical homogeneous ideal $I \triangleleft \C[z_1, \ldots, z_d]$ --- $\cA_I$ is simply the compression of $\cA_d$ to the complement of $I H^2_d$ in $H^2_d$, and likewise for $\cM_I$.

Let $I$ and $J$ be radical homogeneous ideals corresponding to affine varieties $V$ and $W$. 
In \cite[Theorem 8.2]{DRS11} it was shown that $\cA_I = \cA_V$ is completely isometrically isomorphic to $\cA_J = \cA_W$ if and only if $V$ and $W$ are related by a unitary transformation, and in \cite[Theorem 8.5]{DRS11} it was shown that $\cA_I = \cA_V$ is  algebraically isomorphic to $\cA_J = \cA_W$ if and only if $V$ and $W$ are related by a linear map or, equivalently, if $V$ and $W$ are biholomorphic (the proof of that theorem was completed only later, with an important contribution by Hartz \cite{Hartz12}).
Likewise, in \cite[Theorem 11.7]{DRS11} it was shown that the algebras $\cM_I = \cM_V$ and $\cM_J = \cM_W$ are isomorphic/completely isometrically isomorphic under the exact same terms.
Thus, the variety $V = V(I)$ serves as a geometric invariant of the structure of the operator algebras $\cA_I$ and $\cM_I$, when $I$ is a homogeneous and radical ideal.
The question of whether there exists a geometric invariant for classifying the algebras $\cA_I$ and $\cM_I$ for a not-necessarily-radical ideal was left open.

In the noncommutative setting, the geometric invariant becomes evident.
Indeed, it is easy to see that $\cA_I = A(V_{\fB_d}(I))$ and that $\cM_I = H^\infty(V_{\fB_d}(I))$, thus Corollaries \ref{cor:isomorphism_homo} and \ref{cor:iso_homo_cont} give the ``geometric" classification result.
Note that Hilbert's Nullstellensatz explains why we should expect that the affine varieties give a classification for (algebras associated with) ideals, only for the class of radical ideals.
On the other hand, the nc homogeneous Nullstellensatz, Theorem \ref{thm:null_poly}, shows that homogeneous nc varieties are in bijective correspondence with homogeneous ideals (see also Corollary \ref{cor:free_com_NSTZ} below).

Finally, let us point out how the nc theoretic Corollaries \ref{cor:isomorphism_homo} and \ref{cor:iso_homo_cont} contain the function-theoretic Theorems 8.2 and 11.7 in \cite{DRS11} (and this should also shed light on how Corollary \ref{cor:isomorphism} relates to \cite[Theorem 4.4]{DRS15}).
To wit, if $I$ is a radical  ideal, and $V = V(I)$ is the associated affine variety, then $\fV = V_{\fB_d}(I)$ is the smallest nc variety containing $V$.
Thus, if $J$ is another radical ideal and $W$ and $\fW$ the associated affine and nc varieties, respectively, then $V$ and $W$ are related by a unitary/automorphism if and only if $\fV$ and $\fW$ are.
By Proposition \ref{prop:commutingcase}, we conclude that $\cM_V$ is completely isometrically isomorphic to $\cM_W$, if and only if $H^\infty(\fV)$ and $H^\infty(\fW)$ are completely isometrically isomorphic, and this happens (by Corollary \ref{cor:isomorphism_homo}) if and only if $\fV$ and $\fW$ are conformally equivalent (equivalently, if and only if a unitary maps one onto the other), which, by the previous remarks, happens if and only if $V$ and $W$ are conformally equivalent (equivalently, if and only if a unitary maps one onto the other).
Thus, we recapture some of the classification results of \cite{DRS11}.

When an ideal $I \triangleleft \C[z_1, \ldots, z_d]$ is not radical, then $V_{\fB_d(I)}$ is not uniquely determined by the scalar level, and to encode $I$ one is required to use higher matrix levels.

\subsection{An example}
In \cite{McSh16}, a reproducing kernel space consisting of Dirichlet series on the half-plane, which is {\em weakly isomorphic} to the Drury--Arveson space $H^2_d$, was discovered.
Let $\bH_0 = \{z \in \bC : \operatorname{Re} z > 0\}$.
Fix $d \in \bN \cup \{\infty\}$, and let $b = (b_n)_{n=1}^d$ be a sequence of positive numbers such that $\sum b_i^2 = 1$.
Consider the map $f : \bH_0 \to \bB_d$ by
\[
f(s) = (b_1 p_1^{-s}, b_2 p_2^{-s}, \ldots),
\]
(where $p_i$ denotes the $i$th prime number), and define a kernel
\[
k(s,u) = \left(1 - \langle f(s), f(u) \rangle \right)^{-1} = \sum_{n \geq 1} a_n n^{-s - \bar u},
\]
on $\bH_0 \times \bH_0$ (the $a_n$s are positive numbers determined uniquely by this equality).
This kernel gives rise to a reproducing kernel Hilbert space $\cH$ on the set $\bH_0$, which has the complete Pick property.
The elements of $\cH$ are precisely the Dirichlet series $h(s) = \sum_n \gamma_n n^{-s}$, that satisfy $\|h\|^2_\cH := \sum |\gamma_n|^2 a_n^{-1} < \infty$.

One of the main results in \cite{McSh16} is that $\cH$ is {\em weakly isomorphic} to $H^2_d$, via the unitary map $U : k(\cdot,u) \mapsto \frac{1}{1-\langle \cdot, f(u) \rangle}$, which has inverse $U^* : g \mapsto g \circ f$.
Consequently, $\mlt(\cH)$ is unitarily equivalent $\mlt(H^2_d)$, and the inverse associates $\psi \in \mlt(H^2_d)$ with $\psi \circ f \in \mlt(\cH)$.
This is somewhat surprising (especially in the case $d= \infty$), as $\mlt(\cH)$ is an algebra of analytic functions in a single variable, whereas $\mlt(H^2_d)$ has several universal properties.
The norm of a multiplier $\varphi \in \mlt(\cH)$ is given by the highly inexplicit formula $\|\varphi\| = \sup_{\|h\|_\cH=1}\|\varphi h\|$, and it is not comparable to the perhaps-more-accessible supremum norm $\|\varphi\|_{\bH_0,\infty} = \sup_{s \in \bH_0}\|\varphi(s)\|$.
In light of Remark \ref{rem:mult_sup_norm} (as well as some wishful thinking), one might hope that there exists some kind of ``noncommutative half-plane" $\mathfrak{H}_0$, which will enable to find the multiplier norm of a matrix valued multiplier $\varphi$ by an nc supremum $\|\varphi(S)\|_{\mathfrak{H}_0,\infty} = \sup_{S \in \mathfrak{H}_0}\|\varphi(S)\|$.
We will now show that there is no such noncommutative half-plane.

Suppose that $\fH_0 \subseteq \M_1 = \sqcup_n M_n$ is an nc set for which it holds that
\begin{equation}\label{eq:nc_HP}
\|\varphi\| = \sup_{S \in \mathfrak{H}_0}\|\varphi(S)\|,
\end{equation}
for every matrix valued multiplier (we are assuming that $\mathfrak{H}_0$ is an nc set such that every element in $\mlt(\cH)$ can be evaluated at any $S \in \mathfrak{H}_0$).
For convenience, let us assume that $2 \leq d< \infty$.
Now, the unitary equivalence maps the function $s \mapsto b_i p_i^{-s}$ to the function $z \mapsto z_i$, therefore the row multiplier $f(s) = (b_1 p_1^{-s}, b_2 p_2^{-s}, \ldots, b_d p_d^{-s})$ is unitary equivalent to the row multiplier $(z_1, \ldots, z_d)$, and hence is a row contraction.
Therefore, if \eqref{eq:nc_HP} holds for all matrix valued multipliers, then for every $S \in \mathfrak{H}_0$, $\|f(S)\| \leq \|f\| = 1$.
In other words, for every $S \in \mathfrak{H}_0$,
\[
\sum_i b_i^2 e^{-\log p_i (S+S^*) } = \sum_i b_i^2 e^{-\log p_i S }(e^{-\log p_i S })^* \leq I.
\]
It follows that for every eigenvalue $\lambda$ of $S+S^*$, $f(\lambda) \in \ol{\bB}_d$.
This means that $\operatorname{Re}S \geq 0$.
But if $S$ has non-negative real part, then $\ol{\bH}_0$ is a complete spectral set for $S$, meaning that for any matrix valued Dirichlet polynomial $\varphi$, it holds that $\|\varphi(S)\| \leq \sup_{s \in \bH_0}|\varphi(s)|$.
Since, in general, $\sup_{s \in \bH_0}|\varphi(s)| < \|\varphi\|$, we see that \eqref{eq:nc_HP} cannot hold.

We conclude the examination of this example, by finding the natural nc variety in $\fB_d$ on which $\cH$ and $\mlt(\cH)$ can be thought to live.
First, the map $U: k(\cdot, u) \mapsto \frac{1}{1 - \langle \cdot, f(u) \rangle}$ identifies $\cH$ with the subspace
\[
\cH_{f(\bH_0)} = \spn\left\{\frac{1}{1 - \langle \cdot, f(u) \rangle} : u \in \bH_0 \right\} \subseteq H^2_d \subseteq \cH^2_d.
\]
By Lemma \ref{lem:HOmega}, $\cH_{f(\bH_0)} = \cH_{V(\cJ_{f(\bH_0)})}$, where $V(\cJ_{f(\bH_0)})$ is the smallest nc variety (cut out by $H^\infty(\fB_d)$ functions) that contains $f(\bH_0)$.
Now, $V(\cJ_{f(\bH_0)})$ is an nc set, so it clearly contains the nc set $(f(\bH_0))_{nc}$ that consists of all $d$-tuples of diagonal matrices formed by taking direct sums of the $d$-tuples $f(s) = (b_1 p_1^{-s}, \ldots, b_d p_d^{-s})$, where $s \in \bH_0$.
The commutators $x_i x_j - x_j x_i$ vanish on $(f(\bH_0))_{nc}$, and therefore, so does any function in the weakly closed ideal $\cJ_c$ generated by the commutators.
On the other hand, the quotient of $H^\infty(\fB_d)$ by $\cJ_c$ is $\cM_d \cong H^\infty(\fC \fB_d)$.
By \cite[Lemma 34]{McSh16}, there is no non-zero function $g \in \cM_d$ that vanishes on $f(\bH_0)$.
It follows that $\cJ_c = \cJ_{f(\bH_0)}$, and so $V(\cJ_{f(\bH_0)}) = \fC \fB_d$.

It is interesting to note that the supremum of a multiplier $\psi$ on $(f(\bH_0))_{nc}$ is given by the scalar sup norm of $\psi \circ f$ on $\bH_0$, and this is strictly smaller than $\|\psi\|$.
This does not contradict Theorems \ref{thm:quotient_mult} and \ref{thm:mult_are_bounded_on_V}, as $(f(\bH_0))_{nc}$ is not an nc variety in our sense.

\subsection{Commutative free Nullstellensatz}
In connection to the previous discussion on how the higher matrix levels encode the difference between an ideal and its radical, we investigate the matter from a purely algebraic point of view.

Let us denote $\fC\bM_d = \{ X \in \bM_d : X_i X_j = X_j X_i\, , \, i,j=1, \ldots, d\}$. 
In \cite{EisHoch79}, Eisenbud and Hochester obtained a generalization of the Nullstellensatz to the setting of rings with nilpotents.
More precisely, if $A$ is an affine ring and $I \triangleleft A$ is an ideal, then there exists a positive integer $k$, that depends on the nilpotence of $A/I$, such that:
\[
I = \bigcap_{\stackrel{I \subset \fm}{\fm \text{ maximal}}} \left( \fm^k + I \right).
\]
We will now obtain a version of their Nullstellensatz for ideals in $\C[z_1, \ldots, z_d]$ where the ``points" are allowed to be any tuple of commuting matrices. This will provide a more elementary proof of a slightly weaker result than the main result of \cite{EisHoch79}, while emphasizing the role played by finite dimensional representations.
In other words, we will show that zero locus of an ideal $I \triangleleft \C[z_1, \ldots, z_d]$ completely determines the ideal $I$ (see Corollary \ref{cor:free_com_NSTZ}).

\begin{prop} \label{prop:matrices_know_radical}
Let $k$ be a field and let $A$ be a Noetherian commutative local $k$-algebra with maximal ideal $\fm$, such that $A/\fm \cong k$ .
Let $f \in A$. 
If $\varphi(f) = 0$ for every homomorphism $\varphi \colon A \to M_n(k)$, then $f = 0$.
\end{prop}
\begin{proof}
Since $A/\fm \cong k$, we conclude that $f$ maps to $0$ and thus $f \in \fm$. 
Since $A$ is Noetherian, $\fm$ is finitely generated and thus $A/\fm^{\ell}$ is a finite-dimensional vector space over $k$, for every $\ell \geq 1$. 
The natural map $\pi \colon A \to A/\fm^{\ell}$ endows this finite dimensional space with a structure of an $A$-module, and thus $f$ acts as $0$ on this space. 
We conclude that $f \in \fm^{\ell}$.
Now we apply Krull's intersection theorem \cite[Corollary 5.4]{Eisenbud} to deduce that $f = 0$.
\end{proof}

\begin{remark}
For example, the assumption of the above proposition holds if $k$ is algebraically closed and $A$ is a localization of a finite type algebra over $k$ at a maximal ideal or, alternatively, if $k = \C$ and $A = \C\{\{z\}\}$ the ring of germs of analytic functions at $0$.
\end{remark}

\begin{cor} \label{cor:matrices_know_radical_poly}
Let $k$ be an algebraically closed field. 
Let $J \triangleleft k[x_1,\ldots,x_d]$ be an ideal. 
Put $A = k[x_1,\ldots,x_d]/J$, and let $\pi$ be the natural projection onto $A$.
If $f \in k[x_1,\ldots,x_d]$ is such that for every homomorphism $\varphi \colon A \to M_n(k)$ we have that $\varphi(\pi(f)) = 0$, then $f \in J$.
\end{cor}
\begin{proof}
Let us write $\overline{f} = \pi(f)$. 
For every maximal ideal $\fm \triangleleft A$, let $A_{\fm}$ be the localization of $A$ at $\fm$, and let $\iota_{\fm} \colon A \to A_{\fm}$ be the localization map. 
Since every finite dimensional representation of $A_{\fm}$ induces via $\iota_{\fm}$ a finite dimensional representation of $A$, we can conclude by the above proposition that $\iota_{\fm}(\overline{f}) = 0$. 
Since $\overline{f} = 0$ in every localization, it follows that $\overline{f} = 0$ (indeed, otherwise there exists a maximal ideal $\fm$ that contains the annihilator of $\overline{f}$ in $A$, and thus $\iota_{\fm}(\overline{f})$ is non-zero).
\end{proof}

Given $\Omega \subset \fC\bM_d$ and $S \subseteq \bC[z]$, if we denote
\[
I_{\bC[z]}(\Omega) = \{p \in \bC[z] : p(X) = 0 \,\, \textrm{ for all } \,\,X \in \Omega\}
\]
and
\[
V_{\fC\bM_d}(S) = \{X \in \Omega : p(X) = 0 \,\, \textrm{ for all } \,\, p \in S\},
\]
then we can reformulate the above corollary as
\begin{cor}[Commutative free Nullstellensatz]\label{cor:free_com_NSTZ}
For every ideal $J \triangleleft \bC[z]$,
\[
I_{\bC[z]}(V_{\fC\bM_d}(J)) = J.
\]
\end{cor}

\begin{rem}
From the main result of \cite{EisHoch79} it follows that it is enough to consider only finite dimensional representations of $\C[z_1,\ldots,z_d]/J$ of a fixed dimension. 
To see this note that as above we can represent the polynomial ring $A = \C[z_1,\ldots,z_d]$ (and in fact $A/J$) on $\left(\C[z_1,\ldots,z_d]/J\right)_{\fm}/\fm_{\fm}^k$ by multiplication operators, where $\fm$ is a maximal ideal containing $J$ and $k$ is the positive integer obtained using the main theorem of \cite{EisHoch79}. This representation is of course finite dimensional, in fact the dimension of this representation is bounded from above by the number of monomials of degree less than $k$, that we shall denote by $N$. Now if we assume that that $f \in A$ vanishes on all $N$ dimensional representations of $A/J$, then it implies that $f \in \fm^k + J$, for every maximal ideal $\fm$ that contains $J$ and we conclude that $f \in \bigcap_{J \subset \fm} (\fm^k + J)= J$.

\end{rem}


\begin{example} \label{eq:no_fd_nullss}
Take $s \in H^{\infty}(\D)$ a singular inner function (for example $s(z) = e^{\frac{z + 1}{z-1}}$). Then for any $X \in \fB_1$ we have that $s(X) \neq 0$, since we can always conjugate $X$ to an upper triangular form with a unitary and since $s$ does not vanish on the disc it won't vanish on the entries of the diagonal.
Now, note that the ideal generated by $s$ is $\textsc{wot}$-closed and its range is the shift invariant subspace $sH^2(\D)$ (here we used the Beurling-Lax theorem \cite[Theorem V.3.3]{SzNFo10}).
This is a proper subspace of $H^2(\D)$ since $1 \notin s H^2(\D)$.
By \cite{DavPitts2} the $\textsc{wot}$-closed ideal $s H^{\infty}(\D)$ is not trivial and thus the function $1$, which vanishes on precisely the same matrices in the ball as $s$ does, is not in the ideal.
We conclude that one cannot get a version of the Nullstellensatz for $\textsc{wot}$-closed ideals in $H^\infty(\D) = H^\infty(\fB_1)$ considering only finite dimensional representations.
\end{example}

\begin{remark}
It is trivial that if we throw infinite dimensional representations into the mix, then we get a Nullstellensatz (one only needs to consider the representation obtained by compressing the shifts to the orthogonal complement of the range of the ideal).
\end{remark}

\bibliographystyle{abbrv}
\bibliography{nc_bibliography}

\end{document}